\documentclass[11pt]{amsart}
\usepackage[applemac]{inputenc}
\usepackage{
 ulem,
 amsfonts}
 \usepackage{amsmath}
\usepackage{amsmath,esint, amssymb, amsthm,mathbbol,upgreek,exscale,bigints}

 \usepackage{mathtools}
 
\usepackage{fullpage}
\usepackage{
 ulem,
 amsfonts}
 \usepackage{amsmath}
\usepackage{amsmath,esint, amssymb, amsthm,mathbbol,upgreek,exscale}

 \usepackage{mathtools}

\usepackage[colorlinks=true, pdfstartview=FitV, linkcolor=blue, citecolor=blue,
 urlcolor=blue]{hyperref}
\usepackage{color}

\date{}

\newcommand\N{\mathbb{N}} 
\newcommand\R{\mathbb{R}} 
\newcommand\E{\varepsilon} 

\newcommand\EE{\varepsilon}

\theoremstyle{plain}

\numberwithin{equation}{section}
\newtheorem{theorem}{Theorem}[section]
\newtheorem{proposition}[theorem]{Proposition}
\newtheorem{lemma}[theorem]{Lemma}
\newtheorem{remark}[theorem]{Remark}
\newtheorem*{remarks}{Remarks}

\usepackage[textmath,displaymath,floats,graphics,delayed]{preview}
 \title[]{Dynamics of one fold symmetric patches for the aggregation equation and collapse to singular measure}

\author[T. Hmidi]{Taoufik Hmidi}
\address{Univ Rennes, CNRS, IRMAR - UMR 6625, F-35000 Rennes, France}
\email{thmidi@univ-rennes1.fr}

\author[D. Li]{Dong Li}
\address{ 
Department of Mathematics\\
The Hong Kong University of Science and Technology\\
 Clear Water Bay, Kowloon\\
 Hong kong}
\email{madli@ust.hk}

\begin{document}

\begin{abstract}
We are concerned with the dynamics of one fold symmetric patches  for the two-dimensional  aggregation equation associated to the Newtonian potential. We reformulate a suitable   graph model  and prove a local well-posedness  result  in sub-critical and critical spaces. The global existence  is obtained only for small initial data using a weak damping property hidden in the velocity terms. This allows to analyze the concentration phenomenon of the aggregation patches near the blow up time. In particular, we prove that the patch collapses to a collection of disjoint segments and we provide a description of the singular measure through a careful study of the asymptotic behavior of the graph. 
 \end{abstract}

\newpage 

\maketitle{}
\tableofcontents
%
\section{Introduction}
This paper is devoted to the study of the two-dimensional aggregation equation with the Newtonian potential:

\begin{equation}\label{aggreg1}
\left\lbrace
\begin{array}{l}
\partial_t \rho+\textnormal{div}(v\, \rho)=0,\, t\geq0, x\in\R^2, \\
v(t,x)=-\frac{1}{2\pi}\displaystyle{\int_{\R^2}}\frac{x-y}{|x-y|^2}\rho(t,y) dy,\\
\rho(0,x)=\rho_0(x).
\end{array}
\right.
\end{equation}

This model with more general  potential interactions,  with or without dissipation, is used to explain some behavior  in physics and population dynamics. As a matter of fact  it appears in  vortex densities in superconductors \cite{AS,DZ,Kel-Seg}, material sciences \cite{HP,NPS01}, cooperative controls and  biological swarming \cite{BT,Bred,BCMo,GP,ME,MCO,TB}, etc... During the last few decades, a lot of intensive research activity has been devoted to explore several mathematical and numerical aspects of this equation. 
It is known according to \cite{BLL,NPS01} that classical solutions can be constructed for short time. They develop finite time singularity if and only if the initial data is strictly positive at  some points and the blow up time is explicitly given by $T_\star=\frac{1}{\max\rho_0}$. This follows from the equivalent form
$$
\partial_\rho+v\cdot\nabla\rho=\rho^2
$$
which, written with Lagrangian coordinates, gives   exactly a Riccati equation. Note that similarly to Yudovich result for  Euler equations \cite{Yu}, weak  unique solutions in $L^1\cap L^\infty$ can be constructed following the same strategy, for more details  see \cite{BLL,BL,BB10,BCL,BLR,FHK,FH,HJD,L,LR2}. Since $L^1$ norm is conserved at least at the formal level, then lot of efforts were done in order to extend the classical solutions beyond the first blowup time. In \cite{Poup}, Poupaud established the existence of global  generalized  solutions with defect measure when the initial data is a non negative  bounded  Radon  measure. He also showed that when the second moment of the initial data is bounded then for such solutions atomic part appears in finite time. This result is at some extent  in contrast with what is established for Euler equations. Indeed, according to Delort's result \cite{Delort} global weak solutions without defect measure can be  established when the initial vorticity  is a non negative  bounded  Radon  measure and the associated velocity has finite local energy. During the time those solutions do not develop atomic part contrary to the aggregation equation. This  illustrates somehow the gap  between both equations not only at the level of classical solutions but also for the weak solutions. The literature dealing with   measure valued solutions for the aggregation equation  with different potentials is very abundant and we we refer the reader to the papers \cite{BV,CDFLS,CCoV,CRosa, Masmoudi} and the references therein. 

Now we shall discuss another   subject  concerning  the aggregation patches. Assume that the initial data takes the patch form
$$
\rho_0={\bf{1}}_{D_0}
$$
with $D_0$ a bounded domain, then solutions can be  uniquely constructed up to the time $T^\star=1$ and  one can check that
$$
\rho(t)=\frac{1}{1-t}{\bf{1}}_{D_t}\quad \hbox{with}\quad (\partial_t+v\cdot\nabla){{\bf{1}}_{D_t}}=0.
$$
Note that $v$ is computed from $\rho$ through Biot-Savart law. To filter the time factor in the velocity field and find analogous equation to Euler equations it is more convenient  to rescale the time  as it was done in \cite{BLL}. Indeed, set 
$$\tau=-\ln(1-t), \quad u(\tau,x)= -\frac{1}{2\pi}\displaystyle{\int_{\R^2}}\frac{x-y}{|x-y|^2}{\bf{1}}_{\tilde{D}_\tau}(y) dy,\quad  \tilde{D}_\tau=D_t
$$
then we get
$$
(\partial_\tau+u\cdot\nabla){\bf{1}}_{\tilde{D}_\tau}=0, \tilde{D}_0=D_0.
$$
We observe  that with this formulation,  the blow up occurs at infinity and so the solutions do exist globally in time.
To alleviate the notations we  shall write this latter equation  with  the initial variables. Hence the vortex patch problem reduces to understand the evolution  equation
\begin{equation}\label{sqg}
\left\lbrace
\begin{array}{l}
\partial_t \rho+v \cdot\nabla\rho=0,\, t\geq0,\\
v(t,x)=-\frac{1}{2\pi}\displaystyle{\int_{D_t}}\frac{x-y}{|x-y|^2} dy,\\
\rho(0)={\bf{1}}_{D_0}.
\end{array}
\right.
\end{equation}
Let us point out that the area of the domain $D_t$ shrinks to zero exponentially, that is,
\begin{equation}\label{Exp1}
 \forall\, t\geq0,\quad \|\rho(t)\|_{L^1}= e^{-t}|D_0|.
\end{equation}
The solution to this problem is global in time and takes the form $\rho(t)={\bf{1}}_{D_t}, D_t=\psi(t,D_0)$ where $\psi$ denotes the flow associated to the velocity $v$. Similarly to Euler equations \cite{BC,Chem1}, Bertozzi, Garnett, Laurent and Verdera proved in \cite{BGV} the global in time persistence of the boundary regularity in H\"{o}lder spaces $C^{1+s}, s\in(0,1)$. However the asymptotic behavior  of the patches for large time  is still not well-understood despite some interesting numerical simulations giving some indications on  the concentration dynamics. Notice first that the area of the patch shrinks to zero which entails  that  the associated domains will converge in Hausdorff distance to  negligible sets.  The geometric structure of such sets is not well explored and hereafter we will give   two pedagogic and interesting simple examples illustrating the concentration, and one can find more details in \cite{BLL}. The first example  is the disc which shrinks to its center leading after normalization  procedure to the convergence to Dirac mass. The second one is the ellipse patch which collapses  to a segment along  the big axis and the normalized patch converges weakly to Wigner's semicircle law of density
$$
x_1\mapsto \frac{2\sqrt{{\bf{x_0}}^2-x_1^2}}{\pi {\bf{x_0}}^2}{\bf{1}_{[-x_0,x_0]}},\, {\bf{x_0}}=a-b.
$$
It seems that  the mechanisms governing   the concentration are  very complexe  and  related in part for some special class to the initial distribution of the local mass. Indeed,  the numerical experiments implemented  in \cite{BLL} for some  regular shapes indicate that generically   the concentration is organized along a skeleton structure. 
The aim of this  paper is to investigate  this phenomenon and try  to give  a complete answer for special class of initial data   where the concentration occurs along disjoint segments lying in the same line. More precisely, we will deal with a one-fold symmetric patch, and by rotation invariant we can suppose that it coincides with the real axis. We assume in addition  that the boundary of the upper part is the graph of a  slightly smooth function with small amplitude. Then we will show that we can track the dynamics of the graph globally in time
and  prove that the normalized solution   converges   weakly towards a probability measure supported in the union of disjoint segments lying in the real axis. The results will be formulated rigorously in Section \ref{Sec-graph}. The paper is organized as follows. In next section we formulate the graph equation and state our main results. In Sections 3 and 4 we shall discuss basic tools that we use frequently throughout the paper. In Section 5 we prove the local well-posedness for the graph equation. The global existence with small initial data is proved in Section 6. The last section deals with the asymptotic behavior of the normalized density and its convergence towards a singular measure. 
\section{Graph reformulation and main results}\label{Sec-graph}

%
%

The main purpose of this section  is to describe   the boundary motion of the patch associated to the equation \eqref{sqg} under suitable symmetry structure. One of the basic properties of the aggregation equation that we shall use in a crucial way  concerns its group of symmetry which is much more rich than Euler equations. Actually and in  addition to rotation and translation invariance,  the aggregation equation is in fact invariant by reflexion. 
To check this property and without loss of generality we can look for the invariance with respect to the real axis. Set
  $$
 X=(x,y)\in\R^2\quad\hbox{and}\quad \overline{X}=(x,-y)
 $$ 
 and  introduce 
 $$
 \widehat{\rho}(t,X)=\rho(t,\overline{X}),\quad \widehat{v}(t,X)=-\frac{1}{2\pi}\displaystyle{\int_{\R^2}}\frac{X-Y}{|X-Y|^2}\widehat{\rho}(t,Y) dY.
 $$
Using straightforward change of variables, it is quite easy to get
$$
v(t,X)=\overline{ \widehat{v}(t,\overline{X})},\quad \big[v\cdot\nabla\rho\big](t,X)=\big[\widehat{v}\cdot\nabla \widehat{\rho}\big](t,\overline{X}).
$$
Therefore we find that $\widehat{\rho}$ satisfies also the aggregation equation
$$
\partial_t\widehat{\rho}+\widehat{v}\cdot\nabla \widehat{\rho}=0.
$$
Combining this property with  the uniqueness of Yudovich solutions, it follows that  if  the initial data belongs to $ L^1\cap L^{\infty}$ and admits an axis of symmetry   then  the solution remains invariant with respect to the same axis. In the framework of the vortex patches this result means that if  the  initial data is given by $\rho_0={\bf{1}}_{D_0}$ and the domain $D_0$ is symmetric   with respect to the real axis, the domain $D_t$ defining the solution $\rho(t)={\bf{1}}_{D_t}$ remains  symmetric  with respect to the same axis for any positive time. Recall that in the form \eqref{sqg} Yudovich type solutions are global in time. To be precise about   the terminology, here and contrary to the standard definition of domain in topology which means a connected open set, we  mean  by domain any measurable set of strictly positive measure. In addition, a patch whose domain is symmetric with respect to the real axis (or any axis)  is said one fold symmetric.

Along the current study, we shall focus on  the domains $D_0$ such   that  the  boundary part lying  in the upper half-plane  is   described by the graph of a  $C^1$ positive function $f_0:\R\to \R_+$ with compact support. This is equivalent to say
$$
D_0=\Big\{(x,y)\in\R^2;\, x\in {\textnormal{supp }f_0},\, -f_0(x)\leq y\leq f_0(x)\Big\}.
$$  
We point out that concretely  we shall consider the evolution not of $D_0$ but  of its   extended set defined by
$$
\widehat{D}_0=\Big\{(x,y)\in\R^2;\, x\in\R,\, -f_0(x)\leq y\leq f_0(x)\Big\}.
$$
This does not matter since  the domain $D_t$   remains symmetric with the respect to the real  axis and  then  we can  simply track  its  evolution  by knowing  the dynamics of its extended domain: we just remove the extra lines located on the real axis.

 One of the main objective of this paper  is to follow the dynamics of the graph and investigate local and global well-posedness issues in different function spaces.  In the next lines, we shall   derive  the evolution equation governing the motion  of the initial graph $f_0$. Assume that in a short time interval $[0,T]$ the part of the boundary in the upper half-plane is described by  the graph of a $C^1-$function $f_t:\R\to \R_+$. This forces   the points of the boundary  of $\partial D_t$ located  on the real axis  to be  cusp  singularities.  As a material point located at the boundary remains on the boundary then any parametrization $s\mapsto\gamma_t(s)$ of the boundary should satisfy 
 $$
 \big(\partial_t\gamma_t(s)-v(t,\gamma_t(s))\big)\cdot\vec{n}(\gamma_t(s))=0,
 $$
 with $\vec{n}(\gamma_t)$ being a normal unit vector to the boundary at the point $\gamma_t(s)$. Now take the parametrization in the graph form $\gamma_t: x\mapsto \big(x, f(t,x)\big)$, then the preceding   equation reduces to the nonlinear transport equation
\begin{equation}\label{graph1}
\left\lbrace
\begin{array}{l}
\partial_t f(t,x)+u_1(t,x)\partial_x f(t,x)=u_2(t,x), \, t\geq0, x\in \R\\
f(0,x)=f_0(x),
\end{array}
\right.
\end{equation}
where $(u_1,u_2)(t,x)$ is the velocity $(v_1,v_2)(t,X)$ computed at the point $X=(x,f(t,x))$. Sometimes and along  this paper we use the following notations
$$
f_t(x)=f(t,x)\quad \hbox{and}\quad f^\prime(t,x)=\partial_xf(t,x).
$$
To reformulate  the equation \eqref{graph1} in a closed form we shall recover the velocity components with respect to the graph parametrization.
We start with the  computation of $v_1(X).$ Here and for the sake of  simplicity we drop the time parameter from the graph and the domain of the patch.
One writes according to Fubini's theorem 
\begin{eqnarray*}
-2\pi v_1(X)&=&\bigintss_{D}\frac{x-y_1}{|X-Y|^2} dY,\quad Y=(y_1,y_2)\\
&=&\bigintss_{\R}(x-y_1)\bigintss_{-f(y_1)}^{f(y_1)}\frac{dy_2}{(x-y_1)^2+(f(x)-y_2)^2} dy_1.
\end{eqnarray*}
Using the change of variables $y_2-f(x)=(x-y_1) Z$ we find
\begin{eqnarray*}
2\pi v_1(X)&=&\bigintss_{\R}\Bigg\{\arctan\Big(\frac{f(y)-f(x)}{y-x}\Big)+\arctan\Big(\frac{f(y)+f(x)}{y-x}\Big)\Bigg\} dy\\
&=&\bigintss_{\R}\Bigg\{\arctan\Big(\frac{f(x+y)-f(x)}{y}\Big)+\arctan\Big(\frac{f(x+y)+f(x)}{y}\Big)\Bigg\} dy.
\end{eqnarray*}
To  compute  $v_2$ in terms of $f$ we proceed as before and we find 
\begin{eqnarray*}
-2\pi v_2(X)&=&\bigintss_{D}\frac{f(x)-y_2}{|X-Y|^2} dA(Y)\\
&=&\bigintss_{\R}\bigintss_{-f(y_1)}^{f(y_1)}\frac{f(x)-y_2}{(x-y_1)^2+(f(x)-y_2)^2} dy_2 dy_1.
\end{eqnarray*}
Therefore we obtain the following expression
\begin{equation*}
4\pi v_2(x,f(x))=\bigintss_{\R} \log\Bigg(\frac{y^2+\big( f(x+y)-f(x)\big)^2}{y^2+\big(f(x+y)+f(x)\big)^2} \Bigg)dy.
\end{equation*}
With the notation adopted before for $(u_1,u_2)$ we finally get the formulas 
\begin{eqnarray}\label{fields6}
\nonumber u_1(t,x)&=&\frac{1}{2\pi}\bigintss_{\R}\Bigg\{\arctan\Big(\frac{f_t(x+y)-f_t(x)}{y}\Big)+\arctan\Big(\frac{f_t(x+y)+f_t(x)}{y}\Big)\Bigg\} dy\\
\qquad \qquad \qquad&& u_2(t,x)=\frac{1}{4\pi}\bigintss_{\R} \log\Bigg(\frac{y^2+\big( f_t(x+y)-f_t(x)\big)^2}{y^2+\big(f_t(x+y)+f_t(x)\big)^2} \Bigg)dy.
\end{eqnarray} 

We emphasize  that  for the coherence of the model   the graph equation \eqref{graph1} is supplemented with the initial condition $f_0(x)\geq0$. According to   Proposition \ref{prop20},  the positivity is preserved for enough smooth solutions. Furthermore,  and once again according to  this proposition we have a maximum principle estimate :
$$
\forall t\geq0,\,\forall x\in\R,\quad 0\leq f(t,x)\leq \|f_0\|_{L^\infty}.
$$
Notice that the model remains meaningful   even though  the function $f_t$ changes the sign. In this case the geometric domain of the patch is simply obtained by looking to the region delimited by the curve  of $f_t$ and its symmetric with respect to the real axis. This is also equivalent to deal with  positive function $f_t$ but its graph will be less regular and belongs only to the Lipschitz class.
Another essential element that will be analyzed later in Proposition \ref{prop20} concerns the support of the solutions which remains confined through the time. More precisely, if $\hbox{supp}f_0\subset [a,b]$ with $a<b$ then provided that the graph exists for  $t\in[0,T]$ one has
$$
\hbox{supp} f(t)\subset [a,b].
$$
This follows from the fact that the flow associated to the horizontal velocity $u_1$ is contractive on the boundary. 
It is not clear whether global weak solutions satisfying the maximum principle can be constructed. However, to deal with classical solutions one should control higher regularity of the graph and it seems from the transport structure of the equation  that the  optimal scaling for local well-posedness theory  is Lipschitz class. Denote by $g(t,x)=\partial_xf(t,x)$ the slope of the graph then  it is quite obvious from \eqref{graph1} that
 
 \begin{equation}\label{graphder}
 \partial_t g+u_1\partial_x g=-\partial_x u_1g+\partial_x u_2.
\end{equation}
For the computation of the source term we proceed in a classical way using the differentiation under the integral sign and we get  successively,

 \begin{eqnarray}\label{Derr}
  \nonumber2\pi\partial_x u_1(x)&=&\textnormal{p.v.}\bigintss_{\R}\frac{f^\prime(x+y)-f^\prime(x) }{y^2+(f(x+y)-f(x))^2} ydy\\
&+&\textnormal{p.v.} \bigintss_{\R}\frac{f^\prime(x+y)+f^\prime(x) }{y^2+(f(x+y)+f(x))^2} ydy
 \end{eqnarray}
and 
  \begin{eqnarray}\label{Derr1}
  \nonumber2\pi\partial_x u_2(x)&=& \textnormal{p.v.}\bigintss_{\R}\frac{\big(f(x+y)-f(x)\big)\big(f^\prime(x+y)-f^\prime(x)\big) }{y^2+(f(x+y)-f(x))^2} dy\\
  &-&\textnormal{p.v.}\bigintss_{\R}\frac{\big(f(x+y)+f(x)\big)\big(f^\prime(x+y)+f^\prime(x)\big) }{y^2+(f(x+y)+f(x))^2} dy,
 \end{eqnarray}
 where the notation $\textnormal{p.v.}$ is the Cauchy principal value. It is worthy to point out that the first two integrals appearing in the right hand side of the expressions of $\partial_x u_1$ and $\partial_x u_2$ are in fact connected to Cauchy operator associated to the curve $f$ defined in \eqref{Cauchop}. This operator is well-studied in the literature and some details will be given later in the Section \ref{SingXX}. 
 Next, we shall check that the integrals appearing in the right-hand side of the preceding formulas  can actually be restricted over a compact set related to the support of $f$. Let $[-M,M]$ be a symmetric segment  containing the set   $K_0-K_0$, with $K_0$ being the convexe hull of the support of $f_0$ denoted by $\hbox{supp}f_0$. It is clear that  the support of $\partial_x u_1f^\prime$ is contained in $K_0$ and thus  for $x\in K_0$ one has
$$
\textnormal{p.v.}\bigintss_{\R}\frac{f^\prime(x+y)-f^\prime(x) }{y^2+(f(x+y)-f(x))^2} ydy=\textnormal{p.v.}\bigintss_{-M}^{M}\frac{f^\prime(x+y)-f^\prime(x) }{y^2+(f(x+y)-f(x))^2} ydy.
$$
Consequently, we obtain for $x\in\R$,
 \begin{eqnarray*}
  2\pi f^\prime(x)\partial_x u_1(x)&=&\textnormal{p.v.}\bigintss_{-M}^{M}\frac{f^\prime(x+y)-f^\prime(x) }{y^2+(f(x+y)-f(x))^2} ydy\\
  &-&\textnormal{p.v.} \bigintss_{-M}^{M}\frac{f^\prime(x+y)+f^\prime(x) }{y^2+(f(x+y)+f(x))^2} ydy.
 \end{eqnarray*}
Coming back to the integral representation defining $\partial_x u_2$ one can see,  using a cancellation between both integrals,  that the support of $\partial_x u_2$ is contained in $K_0$. Furthermore, for $x\in K_0$ one may write,
  \begin{eqnarray*}
  2\pi\partial_x u_2(x)&=& \textnormal{p.v.}\bigintss_{-M}^{M}\frac{\big(f(x+y)-f(x)\big)\big(f^\prime(x+y)-f^\prime(x)\big) }{y^2+(f(x+y)-f(x))^2} dy\\
  &-&\textnormal{p.v.}\bigintss_{-M}^{M}\frac{\big(f(x+y)+f(x)\big)\big(f^\prime(x+y)+f^\prime(x)\big) }{y^2+(f(x+y)+f(x))^2} dy.
 \end{eqnarray*}
Gathering the preceding identities we deduce that 
 \begin{equation}\label{Ham11}
2\pi\big(-\partial_x u_1f^\prime(x)+\partial_x u_2\big)=F(x)-G(x)
 \end{equation}
 with
 
 $$
 F(x)\triangleq\textnormal{p.v.}\bigintss_{-M}^{M}\frac{\big[f(x+y)-f(x)-y f^\prime(x)\big]\big(f^\prime(x+y)-f^\prime(x)\big) }{y^2+(f(x+y)-f(x))^2} dy
 $$
 and
 $$
 G(x)\triangleq\textnormal{p.v.}\bigintss_{-M}^{M}\frac{\big[f(x+y)+f(x)+y f^\prime(x)\big]\big(f^\prime(x+y)+f^\prime(x)\big) }{y^2+(f(x+y)+f(x))^2} dy.
 $$
 Keeping in mind, and this will be useful at some points, that the foregoing integrals can be also extended to the full real axis. Sometimes and in order to reduce the size of the integral representation, we use the notations
 \begin{equation}\label{Nota11}
 \Delta_y^{\pm} f(x)=f(x+y)\pm f(x).
\end{equation}
 Thus $F$ and $G$ take the form
 \begin{equation}\label{TataX1}
 F(x)=\textnormal{p.v.}\bigintss_{-M}^{M}\frac{\big[\Delta_y^-f(x)-y f^\prime(x)\big]\Delta_y^-f^\prime(x) }{y^2+(\Delta_y^-f(x))^2} dy
 \end{equation}
 and
 \begin{equation}\label{TataX2}
 G(x)=\textnormal{p.v.}\bigintss_{-M}^{M}\frac{\big[\Delta_y^+f(x)+y f^\prime(x)\big]\Delta_y^+f^\prime(x)}{y^2+(\Delta_y^+f(x))^2} dy.
\end{equation}

The first main first result of this paper is devoted to the local well-posedness issue. We shall discuss  two results related to sub-critical and critical regularities.  Denote by $X$ one of the following spaces: H\"{o}lder spaces $C^s(\R)$ with $s\in(0,1)$ or Dini space $C^\star(\R).$ For more details about classical properties of these spaces we refer the reader  to the Section \ref{SecDini}.
\begin{theorem}\label{thm1}
 Let $f_0$ be a positive compactly supported function such that $f_0^\prime\in X$. Then, the following results hold true.
\begin{enumerate}
\item The equation \eqref{graph1} admits a unique local solution such that  $f^\prime\in L^\infty([0,T],X)$, where the time existence $T$ is related to the norm $\|f^\prime_0\|_{X}$ and the size of the support of $f_0$.
In addition, the solution satisfies the maximum principle
$$
\forall t\in[0,T],\quad \|f(t)\|_{L^\infty}\leq \|f_0\|_{L^\infty}.
$$
\item There exists a constant $\varepsilon>0$ depending only on the size of the support of $f_0$ such that if 
\begin{equation}\label{small1}
\|f^\prime_0\|_{C^s}<\varepsilon
\end{equation}
then the equation  \eqref{graph1} admits a unique global solution $f^\prime\in L^\infty(\R_+;C^s(\R))$. Moreover
$$
\forall t\geq0,\quad \|\partial_x f(t)\|_{L^\infty}\leq C_0 e^{-t}
$$
with $C_0$ a constant depending only on $\|f_0^\prime\|_{C^s}.$
\end{enumerate}
\end{theorem}
Before outlining the strategy of the proofs some comments are in order.
\begin{remarks}
\begin{enumerate}
\item 
The global existence result is only proved for the sub-critical case. The critical case is more delicate to handle due to the lack of strong damping  which is only proved in the  sub-critical case.
\item From Sobolev emebeddings we deduce according to  the assumption on $f_0$ listed in \mbox{Theorem $\ref{thm1}$} that $f_0$ belongs to the space $\in C^1_c(\R)$ of compactly supported $C^1$ functions.
\item The maximum principle holds true globally in time, however it is note clear whether some suitable weak global solutions could be constructed in this setting.
\end{enumerate}
\end{remarks}

Now we shall give some details about the proofs. First we establish  local-in-time  a priori estimates based on the transport structure of the equation combined with some refined studies  on  modified curved Cauchy  operators implemented in Section \ref{SingXX} and essentially based on standard arguments from singular integrals. The construction of the solutions  done in the subsection \ref{Const778} in slightly intricate than the usual schemes used  for transport equations. This is due to the fact that the establishment of the a priori estimates is not only purely energetic. First, at some levels we use some nonlinear  rigidity of the equation like in Theorem \ref{thm1}-$(3)$ where the factor $f^\prime$ behind the operator should be the derivative of the function $f$ that appears inside the operator. Second, we use at some point the fact that the support is confined in time. Last we use at different steps  the positivity of the solution. Hence it seems quite difficult to find a linear scheme taking into account of those constraints. The idea is to implement a nonlinear scheme with two regularizing  parameters $\varepsilon$ and $n$. The first one is used to smooth out  the singularity  of the kernel  and the second  to smooth the solution through a nonlinear scheme. We first establish  that one has uniform a priori estimates on $n$ but on some small interval depending on $\varepsilon$. We are also able  to pass to the limit on $n$ and get a solution for a modified nonlinear problem. Second we check that the a priori estimates still be valid uniformly on $\varepsilon$. This ensures that the time existence can be in fact pushed up to the time given by the a priori estimates obtained for the initial equation \eqref{graph1}.  As a consequence we get  a uniform time existence with respect to $\varepsilon$ and finally we establish the  convergence towards  a solution of the initial value problem using  standard compactness arguments.

The global existence for small initial data requires much more careful analysis because there is no apparent dissipation or damping mechanisms in the equation. Moreover the estimates of the source term $G$ contains some linear parts as it is stated in Proposition \ref{prop10}. The basic ingredient to get rid of those linear parts  is to use a  hidden weak damping effect in $G$  that can  just absorb the growth of the linear part. We do not know if the damping proved for lower regularity still happen in the resolution space.  As to the nonlinear terms, they  are always associated with  some subcritical norms  and thus using  an interpolation argument with the exponential decay of the $L^1$ norm we get a  global in time control that  leads to the global existence.
 
The second result that we shall discuss deals with the asymptotic behavior of the solutions to \eqref{sqg} and \eqref{graph1}. We shall  study the collapse of the support to a collection of disjoint  segments  located at the axis of symmetry. Another interesting issue that will covered by this discussion concerns the characterization of the limit behavior  of the probability measure 

\begin{equation}\label{measprop}
dP_t\triangleq e^t\frac{{\bf{1}}_{D_t}}{|D_0|} dA,
\end{equation} 
with $dA$ being Lebesgue measure and $|D_0|$ denotes the Lebesgue measure of $D_0.$ Our result reads as follows.
\begin{theorem}\label{thm2}
Let $f_0$ be a positive compactly supported function such {that  $f^\prime_0\in C^{s}(\R)$,} with  $s\in(0,1)$. Assume that $\textnormal{supp} f_0$ is the  union of  $n-$disjoint segments and satisfying the smallness \mbox{condition $\eqref{small1}$.} Then there exists  a compact set $D_\infty\subset\R$  composed of   exactly of $n-$disjoint  segments  and a constant $C>0$ such that 
$$
\forall \, t\geq0, \quad d_H(D_t,D_\infty)\leq C e^{-t},\quad |D_\infty|\geq \frac12 |D_0|,
$$
with $d_H$ being the  Hausdorff distance and $|D_\infty|$ is the one-dimensional Lebesgue measure of $D_\infty$. In addition, the probability measures $\{dP_t\}_{t\geq0}$ defined in \eqref{measprop}  converges weakly  as $t$ goes to $+\infty$ to the   probability measure 
$$
dP_\infty:=\Phi\,\delta_{D_\infty\otimes\{0\}},
$$
with $\Phi$ being  a compactly supported   function in $D_\infty$ belonging to $C^\alpha(\R), $ for any $\alpha\in (0,1)$ and  can  be expressed in the form
\begin{equation}\label{Densmea}
\Phi(x)=\frac{f_0(\psi_\infty^{-1}(x))}{\|f_0\|_{L^1}} e^{g(x)},
\end{equation}
with  $g$  a function that can be implicitly  recovered from  the full  dynamics of solution $\{f_t, t\geq0\}$ and 
$$
\psi_\infty=\lim_{t\to+\infty}\psi(t).
$$
Note that $\psi(t)$ is the one-dimensional  flow associated to $u_1$ defined in  \eqref{flot1} and 
$$
D_t=\Big\{(x,y),\, x\in {\textnormal{supp }f_t};\, -f_t(x)\leq y\leq f_t(x)\Big\}.
$$

\end{theorem}
\begin{remark}
The regularity of the profile $\Phi$ might be improved and we expect that $\Phi$ keeps the same regularity as the graph.
\end{remark}
The proof of the collapse of the support to a  disjoint union of segments can be easily derived from the formula \eqref{Densmea} which ensures that the support of the limit measure is exactly the image of the support of $f_0$ by the limit flow $\psi_\infty$ which is a homeomorphism of the real axis. To get the convergence with  the Hausdorff distance we just use the  exponential damping of the amplitude of the curve.  As to the characterization of the limit measure it is based on the exponential decay decay of the amplitude of graph  combined with the scattering as $t$ goes to infinity of the normalized solution $e^tf(t)$. In fact, we prove that the density is nothing but the formal quantity 
$$
\Phi(x)=2\lim_{t\to+\infty} e^{t}f(t,x)
$$
whose existence is obtained using the transport structure of the  equation through the characteristic method  combined with  the damping effects of the nonlinear source terms. 
%
%

\section{Generalities on the limit shapes}
In this short section we shall discuss a simple result dealing with  the role of symmetry in the structure of the limit shape $D_\infty$.  Roughly speaking, we shall prove  that thin initial domains along their axis of symmetry generate concentration to segments. Notice that 
$$
D_\infty\triangleq\Big\{\lim_{t\to+\infty}\psi(t,x), \, x\in D_0\Big\}
$$ 
where $\psi$ is the flow associated to the velocity $v$ and defined through the ODE.
\begin{equation}\label{flot}
\left\lbrace
\begin{array}{l}
\partial_t \psi(t,x)=v(t,\psi(t,x)),\, t\geq0, x\in\R^2, \\
\psi(0,x)=x
\end{array}
\right.
\end{equation}
The existence of the set $D_\infty$ will be proved  below.
We intend to prove the following.
\begin{proposition}\label{Skel}
 The following assertions hold.

\begin{enumerate}
\item If  $D_0$ is a bounded domain of $\R^2$, then  for any $x\in \R^2$ the quantity
$
\displaystyle{\lim_{t\to+\infty}\psi(t,x)}
$ exists.
\item If  $D_0$ is  a simply connected bounded domain symmetric with respect to an axis $\Delta$. Denote by ${d_0}=Length(D_0\cap \Delta).$There exists an absolute constant $C$ such that if
$$
{d_0>C|D_0|^{\frac12}}
$$
then the shape $D_\infty$ contains an interval  of the size ${d_0-C|D_0|^{\frac12}}.$

\end{enumerate}
\end{proposition}  
\begin{proof}
{\bf{(1)}} Integrating in time the flot equation \eqref{flot} yields
$$
\psi(t,x)=x+\int_0^tv(\tau,\psi(\tau,x))d\tau.
$$
Now observe that pointwisely
$$
|v(t,x)|\le\frac{1}{2\pi}\big(\frac{1}{|\cdot|^2}\star|\rho(t)|\big)(x).
$$
Thus  interpolation inequalities combined with \eqref{Exp1} lead to
\begin{eqnarray}\label{Expp11}
\nonumber \|v(t)\|_{L^\infty}&\leq& C\|\rho(t)\|_{L^{1}}^{\frac12}\|\rho(t)\|_{L^\infty}^{\frac12}\\
&\le& Ce^{-\frac{t}{2}}|D_0|^{\frac12},
\end{eqnarray}
with $C$ an absolute constant.
This implies that the integral $\int_0^{+\infty}v(\tau,\psi(\tau,x))d\tau$ converges absolutely and therefore $\displaystyle{\lim_{t\to+\infty}\psi(t,x)}$ exists in $\R^2.$  This allows to define the limit shape  $D_\infty$ as follows:
$$
D_\infty=\Big\{\lim_{t\to+\infty}\psi(t,x), \,\forall\,x\in D_0\Big\}.
$$

{\bf{(2)}} Without loss of generality we will suppose that the straight line  $\Delta$ coincides with the real axis. Since $D$ is simply connected bounded domain, then  there exist two different points $X^-_0, X^+_0\in\R$ such that
$$
 \overline{D_0}\cap\Delta=[X^-_0, X^+_0].
$$
Then it is clear that $\textnormal{Length}(\overline{D_0}\cap\Delta)=X_0^+-X_0^-:=d_0$. By assumption $D_0$ is symmetric with respect to $\Delta$ then the domain $D_t$ remains also  symmetric with respect to the same axis and the points $X^\pm_0$ move necessary along this axis. Denote by
$$
X^\pm(t)=\psi(t, X_0^\pm)
$$
then as the flot is an homeomorphism then 
$$
\overline{D_t}\cap\Delta=[X^-(t), X^+(t)].
$$
Now we wish  to follow the evolution of the distance $d(t):=X^+(t)-X^-(t)$ and find a sufficient condition such that this distance remains away from zero up to infinity. Notice  from the first point that $\displaystyle{\lim_{t\to+\infty} d(t)}$ exists and equals to some positive number $d_\infty$.  From the triangular inequality, one  easily gets that
$$
d(t)\geq d_0-2\int_0^{t}\|v(\tau)\|_{L^\infty}d\tau.
$$
Inequality \eqref{Expp11} ensures that
$$
d(t)\geq d_0-C|D_0|^{\frac12}
$$
and therefore $d_\infty\geq d_0-C|D_0|^{\frac12}.$
Consequently, if $d_0>C|D_0|^{\frac12}$ then the points $\{X^\pm(t)\}$ do not collide up to infinity and thus  the set $D_\infty$ contains a non trivial interval as claimed.
\end{proof}
 \section{Basic properties of Dini and H\"{o}lder spaces}\label{SecDini}
 In this section we set up some function spaces that we shall use and review some of their important properties. 
  Let $f:\R\to\R$ be  a continuous function, we define its modulus of continuity $\omega_f:\R_+\to\R_+$ by
 $$
 \omega_f(r)=\sup_{|x-y|\leq r}|f(x)-f(y)|.
 $$
 This is a nondecreasing function satisfying $\omega_f(0)=0$ and sub-additive, that is for $r_1, r_2\geq0$ we have
 \begin{equation}\label{Sub-ad}
\omega_f(r_1+r_2)\leq \omega_f(r_1)+\omega_f(r_2).
 \end{equation}
 Now we intend to recall Dini and H\"{o}lder spaces. Dini space denoted by  $C^\star(\R)$ is the set  of continuous bounded functions $f$ such that
  \begin{equation*}
 \|f\|_{L^\infty}+\|f\|_{D}<\infty\quad \textnormal{with}\quad \|f\|_{D}=\int_0^1\frac{\omega_f(r)}{r}dr.
\end{equation*}
Another space that we frequently  use throughout this paper is H\"{o}lder space.  Let $s\in(0,1)$ we denote by $C^s(\R)$ the set of functions  $f:\R\to\R$ such that
$$
\|f\|_{L^\infty}+\|f\|_{s}<\infty\quad \textnormal{with}\quad \|f\|_{s}=\sup_{0<r<1}\frac{\omega_f(r)}{r^s}\cdot
$$
 Let $K$ be a compact set of $\R$, we define   $C^\star_K$  as the subspace of $C^\star(\R)$ whose elements are supported in $K.$  Note that $C^\star_K\hookrightarrow L^\infty(\R)$ which means that a constant $C$ depending only on the diameter  of the compact $K$ exists  such that
\begin{equation}\label{Imbed1}
 \forall f\in C^\star_K,\quad \|f\|_{L^\infty}\le C\|f\|_{D}.
\end{equation} 
This follows easily from the observation
$$
\forall r\in (0,1/2], \quad \omega(r)\ln2\le \|f\|_D.
$$
From \eqref{Imbed1} we deduce that  for any $A\geq1$ 
 \begin{eqnarray}\label{L1}
 \nonumber \int_0^A\frac{\omega_f(r)}{r}dr&\le&\|f\|_{D}+2\|f\|_{L^\infty}\ln A\\
 &\le& C\|f\|_{D}\big(1+\ln A\big).
 \end{eqnarray}
Coming back to the definition of Dini semi-norm one deduces  the law products: for $f,g\in C^\star_K$ 
\begin{equation}\label{law2}
 \| fg\|_{D} \le  \| f\|_{L^\infty}\|g\|_D+ \|g\|_{L^\infty}\|f\|_D\quad \hbox{and}\quad  \| fg\|_{D} \le C \| f\|_{D}\|g\|_D.
\end{equation}
Another useful space is   $C^s_K$ which is the subspace of $ C^s(\R)$  whose functions are supported  on the compact $K.$ It is quite obvious that 
\begin{equation}\label{Imbed17}
C^s_K\hookrightarrow  C^\star_K \hookrightarrow  L^\infty.
\end{equation} 
We point out  that all these spaces  are complete. Another property which will be very useful is the following composition law. If $f\in  C^s(\R)$ with $0<s<1$  and $\psi:\R\to\R$  a Lipschitz function then $f\circ\psi\in C^s(\R)$ and
\begin{equation}\label{comp1}
\|f\circ\psi\|_{s}\leq \big[\|f\|_{s}+2\|f\|_{L^\infty}\big]\|\nabla\psi\|_{L^\infty}^s.
\end{equation}

It is worth pointing out that in the case of Dini space $C^\star(\R)$ we get more precise estimate of   logarithmic type,
\begin{equation}\label{comp3}
\|f\circ\psi\|_{D}\leq C\big(\|f\|_{D}+\|f\|_{L^\infty}\big)\Big(1+\ln_+\big(\|\nabla\psi\|_{L^\infty}\big)\Big),
\end{equation}
with the notation
\begin{equation*}
\ln_+ x\triangleq\left\lbrace
\begin{array}{l}
\ln x, \quad \hbox{if} \quad x\geq1\\
0, \quad \hbox{otherwise}.
\end{array}
\right.
\end{equation*}
Another  estimate of great interest is the  following law product,
\begin{equation}\label{lawX3}
 \| fg\|_{s} \le  \| f\|_{L^\infty}\|g\|_{s}+ \|g\|_{L^\infty}\|f\|_{s}.
\end{equation}
In the next task we will be concerned with   a pointwise estimate connecting   a positive smooth function  to its derivative and explore how this property is affected by the regularity. This kind of property will be required in Section \ref{SingXX} in studying  Cauchy operators with special forms.  
   \begin{lemma}\label{lem1}
   Let $K$ be a compact set of $\R$ and $f:\R\to\R_+$ be a continuous positive  function supported in $K$ such that $f^\prime\in C^\star(\R).$ Then we have,
   $$
   \forall \, x\in \R,\quad |f^\prime(x)|\le C\frac{\|f^\prime\|_{D}+\|f^\prime\|_{L^\infty}}{1+\ln_+\big(\frac{\|f^\prime\|_D}{f(x)}\big)}.
    $$
   A weak version of this inequality is 
 $$
\forall x\in \R,\quad  |f^\prime(x)|\leq C\frac{\big(\|f^\prime\|_{D}+\|f^\prime\|_{L^\infty}\big) \big(1+\ln_+(1/\|f^\prime\|_D)}{1+\ln_+(\frac{1}{f(x)}\big)},
 $$
 with $C$ an absolue   constant.
If in addition   $f^\prime\in C^{s}(\R)$ with $s\in (0,1)$, then
 $$
\forall x\in \R,\quad  |f^\prime(x)|\leq C \|f^\prime\|_s^{\frac{1}{1+s}} [f(x)]^{\frac{s}{1+s}}
 $$
 and the constant $C$ depends only on $s$.
   \end{lemma}
   \begin{proof}
 Let $x$ be a given point,  without any loss of generality one can assume that $f^\prime(x)\geq0$. Now let $h\in [0,1]$ then using the mean value theorem, there exists $c_h\in [x-h,x)$ such that
 \begin{eqnarray*}
 f(x-h)&=&f(x)-h f^\prime (c_h)\\
 &=&f(x)-hf^\prime (x) -h[f^\prime(c_h)-f^\prime(x)]\\
 &\le& f(x)-hf^\prime (x)+ h\,\omega_{f^\prime}(h).
 \end{eqnarray*}
From the positivity of the function $f$ we deduce that for any $h\in [0,1]$ one gets 
$$
 f(x)-hf^\prime (x)+ h\,\omega_{f^\prime}(h)
\geq0.
$$
Then dividing by $h^2$ and integrating in $h$ between $\EE$ and $1$ , with  $\EE\in (0,1]$, we get 
$$
 f(x) \frac1\EE+f^\prime (x) \ln\EE+\|f^\prime\|_{D}\geq0.
$$
Multiplying by $\EE$ we obtain
\begin{equation}\label{Eq5}
\forall \,\EE\in(0,1),\quad f(x) +f^\prime (x) \EE \ln \EE+\|f^\prime\|_{D}\, \EE \geq0.
\end{equation}
By studying the variation with respect to $\EE$ we find that the suitable value of $\EE$ is given by 
$$
\ln\EE=-1-\frac{\|f^\prime\|_{D}}{f(x)}.
$$
Inserting  this choice  into \eqref{Eq5} we find that
$$
\EE f^\prime(x)\le f(x)
$$
that is
$$
e^{-1-\frac{\|f^\prime\|_D}{f^\prime(x)}} f^\prime(x)\leq f(x).
$$
From the inequality  $te^{-t}\leq e^{-1}$ we deduce that 
$$
e^{-1}\geq \frac{\|f^\prime\|_D}{f^\prime(x)}e^{-\frac{\|f^\prime\|_D}{f^\prime(x)}}
$$
which implies in turn that
$$
e^{-1-\frac{\|f^\prime\|_D}{f^\prime(x)}} f^\prime(x)\geq e^{- 2\frac{\|f^\prime\|_D}{f^\prime(x)}} \|f^\prime\|_{D}.
$$
Consequently we get
$$
 e^{- 2\frac{\|f^\prime\|_D}{f^\prime(x)}} \|f^\prime\|_{D}\le f(x).
$$
Thus when $\frac{f(x)}{\|f^\prime\|_D}>1$ this estimate does not give any useful information and then we simply write
$$
f^\prime(x)\le \|f^\prime\|_{L^\infty}.
$$ However for $\frac{f(x)}{\|f^\prime\|_D}<1$ we get
 $$
 f^\prime(x)\leq C\frac{\|f^\prime\|_{D}}{1+\ln_+(\frac{\|f^\prime\|_{D}}{f(x)})}.
 $$
 From which we deduce that
 $$
 f^\prime(x)\leq C\frac{\|f^\prime\|_{D} (1+\ln_+(1/\|f^\prime\|_D)}{1+\ln_+(\frac{1}{f(x)})}.
 $$
 Indeed, one may use the estimate
 $$
\forall x>0,\quad  \frac{1+\ln_+(1/x)}{1+\ln_+(a/x)}\leq 1+\ln_+(1/a),
 $$
 which can be checked easily  by studying the variation of the fractional function. 
 
 Now let us move to  the proof when $f^\prime$ is assumed to belong to H\"{o}lder space  $C^s$, with $s\in (0,1)$. Following the same proof as before one deduces that under the assumption  $f^\prime(x)\geq0$ one obtains   for any $h\in \R_+$
 $$
 f(x)-h f^\prime(x)+h^{1+s}\|f^\prime\|_s\geq0.
 $$
 By studying the variation of this function with respect to $h$ we find that  the best choice of $h$ is given by
 $$
 h^s=\frac{f^\prime(x)}{(1+s)\|f^\prime\|_s},
 $$
  which implies the desired result, that is, 
  $$
  f^\prime(x)\le C\|f^\prime\|_s^{\frac{1}{1+s}} [f(x)]^{\frac{s}{1+s}}.
  $$
  The proof is now achieved.
   \end{proof}
   %
   \section{Modified curved Cauchy operators}\label{SingXX}
   This section is devoted to the study of some variants  of Cauchy operators which are closely connected to the operators arising in   \eqref{Derr} and \eqref{Derr1}. Let us first recall the classical Cauchy operator associated to  the graph of a Lipschitz function $f:\R\to\R$ ,
 \begin{equation}\label{Cauchop}
   \mathcal{C}_fg(x)=\bigintsss_{\R}\frac{g(x+y)-g(x)}{y+i(f(x+y)-f(x))}dy.
\end{equation}
   which is well-defined at least for smooth function $g.$ According to a famous theorem of Coifman, McIntosh, and Meyer \cite{Coifman}, this operator can be extended as a bounded operator from $L^p$ to $L^p$ for $1<p<\infty.$ By adapting the proof of the paper of  Wittmann \cite{Wittman}, this operator can also be extended continuously from $C^s_K$ to $C^s(\R)$ for $0<s<1$, provided that $f$ belongs to $C^{1+s}(\R).$ However  this operator  fails to be extended continuously  from Dini space $C^\star_K$ to itself as it can be checked from Hilbert transform.  The structure of the operators that we have to deal with, as one may observe from the expression of $F$  following  \eqref{Ham11}, is slightly different from the Cauchy operators. It can be associated to the truncated  bilinear Cauchy operator defined as follows: for given $M>0$,  $\theta\in [0,1]$,
    \begin{equation*}
   \mathcal{C}_f^\theta(g,h)(x)=\bigintss_{-M}^M\frac{\big(g(x+\theta y)-g(x)\big)\big(h(x+y)-h(x)\big)}{y+i(f(x+y)-f(x))}dy.
\end{equation*}
The real and imaginary parts of this operator are given respectively by 
\begin{equation}\label{BiCauchop}
   \mathcal{C}_f^{\theta,\Re}(g,h)(x)=\bigintss_{-M}^M\frac{y\big(g(x+\theta y)-g(x)\big)\big(h(x+y)-h(x)\big)}{y^2+[f(x+y)-f(x)]^2}dy
\end{equation}
and
\begin{equation*}
   \mathcal{C}_f^{\theta,\Im}(g,h)(x)=-\bigintss_{-M}^M\frac{\big(f(x+y)-f(x)\big)\big(g(x+\theta y)-g(x)\big)\big(h(x+y)-h(x)\big)}{y^2+[f(x+y)-f(x)]^2}dy.
\end{equation*}
   In what follows we denote by $X$ one of the spaces $C^s_K,$ with $0<s>1$ or $C^\star_K$. The  result that we shall discuss deals with  the continuity of the bilinear operator on the spaces $X$. This could have been discussed in the literature and as we need to control the continuity constant we shall give a detailed proof.
   \begin{proposition}\label{propCart1}
  Let $K$ be a compact  set of $\R$ and $f$ be  a compactly supported function such that $f^\prime\in X$. Then the  following assertions hold true.
   \begin{enumerate}
   \item The bilinear operator  
   ${\mathcal{C}}_f^\theta: X\times X\to X
   $ is well-defined and continuous. More precisely, there exits a constant $C$ independent  of $\theta$ such that for any $g,h\in X$
    $$
 \|{\mathcal{C}}_f^{\theta,\Re}(g,h)\|_{X}\leq C\big(1+\|f^\prime\|_{L^\infty}\|f^\prime\|_X\big)\Big(\| g\|_{D}\|h\|_{X}+\| h\|_{D}\|g\|_{X}\Big)   
   $$
and 
   $$
  \|{\mathcal{C}}_f^{\theta,\Im}(g,h)\|_{X}\leq C \|f^\prime\|_{X}\big(1+\|f^\prime\|_{L^\infty}^2\big)\Big(\|g\|_{D}\| h\|_{X}+\|g\|_{X}\|h\|_D\Big).
   $$
  
     \end{enumerate}
   
   \end{proposition}
   \begin{proof}
   We shall first establish the result for the real part operator given by \eqref{BiCauchop}. First we note that one may  rewrite the expression using the  notation \eqref{Nota11} as follows
   $$
{\mathcal{C}}_f^{\theta,\Re}(g,h)(x)=\bigintsss_{-M}^M\frac{y\Delta_{\theta y}g(x)\Delta_yh(x)}{y^2+(\Delta_yf(x))^2}dy
   $$
   where we simply replace the notation $\Delta_y^-$ by $\Delta_y$.
   Using the law products \eqref{law2} and \eqref{lawX3} one obtains 
  \begin{eqnarray*}
 \|{\mathcal{C}}_f^{\theta,\Re}(g,h)\|_{X}&\le&\bigintsss_{-M}^{M}{\|\Delta_{\theta y}g\Delta_y h\|_{X} }\frac{dy}{|y|}\\
 &+&\bigintsss_{-M}^{M}{|y|\|\Delta_{\theta y}g\Delta_y h\|_{L^\infty} }\Big\|\frac{1}{y^2+(\Delta_yf)^2}\Big\|_{X} dy.
  \end{eqnarray*}
Using once again those law products  it comes
  \begin{eqnarray*}
 \|\Delta_{\theta y}g\Delta_y h\|_{X}&\le&   \|\Delta_{\theta y}g\|_{L^\infty}\|\Delta_y h\|_{X}+\|\Delta_{\theta y}g\|_{X}\|\Delta_y h\|_{L^\infty}\\
 &\le& \omega_{g}(|y|)\| h\|_{X}+2\|g\|_{X}\omega_{h}(|y|),
   \end{eqnarray*}
   where we have used that for $\theta\in[0,1], y\in\R$ 
  \begin{equation}\label{Rod1}
   \|\Delta_yh \|_{X}\le 2 \|h\|_{X},\quad \|\Delta_{\theta y}h\|_{L^\infty}\le \omega_{h}(|y|).
 \end{equation}
     
   Consequently
 \begin{equation}\label{tat1}
  \bigintsss_{-M}^{M}{\|\Delta_{\theta y}g\Delta_y h\|_{X} }\frac{dy}{|y|}\le  C\big(\| g\|_{D}\|h\|_{X}+\| h\|_{D}\|g\|_{X}\big).
 \end{equation}
   
   By the definition it is quite easy to check that for any function $\varphi\in X \cap L^\infty(\R)$
   \begin{eqnarray*}
 \Big\|\frac{1}{y^2+\varphi^2}\Big\|_{X}\le\frac{2\|\varphi\|_{L^\infty}}{y^4}\|\varphi\|_{X}.
   \end{eqnarray*}
   Hence we get
      \begin{eqnarray}\label{frac1}
  \nonumber  \Big\|\frac{1}{y^2+(\Delta_yf)^2}\Big\|_{X}&\le& 2 \frac{\|\Delta_yf\|_{L^\infty}}{y^4}\|\Delta_yf\|_{X}\\
   &\le& C y^{-2}\|f^\prime\|_{L^\infty}\|f^\prime\|_{X}
     \end{eqnarray}
     where we have used the inequalities
     $$
     \|\Delta_y f\|_{L^\infty}\le |y|\|f^\prime\|_{L^\infty}\quad \textnormal{and}\quad \omega_{\Delta_y f}(r)\le |y|\omega_{f^\prime}(r).
     $$
     Therefore we get in view of \eqref{Rod1},
     
           \begin{eqnarray*}
    \bigintsss_{-M}^{M}{|y|\|\Delta_{\theta y}g\Delta_y h\|_{L^\infty} }\Big\|\frac{1}{y^2+(\Delta_yf)^2}\Big\|_{X} dy&\le& C\|f^\prime\|_{L^\infty}\|f^\prime\|_X\|h\|_{L^\infty} \bigintsss_{-M}^M\frac{\omega_{g}(|y|)}{|y|} dy\\
     &\le&   C\|f^\prime\|_{L^\infty}\|f^\prime\|_X\|h\|_{L^\infty} \|g\|_{D}.
           \end{eqnarray*}    
      Combining this last estimate with \eqref{tat1} we find that
      $$
      \|{\mathcal{C}}_f^{\theta,\Re}(g,h)\|_{X}\le   C\Big(\| g\|_{D}\|h\|_{X}+\| h\|_{D}\|g\|_{X}+\|f^\prime\|_{L^\infty}\|f^\prime\|_X\|h\|_{L^\infty} \|g\|_{D}\Big).
             $$
    To deduce the result it is enough to use \eqref{Imbed17}.
               
We are left with the task of  estimating  the imaginary part which takes the form
 \begin{equation*}
   \mathcal{C}_f^{\theta,\Im}(g,h)(x)=\bigintsss_{-M}^M\frac{\Delta_yf(x)\Delta_{\theta y}g(x)\Delta_yh(x)}{y^2+(\Delta_yf(x))^2}dy.
\end{equation*}  
Note that we have dropped the sign minus before the integral which of course has no consequence on the computations.   
Using Taylor formula we get
$$
\Delta_yf(x)=y\int_0^1f^\prime(x+\tau y) d\tau
$$
and thus
\begin{equation*}
   \mathcal{C}_f^{\theta,\Im}(g,h)(x)=\bigintsss_{-M}^M\bigintsss_0^1\frac{yf^\prime(x+\tau y)\Delta_{\theta y}g(x)\Delta_yh(x)}{y^2+(\Delta_yf(x))^2}dyd\tau.
\end{equation*}
It suffices to reproduce the preceding   computations using in particular the estimates
$$
\|f^\prime(\cdot+\tau y)\Delta_{\theta y}g\Delta_yh\|_{L^\infty}\leq\|f^\prime\|_{L^\infty}\|h\|_{L^\infty}\omega_g(|y|)
$$
and
\begin{eqnarray*}
\|f^\prime(\cdot+\tau y)\Delta_{\theta y}g\Delta_yh\|_{X}\leq\|f^\prime\|_{L^\infty}\|\Delta_{\theta y}g\Delta_yh\|_{X}+\|f^\prime\|_{X}\|\Delta_{\theta y}g\Delta_yh\|_{L^\infty}\\
\leq \|f^\prime\|_{L^\infty}\Big(\omega_{g}(|y|)\| h\|_{X}+\|g\|_{X}\omega_{h}(|y|)\Big)+2\|f^\prime\|_{X}\|\|g\|_{L^\infty}\omega_h(|y|).
\end{eqnarray*}
This implies according to  Sobolev embedding \eqref{Imbed17}
\begin{eqnarray*}
\bigintsss_{-M}^M\bigintsss_0^1{\|f^\prime(\cdot+\tau y)\Delta_{\theta y}g\Delta_yh\|_X}\frac{dy}{|y|}d\tau\leq C \|f^\prime\|_{X}\Big(\|g\|_{D}\| h\|_{X}+\|g\|_{X}\|h\|_D\Big).
\end{eqnarray*}
Using \eqref{frac1}  one may easily get
\begin{eqnarray*}
   \bigintsss_{-M}^{M}{|y|\| f^\prime(\cdot+\tau y)\Delta_{\theta y}g\Delta_y h\|_{L^\infty} }\Big\|\frac{1}{y^2+(\Delta_yf)^2}\Big\|_{X} dy     &\le&   C\|f^\prime\|_{L^\infty}^2\|f^\prime\|_X\|h\|_{L^\infty} \|g\|_{D}      \end{eqnarray*} 
    which gives the desired result using Sobolev embeddings \eqref{Imbed17}. The proof of the proposition is now achieved.
   \end{proof}

The second kind  of Cauchy integrals that we have to deal with and  related to the integral terms  in \eqref{Derr} and \eqref{Derr1}  is  given by the following linear operators
$$
   T^{\alpha,\beta}_fg(x)=\textnormal{p.v.}\bigintsss_{\R}\frac{y \, g(\alpha x+\beta y)}{y^2+[f(x)+f(x+y)]^2} dy
   $$
   with $\alpha$ and $\beta$ two parameters. The continuity of these operators in classical Banach spaces is  not in general easy to establish  and could fail for some special cases. We point out that it is not our purpose in this exposition to implement a complete  study of those operators. A more complete theory may be achieved but this topics exceeds the scope of this paper and     we shall restrict ourselves to some special configurations that fit with the application to the aggregation equation.  Our result in this direction reads as follows. 
   \begin{theorem}\label{th1}
   Let $\alpha, \beta\in [0,1]$, $K$ be a compact set of $\R$ and  $f:\R\to\R_+$ be a compactly supported continuous positive function such that  $f^\prime\in C^\star_K$. 
   Then  the following assertions hold true.
   \begin{enumerate}
      
   \item The operator $T^{\alpha,\beta}_f:  C^\star_K\to L^\infty(\R)$ is well-defined and continuous
   $$
   \|T^{\alpha,\beta}_fg\|_{L^\infty}\le C  \Big(1+\|f^\prime\|_{L^\infty}^2+\|f^\prime\|_{L^\infty}\|f^\prime\|_{D}\Big) \|g\|_{D} $$
   with $C$  a  constant depending only on $K$  and not on $\alpha$ and $\beta$. 
      \item  The modified operator    $f^\prime T_f^{\alpha,\beta}:  C^\star_K\to C^\star_K$ is continuous. More precisely,  $$
   \|f^\prime T^{\alpha,\beta}_fg\|_{D}\le C  \|f^\prime\|_{D} \Big(C_\beta \ln_+(1/\|f^\prime\|_D) +\|f^\prime\|_{D}^{14}\Big) \|g\|_{D}
   $$
   with $C$ a constant depending only on $K$ and 
   \begin{equation*}
C_\beta\triangleq\left\lbrace
\begin{array}{l}
(1-\ln\beta), \quad \beta\in(0,1] \\
1,\quad \beta=0.
\end{array}
\right. \end{equation*}
\item Let $s\in(0,1)$ and assume that  $f^\prime \in C^s_K$, then $f^\prime T^{\alpha,\beta}_f:  C^s_K\to C^s_K(\R)$ is well-defined and continuous. More precisely,  there exists a constant $C$ depending only on the compact $K$ and $s$ such that
\begin{equation}\label{hmi23}
\|f^\prime T^{\alpha,\beta}_fg\|_s\leq C\Big(C_\beta  \|f^\prime\|_{L^\infty}^{\frac{1}{1+s}}+\|f^\prime\|_s^{14}\Big)  \|g\|_{s}.\end{equation}
In addition, one has the refined estimate
\begin{eqnarray}\label{hmi24}
\nonumber\|f^\prime T_f^{\alpha,\beta}g\|_s&\leq &C \|f^\prime\|_{L^\infty}^{\frac{1}{2+s}}\Big[\|f^\prime\|_s^{\frac{1}{2+s}}C_\beta+\|f^\prime\|_s^{14}\Big] \|g\|_{s} \\
&+&C\|g\|_{L^\infty}^{\frac{1}{2+s}}\|g\|_s^{\frac{1+s}{2+s}}\|f^\prime\|_{s},
\end{eqnarray}
 with
  \begin{equation*} 
C_\beta\triangleq\left\lbrace
\begin{array}{l}
\beta^{-\frac12}, \quad \beta\in(0,1] \\
1,\quad \beta=0.
\end{array}
\right. \end{equation*}
   \end{enumerate}
   \end{theorem}
\begin{proof}

 To alleviate the notation we shall along this proof write $T_fg$ instead of $T_f^{\alpha,\beta}g$.
 \vspace{0,3cm}
 
${\bf 1)}$ By symmetrizing we get
\begin{eqnarray}\label{Eq10}
\nonumber T_f g(x)&=&\bigintsss_{0}^{+\infty}\frac{y \, \big[g(\alpha x+\beta y)-g(\alpha x-\beta y)\big]}{y^2+[f(x)+f(x+y)]^2} dy\\
\nonumber &+&\lim_{\varepsilon\to0}\bigintsss_{\varepsilon}^{+\infty}\frac{y \,g(\alpha x-\beta y)\big[f(x-y)-f(x+y)\big]\big[\Delta_y^+f(x)+\Delta_{-y}^+f(x)\big]}{\Big(y^2+[\Delta_{y}^+f(x)]^2\Big) \Big(y^2+[\Delta_{-y}^+f(x)]^2\Big)} dy\\
&\triangleq&T_f^1g(x)+T_f^2 g(x).
\end{eqnarray}
Without loss of generality we can assume that   $K=[-1,1]$ and $\textnormal{supp}{g}\subset [-1,1]$ and deal only with     $x\geq0$. We shall distinguish two cases  $0\le \alpha x\leq2$ and $\alpha x\geq2$. In the first case reasoning on the support of $g$ we simply  get 
$$
T_f^1 g(x)=\bigintsss_{\{0\leq \beta y\leq 3\}}\frac{y \, \big[g(\alpha x+\beta y)-g(\alpha x-\beta y)\big]}{y^2+[f(x)+f(x+y)]^2} dy.
$$
Hence we get by the  definition of the modulus of continuity, a change of variables and \eqref{L1}
\begin{eqnarray}\label{tre1}
\nonumber |T_f^1g(x)|&\leq &\bigintsss_{\{0\leq \beta y\leq 3\}}\frac{ \omega_g(2\beta y)}{y} dy\\
&\le& C\|g\|_{D}.
\end{eqnarray}
Coming back to the case $\alpha x\geq 2$ one may write \begin{eqnarray*}
|T_f^1 g(x)|
&\le&\bigintsss_{\{\alpha x-1\le\beta y\leq 1+\alpha x\}}\frac{ \omega_g( 2\beta y)}{y} dy\\
&\le&2\|g\|_{L^\infty}\int_{\alpha x-1}^{1+\alpha x}\frac{1}{y} dy\\
&\leq& \|g\|_{L^\infty}\ln\Big(\frac {1+\gamma}{-1+\gamma}\Big), \quad \gamma={\alpha x}\geq2\\
&\le& C \|g\|_{L^\infty}.
\end{eqnarray*}
Combining this last inequality with \eqref{tre1}  we  deduce that
\begin{equation}\label{HK1}
\|T_f^1 g\|_{L^\infty}\leq C\|g\|_{D}.
\end{equation}
For the second term $T_f^2 g$  we split it into two parts  as follows  
\begin{eqnarray}\label{lila90}
\nonumber T_f^2 g(x)&=&\lim_{\varepsilon\to0}4f(x)\bigintsss_{\varepsilon}^{+\infty}\frac{y \,g(\alpha x-\beta y)\big[f(x-y)-f(x+y)\big]}{\Big(y^2+[f(x)+f(x+y)]^2\Big) \Big(y^2+[f(x)+f(x-y)]^2\Big)} dy\\
\nonumber &+&\bigintsss_{0}^{+\infty}\frac{y \,g(\alpha x-\beta y)\big[f(x-y)-f(x+y)\big]\psi(x,y)}{\Big(y^2+[f(x)+f(x+y)]^2\Big) \Big(y^2+[f(x)+f(x-y)]^2\Big)} dy\\
&\triangleq&T_f^{2,1}g(x)+T_f^{2,2}g(x)
\end{eqnarray}
with
\begin{eqnarray*}
\psi(x,y)&=&f(x+y)+f(x-y)-2f(x)\\
&=&y\int_0^1\big[f^\prime(x+\theta y)-f^\prime(x-\theta y)\big] d\theta.
\end{eqnarray*}
The first term $T_f^{2,1}g$ is easily estimated. Indeed, one can assume that $f(x)>0$, otherwise the integral vanishes. Thus  using the mean value theorem and a change of variables we obtain
\begin{eqnarray}\label{embedd11}
\nonumber|T_f^{2,1} g(x)|&\le &8\|g\|_{L^\infty}\|f^\prime\|_{L^\infty}f(x)\bigintsss_{0}^{+\infty}\frac{y^2}{\Big(y^2+[f(x)]^2\Big)^2} dy\\
\nonumber&\leq &8\|g\|_{L^\infty}\|f^\prime\|_{L^\infty}\bigintsss_{0}^{+\infty}\frac{y^2}{\big(y^2+1\big)^2} dy\\
&\leq& C \|g\|_{L^\infty}\|f^\prime\|_{L^\infty}.
\end{eqnarray}
As to  the  term $T_f^{2,2}$ straightforward arguments yield
\begin{eqnarray*}
|T_f^{2,2}g(x)|&\leq&8\|g\|_{L^\infty}\|f\|_{L^\infty}^2\bigintsss_{y\geq\frac12}\frac{1}{y^3} dy+2\|g\|_{L^\infty}\|f^\prime\|_{L^\infty}\bigintsss_{0}^{\frac12}\frac{ |\psi(x,y)|}{y^2} dy\\
&\leq& C\|g\|_{L^\infty}\Bigg(\|f\|_{L^\infty}^2+C\|f^\prime\|_{L^\infty}\bigintsss_{0}^{\frac12}\frac{ \omega_{f^\prime}(2y)}{y}dy\Bigg)\\
&\leq& C\|g\|_{L^\infty}\Big(\|f^\prime\|_{L^\infty}^2+C\|f^\prime\|_{L^\infty}\|f^\prime\|_D\Big)
\end{eqnarray*}
where we have used he  fact
$$
|\psi(x,y)|\le 2y \omega_{f^\prime}(2y)|.
$$
Consequently we obtain
\begin{equation}\label{Eq6}
\|T_f^{2}g\|_{L^\infty}\le C\|g\|_{L^\infty}\Big(\|f^\prime\|_{L^\infty}^2+\|f^\prime\|_{L^\infty}\|f^\prime\|_{D}+\|f^\prime\|_{L^\infty}\Big).
\end{equation}
Putting together this estimate with \eqref{embedd11} and \eqref{Imbed1} we obtain the desired estimate.
\vspace{0,3cm}

${\bf{2)}}$ First, recall from the point ${\bf{(1)}}$ of this proof  the following  decomposition
\begin{equation}\label{splitjun1}
T_f g(x)=T_f^1 g(x)+T_f^{2,1}g(x)+T_f^{2,2}g(x).
\end{equation}
The second term is more easier to deal with and  one has
\begin{eqnarray}\label{Eqw23}
\|T_f^{2,1}g\|_{D}\le C\|g\|_D\|f^\prime\|_{D}(1+\|f^\prime\|_{L^\infty}^{13}).
\end{eqnarray}
This implies in view of the law products \eqref{law2} and \eqref{embedd11} that
\begin{eqnarray}\label{EqwZZ23}
\|T_f^{2,1}g\|_{D}\le C\|g\|_D\|f^\prime\|_{D}(\|f^\prime\|_{L^\infty}+\|f^\prime\|_{L^\infty}^{14}).
\end{eqnarray}
To establish \eqref{Eqw23} we first note that when $f(x)=0$ then $T_f^{2,1}g(x)=0$. However for $f(x)>0$,  using the mean value theorem and a change of variables  $y\to f(x)y$ we get
\begin{equation}\label{xident1}
T_f^{2,1}g(x)=-4\bigintsss_{0}^{+\infty}\frac{y^2 \,g(\alpha x-\beta f(x) y)\int_0^1[f^\prime(x+\theta f(x) y)+f^\prime(x-\theta f(x) y)]d\theta}{\varphi(x,y)\varphi(x,-y)} dy
\end{equation}
with 
$$
\varphi(x,y)=y^2+\Big[2+y\int_0^1f^\prime(x+\theta f(x) y)d\theta\Big]^2.
$$
Observe that the identity \eqref{xident1} is meaningful even for $f(x)=0$ and we can check easily that it vanishes. This follows from the fact that  owing to the positivity of $f$ when $f(x)=0$ then necessary $f^\prime(x)=0.$
To alleviate the expressions we introduce the functions
$$
\mathcal{N}_1(x,y)= g\big(\alpha x-\beta f(x) y\big)\int_0^1\big[f^\prime\big(x+\theta f(x) y\big)+f^\prime\big(x-\theta f(x) y\big)\big]d\theta
$$
and
$$
\mathcal{D}_1(x,y)=\varphi(x,y)\varphi(x,-y).
$$
Then by \eqref{law2}  we obtain for fixed $y$
\begin{eqnarray*}
\|\mathcal{N}_1(\cdot,y)\|_{D}&\leq& 2 \|g\circ\big(\alpha \hbox{Id}-\beta y f\big)\|_{D}\|f^\prime\|_{L^\infty}\\
&+&\|g\|_{L^\infty}\int_0^1\big[\|f^\prime\circ\big(\hbox{Id}+\theta y f \big)\|_{D}+\|f^\prime\circ\big(\hbox{Id}-\theta y f \big)\|_{D}\big]d\theta.
\end{eqnarray*}
Using the composition law \eqref{comp3} we get successively 
$$
\|g\circ\big(\alpha \hbox{Id}-\beta y f\big)\|_{D}\le C \|g\|_{D} \Big(1+\ln_+\big(\alpha +\beta \|f^\prime\|_{L^\infty} y\big)\Big)
$$
and
$$
\|f^\prime\circ\big(\hbox{Id}+\theta y f \big)\|_{D}\le C \|f^\prime\|_{D}\Big(1+\ln\big(1+\theta \|f^\prime\|_{L^\infty} y\big)\Big).
$$
This implies that
\begin{eqnarray*}
\|\mathcal{N}_1(\cdot,y)\|_{D}
&\le&C \|g\|_{D}\Big(1+\ln_+\big(\alpha +\beta \|f^\prime\|_{L^\infty} y\big)\Big)\|f^\prime\|_{L^\infty}\\
&+&C\|g\|_{L^\infty}\|f^\prime\|_{D}\int_0^1\Big(1+\ln\big(1+\theta \|f^\prime\|_{L^\infty} y\big)\Big)d\theta.
\end{eqnarray*}
Since
$$
\ln(1+\Pi_{i=1}^n x_i)\leq \sum_{i=1}^n\ln(1+x_i), \forall x_i\geq0
$$
then 
\begin{equation}\label{G11}
\|\mathcal{N}_1(\cdot,y)\|_{D}\leq C\|g\|_{D}\|f^\prime\|_{D}\Big(1+\ln_+\|f^\prime\|_{L^\infty}+\ln_+ y\Big).
\end{equation}
On the other hand it is plain that 
\begin{equation}\label{qwf1}
\|\mathcal{N}_1(\cdot,y)\|_{L^\infty}\leq C\|g\|_{L^\infty}\|f^\prime\|_{L^\infty}.
\end{equation}
To estimate $\frac{1}{\mathcal{D}_1(\cdot,y)}$ in Dini space $C^\star_K$ we come back to the definition which implies that
\begin{equation}\label{G12Z}
\|1/\mathcal{D}_1(\cdot,y)\|_{D}\leq  \|\mathcal{D}_1(\cdot,y)\|_{D}\|1/\mathcal{D}_1(\cdot,y)\|_{L^\infty}^2.
\end{equation}
Now using  the law product \eqref{law2} we deduce that
\begin{eqnarray*}
\|\mathcal{D}_1(\cdot,y)\|_{D}&\leq& \|\varphi(\cdot,y)\|_{L^\infty}\|\varphi(\cdot,-y)\|_{D}+\|\varphi(\cdot,y)\|_{D}\|\varphi(\cdot,-y)\|_{L^\infty}.
\end{eqnarray*}
From simple calculations we get
\begin{eqnarray*}
\|\varphi(\cdot,\pm y)\|_{L^\infty}&\leq& y^2+\big(2+y\|f^\prime\|_{L^\infty}\big)^2\\
&\le&C (1+\|f^\prime\|_{L^\infty}^2)(1+y^2).
\end{eqnarray*}
  Applying \eqref{law2} and \eqref{comp3} to the expression of $\varphi$ it is quite easy to check that
  
  \begin{eqnarray*}
\|\varphi(\cdot,\pm y)\|_{D}&\leq& C\big(1+y\|f^\prime\|_{L^\infty}\big) y\int_0^1\|f^\prime\circ(\hbox{Id}\pm \theta y f )\|_{D} d\theta \\
&\le& C\big(y+y^2\|f^\prime\|_{L^\infty}\big)\|f^\prime\|_{D}\Big(1+\ln_+\|f^\prime\|_{L^\infty}+\ln_+ y\Big).
 \end{eqnarray*}
 Thus combining the preceding estimates we find 
 
\begin{eqnarray}\label{Hamo1}
\nonumber \|\mathcal{D}_1(\cdot,y))\|_{D}&\le& C\big(y+y^2\|f^\prime\|_{L^\infty}\big)\|f^\prime\|_{D}\Big(1+\ln_+\|f^\prime\|_{L^\infty}+\ln_+ y\Big)(1+\|f^\prime\|_{L^\infty}^2)(1+y^2)\\
&\le& C\big(1+y^4\ln_+y)\big)\|f^\prime\|_{D}\Big(1+\ln_+\|f^\prime\|_{L^\infty}\Big)(1+\|f^\prime\|_{L^\infty}^3).
 \end{eqnarray}
 Now we shall use the following inequalities that can be proved in a straightforward way: for any $y\in \R_+$  and for any  $a,b\in\R$ with $|a|\leq b$, one has
 
 \begin{eqnarray*}
 y^2+(2+ya)^2&\geq& y^2+(2-ya)^2\\
 &\geq&\frac{1+y^2}{1+a^2}\\
 &\geq&\frac{1+y^2}{1+b^2}\cdot
\end{eqnarray*}
It follows that
  \begin{eqnarray}\label{Eqw1}
 \|1/\varphi(\cdot,\pm y)\|_{L^\infty}&\leq& 
 \frac{1+\|f^\prime\|_{L^\infty}^2}{1+y^2}.
\end{eqnarray}
  Putting  this estimate together  with \eqref{Hamo1} and \eqref{G12Z}  yields
  \begin{eqnarray*}
\|1/\mathcal{D}_1(\cdot,y)\|_{D}
&\le& C\frac{1+y^4\ln_+y}{1+y^8}\|f^\prime\|_{D}\Big(1+\ln_+\|f^\prime\|_{L^\infty}\Big)(1+\|f^\prime\|_{L^\infty}^{11}) \\
&\le& C \frac{1+\ln_+y}{1+y^4}\|f^\prime\|_{D}(1+\|f^\prime\|_{L^\infty}^{12})  .
\end{eqnarray*}
    Therefore we obtain using \eqref{G11}, \eqref{qwf1} and \eqref{Eqw1} 
       \begin{eqnarray*}
\|(\mathcal{N}_1/\mathcal{D}_1)(\cdot,y)\|_{D}&\leq&\|(\mathcal{N}_1(\cdot,y)\|_{L^\infty}\|1/\mathcal{D}_1)(\cdot,y)\|_{D}+\|(\mathcal{N}_1(\cdot,y)\|_{D}\|1/\mathcal{D}_1)(\cdot,y)\|_{L^\infty}\\
&\le&C \|g\|_{L^\infty}\|f^\prime\|_{L^\infty} \frac{1+\ln_+y}{1+y^4}\|f^\prime\|_{D}(1+\|f^\prime\|_{L^\infty}^{12})\\
&+&C\|g\|_{D}\|f^\prime\|_{D}\frac{1+\ln_+\|f^\prime\|_{L^\infty}+\ln_+ y}{1+y^4}(1+\|f^\prime\|_{L^\infty}^{4})\\
&\le&C \|g\|_{D}\|f^\prime\|_{D}\frac{1+\ln_+y}{1+y^4}(1+\|f^\prime\|_{L^\infty}^{13}) .
\end{eqnarray*}
%
Plugging  this estimate into \eqref{xident1} we find
\begin{eqnarray}\label{June4w}
\nonumber \|T_f^{2,1}g\|_{D}&\le& 4\int_{0}^{+\infty} y^2 \|(\mathcal{N}_1/\mathcal{D}_1)(\cdot,y)\|_{D} dy\\
&\le&
 C \|g\|_{D}\|f^\prime\|_{D}(1+\|f^\prime\|_{L^\infty}^{13}) .
\end{eqnarray}
This concludes the proof of \eqref{Eqw23}.

Now we intend to estimate $\|T_f^1g\|_{D}$ which is more tricky. Let  $r\in(0,1)$ and $x_1,x_2\in\R$ such that $|x_1-x_2|\le  r$. We shall decompose $T_f^1 g$ as follows 
\begin{equation}\label{jul34}
T_f^1g=T_{f,\textnormal{int}}^{ r,1}\,g+T_{f,\textnormal{ext}}^{ r,1}\,g
\end{equation}
with
$$
T_{f,\textnormal{int}}^{ r,1}\,g(x)=\bigintss_0^r\frac{y \, \big[g(\alpha x+\beta y)-g(\alpha x-\beta y)\big]}{y^2+[f(x)+f(x+y)]^2} dy
$$
and
$$
T_{f,\textnormal{ext}}^{ r,1}\,g(x)=\bigintss_{r}^{+\infty}\frac{y \, \big[g(\alpha x+\beta y)-g(\alpha x-\beta y)\big]}{y^2+[f(x)+f(x+y)]^2} dy.
$$
From the sub-additivity of the modulus of continuity we get
\begin{eqnarray*}
|f^\prime(x)T_{f,\textnormal{int}}^{ r,1}\,g(x)|&\leq& C |f^\prime(x)|\int_0^r\frac{y \omega_g( y)}{y^2+[f(x)]^2} dy\\
&\leq&C|f^\prime(x)|\int_0^r\frac{ \omega_g( y)}{y+f(x)} dy.
\end{eqnarray*}
Using Lemma \ref{lem1} we find
\begin{eqnarray}\label{TX02}
|f^\prime(x)T_{f,\textnormal{int}}^{ r,1}\,g(x)|
&\leq&C\frac{\gamma(f)}{1+\ln_+(\frac{1}{f(x)})}\int_0^r\frac{ \omega_g( y)}{y+f(x)} dy
\end{eqnarray}
where 
\begin{equation}\label{gammaZ}
\gamma(f)\triangleq\|f^\prime\|_{D}\big(1+\ln_+(1/\|f^\prime\|_D)\big).
\end{equation}
Now we claim that: for $y\in (0,1)$
\begin{equation}\label{gkm11}
\sup_{\EE>0}\frac{1}{1+\ln_+(1/ \EE)}\frac{1}{y+\EE}\leq \frac{C}{y(1+|\ln y|)}+\frac{1}{1+y}
\end{equation}
for some universal constant $C>0$. To prove this result it is enough to get
\begin{equation*}
\sup_{\EE\in(0,1)}\frac{1}{1+\ln(1/ \EE)}\frac{1}{y+\EE}\leq \frac{C}{y(1+|\ln y|)}\cdot
\end{equation*}
Indeed, we shall  consider the two cases $\EE\geq\sqrt{y}$ and  $\EE\leq\sqrt{y}$. In the first case we observe 
$$
\frac{1}{y+\EE}\le\frac{1}{\sqrt{y}}\quad\hbox{and}\quad \frac{1}{1+\ln(1/ \EE)}\leq1
$$
which implies that
\begin{eqnarray*}
\frac{1}{1+\ln(1/ \EE)}\frac{1}{y+\EE}&\le&\frac{1}{\sqrt{y}}\\
&\leq& \frac{C}{y(1+|\ln y|)}.
\end{eqnarray*}
However in the second case  $\EE\leq\sqrt{y}$ we write simply that
$$
\frac{1}{y+\EE}\le\frac{1}{{y}}\quad\hbox{and}\quad \frac{1}{1+\ln(1/ \EE)}\leq \frac{1}{1+\frac12\ln(1/y)}
$$
which gives the desired result. Coming back to \eqref{TX02} and using \eqref{gkm11} we deduce that
\begin{eqnarray}\label{Eq11}
\nonumber\sup_{x}\big|f^\prime(x) T_{f,\textnormal{int}}^{ r,1}\,g(x)\big|&\le& C{\gamma(f)}\bigintsss_0^r\sup_{x}\frac{\omega_g( y)}{\big(1+\ln_+(\frac{1}{f(x)})\big)(y+f(x))} dy\\
&\le& C\gamma(f)
\Bigg(\bigintsss_0^r\frac{ \omega_g( y)}{y(1+|\ln y|)} dy+\bigintsss_0^r\frac{ \omega_g( y)}{1+y} dy\Bigg).
\end{eqnarray}
Consequently
\begin{eqnarray*}
\sup_{|{x}_1-{x}_2|\le  r} \Big|f^\prime({x}_1)T_{f,\textnormal{int}}^{ r,1}\,g({x}_1)-f^\prime({x}_2)T_{f,\textnormal{int}}^{ r,1}\,g({x}_2)\Big|&\leq & C\gamma(f)
\Bigg(\bigintsss_0^r\frac{ \omega_g( y)}{y(1+|\ln y|)} dy+\bigintsss_0^r{ \omega_g( y)} dy\Bigg).
\end{eqnarray*}
Therefore we get by  using Fubini's theorem
\begin{eqnarray*}
\bigintsss_0^1\sup_{|{x}_1-{x}_2|\le  r}\Big|f^\prime({x}_1)T_{f,\textnormal{int}}^{ r,1}\,g({x}_1)-f^\prime({x}_2)T_{f,\textnormal{int}}^{ r,1}\,g({x}_2)\Big|\frac{d r}{ r}&\leq & C\gamma(f)\bigintsss_0^1
\frac{ \omega_g( y)}{y}\frac{|\ln y|}{(1+|\ln y|)} dy\\
&+&C\gamma(f)\bigintsss_0^1
|\ln y|{ \omega_g( y)} dy\\
&\le&C\gamma(f)\|g\|_{D}.
\end{eqnarray*}

As to $T_{f,\textnormal{ext}}^{ r,1}g$ we write
\begin{eqnarray}\label{June7}
\nonumber f^\prime(x_1)T_{f,\textnormal{ext}}^{ r,1}g(x_1)- f^\prime(x_2)T_{f,\textnormal{ext}}^{ r,1}g(x_2)&=&\big(f^\prime(x_1)-f^\prime(x_2)\big)T_{f,\textnormal{ext}}^{ r,1}g(x_2)\\
\nonumber &+&f^\prime(x_1)\Big(T_{f,\textnormal{ext}}^{ r,1}g(x_1)-T_{f,\textnormal{ext}}^{ r,1}g(x_2)\Big)\\
&\triangleq&\mu_1(x_1,x_2)+\mu_2(x_1,x_2).
\end{eqnarray}
Our current goal is to prove that for $j\in\{1,2\}$
$$
\bigintsss_0^1\sup_{|x_1-x_2\leq  r}\frac{\mu_j(x_1,x_2)}{ r} d r
$$
is well-estimated.
For the first term we use  \eqref{HK1} leading to 
\begin{eqnarray*}
\bigintsss_0^1\sup_{|x_1-x_2\leq r}\frac{\mu_1(x_1,x_2)}{r} dr&\leq&\|T_{f,\textnormal{ext}}^{ r,1}g\|_{L^\infty}\bigintsss_0^1\frac{\omega_{f^\prime}(r)}{r} dr\\
&\le& C\|g\|_{D}\|f^\prime\|_{D}.
\end{eqnarray*}

The second term is more subtle.  First note that if $|x_1-x_2|\leq1$ then the quantity  $f^\prime(x_1)T_{f,\textnormal{ext}}^{r,1}g(x_1)-f^\prime(x_2)T_{f,\textnormal{ext}}^{r,1} g(x_2)$ vanishes for $x_1, x_2$ outside a compact set related only to the support of $f$. Therefore  the integrals defining $\mu_2(x_1,x_2)$ may be restricted to  the set $\{\beta r \le\beta y\le B\}$ with $B$ being some  constant related to the size of the supports of $f$ and $g$, and without loss of generality we can take $B=1$. It follows that
\begin{eqnarray}\label{T002}
\nonumber \mu_2(x_1,x_2)&=&f^\prime(x_1)\bigintsss_{\{\beta  r \le\beta y\leq 1\}}\frac{y \, \big[\widehat{g}(x_1,y)-\widehat{g}(x_2,y)\big]}{y^2+[f(x_1)+f(x_1+y)]^2} dy\\
\nonumber&+&f^\prime(x_1)\bigintsss_{\{\beta  r \le\beta y\leq 1\}}\frac{y \, \widehat{g}(x_2,y)\big[ \Delta ^+_yf(x_2)-\Delta^+_yf(x_1)\big]\big[\Delta^+_yf(x_2)+\Delta^+_yf(x_1\big]}{\big(y^2+[\Delta^+_yf(x_1)]^2\big)\big(y^2+[\Delta^+_yf(x_2)]^2\big)} dy\\
&\triangleq&\mu_{2,1}(x_1,x_2)+\mu_{2,2}(x_1,x_2),
\end{eqnarray}
with
$$
\widehat{g}(x,y)\triangleq g(\alpha x+\beta y)-g(\alpha x-\beta y)\quad\hbox{and}\quad  \Delta^+_yf(x)= f(x+y)+f(x).
$$
To estimate $\mu_{2,1}$ we shall use the next inequality which is  a consequence of Lemma \ref{lem1},
\begin{eqnarray}\label{T4}
\nonumber \int_0^L\frac{|f^\prime(x)|}{y+f(x)} dy&=& |f^\prime(x)|\ln\Big(1+\frac{L}{f(x)}\Big)\\
&\le&C\gamma(f)\big(1+\ln_+L)
\end{eqnarray}
with $C$ an absolute constant. This implies that
\begin{eqnarray*}
\mu_{2,1}(x_1,x_2)&\le&C\omega_g(\alpha|x_1-x_2|)|f^\prime(x_1)|\int_{0}^{ 1/\beta}\frac{1}{y+f(x_1)} dy\\
&\le&C\omega_g(|x_1-x_2|)\gamma(f)(1+|\ln \beta|).
\end{eqnarray*}
Consequently, we find that
\begin{eqnarray*}
\sup_{|x_1-x_2|\leq  r}|\mu_{2,1}(x_1,x_2)|&\le&C\omega_g( r)\gamma(f)\big(1+|\ln \beta|\big)\end{eqnarray*}
and therefore
\begin{eqnarray*}
\bigintsss_0^1\sup_{|x_1-x_2|\leq  r}|\mu_{2,1}(x_1,x_2)|\frac{d r}{ r}&\le&C\gamma(f) \big(1+|\ln \beta|\big) \|g\|_{D}.
\end{eqnarray*}
We emphasize that for $\beta=0$ one can still get an estimate since $\mu_{2,1}(x_1,x_2)=0$ and therefore we get the desired estimate.\\
Now we shall move to the estimate of  $\mu_{2,2}(x_1,x_2)$. We start with  using  the estimate 
$$
\sup_{a>0}{\frac{a}{y^2+a^2}}\leq \frac{C}{|y|}
$$
which implies that
\begin{eqnarray*}
\frac{y \, |\widehat{g}(x_2,y)|\big|\Delta^+_yf(x_2)-\Delta^+_yf(x_1)\big|\big|\Delta^+_yf(x_2)+\Delta^+_yf(x_1\big|}{\big(y^2+[\Delta^+_yf(x_1)]^2\big)\big(y^2+[\Delta^+_yf(x_2)]^2\big)}&\le& C|x_2-x_1|\|f^\prime\|_{L^\infty}\frac{\omega_g( 2\beta y)}{y^2}.
\end{eqnarray*}
Thus 
\begin{eqnarray*}
\sup_{|x_1-x_2|\leq  r}\mu_{2,2}(x_1,x_2)&\le& C r\|f^\prime\|_{L^\infty}^2 \int_{r}^{\frac1\beta}\frac{\omega_g( 2\beta y)}{y^2}\end{eqnarray*}
which yields in view of  Fubini's theorem
\begin{eqnarray*}
\bigintsss_0^1\sup_{|x_1-x_2|\leq  r}\mu_{2,2}(x_1,x_2)\frac{d r}{ r}&\le& C \|f^\prime\|_{L^\infty}^2\bigintsss_0^1 \bigintsss_{\{\beta r\le \beta y\le 1\}}\frac{\omega_g( 2\beta y)}{y^2}dyd r\\
&\le& C \|f^\prime\|_{L^\infty}^2 \bigintsss_{\{0\le \beta y\le 1\}}\frac{\omega_g( 2\beta y)}{y}dy\\
&\le&C \|f^\prime\|_{L^\infty}^2 \bigintsss_{0}^2\frac{\omega_g(  y)}{y}dy\\
&\le&C \|f^\prime\|_{L^\infty}^2\|g\|_{D}.
\end{eqnarray*}
Remark that the last constant does not depend on $\beta$. Putting together the preceding estimates we find that
\begin{equation}\label{June03}
\|f^\prime T_f^1g\|_{D}\le C\|g\|_{D} \Big[\big(1+|\ln\beta|\big)\gamma(f)+\|f^\prime\|_{L^\infty}^2\Big],
\end{equation}
where $\gamma(f)$ has been defined in \eqref{gammaZ}
As pointed before  the case  $\beta=0$ has a  special structure and one gets
\begin{equation*}
\|f^\prime T_f^1g\|_{D}\le C\|g\|_{D} \Big[ \gamma(f)+\|f^\prime\|_{L^\infty}^2\Big].
\end{equation*}
Now let us move  to the estimate of $f^\prime(x)T_f^{2,2}g$ given by
\begin{eqnarray}\label{Judf}
\nonumber T_f^{2,2}g(x)&=&\bigintsss_{0}^{+\infty}\frac{y \,g(\alpha x-\beta y)\big[f(x-y)-f(x+y)\big]\psi(x,y)}{\Big(y^2+[f(x)+f(x+y)]^2\Big) \Big(y^2+[f(x)+f(x-y)]^2\Big)} dy\\
&=&T_{f,\textnormal{int}}^{r,2,2}g(x)+T_{f,\textnormal{ext}}^{r,2,2}g(x)
\end{eqnarray}
where
\begin{eqnarray*}
\psi(x,y)&=&y\int_0^1\big[f^\prime(x+\theta y)-f^\prime(x-\theta y)\big] d\theta.
\end{eqnarray*}
and the cut-off operators are given by
\begin{eqnarray*}
T_{f,\textnormal{int}}^{r,2,2}g(x)&\triangleq&\bigintsss_{0}^r\frac{y \,g(\alpha x-\beta y)\big[f(x-y)-f(x+y)\big]\psi(x,y)}{\Big(y^2+[\Delta^+_yf(x)]^2\Big) \Big(y^2+[\Delta^+_{-y}f(x)]^2\Big)} dy
\end{eqnarray*}
and
\begin{eqnarray}\label{T03}
\nonumber T_{f,\textnormal{ext}}^{r,2,2}g(x)&=&\bigintsss_{ r}^1\frac{y \,g(\alpha x-\beta y)\big[f(x-y)-f(x+y)\big]\psi(x,y)}{\Big(y^2+[\Delta^+_yf(x)]^2\Big) \Big(y^2+[\Delta^+_{-y}f(x)]^2\Big)} dy\\
&\triangleq&\bigintsss_{ r}^1\frac{\mathcal{N}(x,y)}{\mathcal{D}(x,y)} dy.
\end{eqnarray}
We shall proceed in a similar way to $T_f^1g$.
Let us start  with $f^\prime(x)T_{f,\textnormal{int}}^{r,2,2}g$. Since 
\begin{equation}\label{June1}
|\psi(x,y)|\leq 2 y\omega_{f^\prime}(y)
\end{equation}
then 
one has
\begin{eqnarray*}
\big|f^\prime(x)T_{f,\textnormal{int}}^{r,2,2}g(x)\big|&\le &C\|g\|_{L^\infty}\|f^\prime\|_{L^\infty}|f^\prime(x)|\bigintsss_{ 0}^r\frac{y^3\omega_{f^\prime}(y) }{\Big(y^2+[f(x)]^2\Big)^2 } dy\\
&\le &C\|g\|_{L^\infty}\|f^\prime\|_{L^\infty}|f^\prime(x)\bigintsss_{ 0}^r\frac{\omega_{f^\prime}(y) }{y+f(x) } dy.
\end{eqnarray*}
Thus following the same steps of  \eqref{Eq11} we obtain
\begin{eqnarray*}
\sup_{|{x}_1-{x}_2|\le  r} \Big|f^\prime({x}_1)T_{f,\textnormal{int}}^{ r,2,2}\,g({x}_1)-f^\prime({x}_2)T_{f,\textnormal{int}}^{ r,2,2}\,g({x}_2)\Big|&\leq & C\|g\|_{L^\infty}\|f^\prime\|_{L^\infty}\gamma(f)
\bigintsss_{0}^r\frac{ \omega_{f^\prime}( y)}{y(1+|\ln y|)} dy\\
&+ & C\|g\|_{L^\infty}\|f^\prime\|_{L^\infty}\gamma(f)
\bigintsss_{0}^r| \omega_{f^\prime}( y)dy.
\end{eqnarray*}
Thus by Fubini's theorem and \eqref{Imbed1}  we obtain
\begin{eqnarray*}
\int_0^1\sup_{|{x}_1-{x}_2|\le  r}\Big|f^\prime({x}_1)T_{f,\textnormal{int}}^{ r,2,2}\,g({x}_1)-f^\prime({x}_2)T_{f,\textnormal{int}}^{ r,2,2}\,g({x}_2)\Big|\frac{d r}{ r}
&\le&C\|g\|_{L^\infty}\|f^\prime\|_{D}^2\gamma(f).
\end{eqnarray*}
What is left is  to  estimate the quantity $f^\prime(x)T_{f,\textnormal{ext}}^{r,2,2}g$. First, it is obvious that
 \begin{eqnarray}\label{aout1}
\nonumber f^\prime(x_1)T_{f,\textnormal{ext}}^{ r,2,2}g(x_1)- f^\prime(x_2)T_{f,\textnormal{ext}}^{ r,2,2}g(x_2)&=&\big(f^\prime(x_1)-f^\prime(x_2)\big)T_{f,\textnormal{ext}}^{ r,2,2}(x_2)\\
&+&f^\prime(x_1)\Big(T_{f,\textnormal{ext}}^{ r,2,2}(x_1)-T_{f,\textnormal{ext}}^{ r,2,2}(x_2)\Big).
\end{eqnarray}
The first term of the right-hand side is easy to estimate. Indeed,
\begin{eqnarray*}
|\big(f^\prime(x_1)-f^\prime(x_2)\big)T_{f,\textnormal{ext}}^{ r,2,2}(x_2)|\le \omega_{f^\prime}(|x_1-x_2|) \|T_{f,\textnormal{ext}}^{ r,2,2}\|_{L^\infty}.
\end{eqnarray*}
It is  clear  that 
\begin{eqnarray*}
|T_{f,\textnormal{ext}}^{r,2,2}g(x)|&\le&C\|g\|_{L^\infty}\|f^\prime\|_{L^\infty}\bigintsss_{ r}^1\frac{\omega_{f^\prime}(y)}{y} dy\\
&\leq & C\|g\|_{L^\infty}\|f^\prime\|_{D}^2.
\end{eqnarray*}
Hence 
\begin{eqnarray*}
\int_0^1\sup_{|{x}_1-{x}_2|\le  r}|\big(f^\prime(x_1)-f^\prime(x_2)\big)T_{f,\textnormal{ext}}^{ r,2,2}(x_2)|\frac{d r}{ r}&\leq & C\|g\|_{L^\infty}\|f^\prime\|_{D}^3.
\end{eqnarray*}
To deal with the second term  we proceed as for the term $\mu_{2}(x_1,x_2)$  in  \eqref{T002}. From \eqref{T03} one has
 \begin{eqnarray}\label{aout2}
\nonumber f^\prime(x_1)\Big(T_{f,\textnormal{ext}}^{ r,2,2}(x_1)-T_{f,\textnormal{ext}}^{ r,2,2}(x_2)\Big)&=& f^\prime(x_1)\bigintsss_{ r}^1\frac{\mathcal{N}(x_1,y)-\mathcal{N}(x_2,y)}{\mathcal{D}(x_1,y)} dy\\
&+&f^\prime(x_1)\bigintsss_r^1\frac{\mathcal{N}(x_2,y)\big(\mathcal{D}(x_2,y)-\mathcal{D}(x_1,y)\big)}{\mathcal{D}(x_1,y)\mathcal{D}(x_2,y)} dy.
\end{eqnarray}
It is quite obvious from some straightforward computations using in particular \eqref{June1}  that for $|x_1-x_2|\leq r$
$$
\big|\mathcal{N}(x_1,y)-\mathcal{N}(x_2,y)\big|\leq C\|f^\prime\|_{L^\infty}y^2\Big[\omega_g(\alpha r) \omega_{f^\prime}(y) y+\|g\|_{L^\infty} \omega_{f^\prime}(r)y+\|g\|_{L^\infty} r\, \omega_{f^\prime}(y) \Big].
$$
Since 
$$
\frac{1}{\mathcal{D}(x,y)}\leq\frac{C}{[y+f(x)|^4}\leq  \frac{C}{y^4}
$$
then we get
 \begin{eqnarray*}
\frac{\big|\mathcal{N}(x_1,y)-\mathcal{N}(x_2,y)\big|}{\mathcal{D}(x_1,y)}\leq C\|f^\prime\|_{L^\infty}\Bigg[\omega_g(\alpha r)\frac{ \omega_{f^\prime}(y)}{ y}+\|g\|_{L^\infty} \frac{\omega_{f^\prime}(r)}{y+f(x_1)}+\|g\|_{L^\infty} r\, \frac{\omega_{f^\prime}(y)}{y^2} \Bigg].
\end{eqnarray*}
This leads in view of \eqref{Imbed1},
 \begin{eqnarray}\label{aout3}
\nonumber  |f^\prime(x_1)|\bigintsss_{ r}^1\frac{\big|\mathcal{N}(x_1,y)-\mathcal{N}(x_2,y)\big|}{\mathcal{D}(x_1,y)} dy&\leq& C\|f^\prime\|_{D}\Bigg[ \|f^\prime\|_{D}^2\omega_g(\alpha r)+\|g\|_{D} \omega_{f^\prime}(r) \int_0^{1}\frac{|f^\prime(x_1)|}{y+f(x_1)} dy\Bigg]\\
 &+&\|g\|_{D} \|f^\prime\|_{D}^2\,r\,\int_{r}^1 \frac{\omega_{f^\prime}(y)}{y^2}dy 
 \end{eqnarray}
 which implies according to \eqref{T4}
\begin{eqnarray*}
\bigintsss_0^1\sup_{|x_1-x_2|\leq r}|f^\prime(x_1)|\bigintsss_{r}^1\frac{\big|\mathcal{N}(x_1,y)-\mathcal{N}(x_2,y)\big|}{\mathcal{D}(x_1,y)} dy\frac{dr}{r}&\leq& C \Big(\|f^\prime\|_{D}^3+\|f^\prime\|_{D}^2\gamma(f)\Big)\|g\|_{D}.
 \end{eqnarray*}
Now straightforward computations show that
\begin{equation}\label{Jun8}
\frac{\big|\mathcal{N}(x_2,y)\big(\mathcal{D}(x_2,y)-\mathcal{D}(x_1,y)\big)\big|}{\mathcal{D}(x_1,y)\mathcal{D}(x_2,y)} \leq C\|g\|_{L^\infty}\|f^\prime\|_{L^\infty}^2|x_1-x_2|\frac{\omega_{f^\prime}(y)}{y^2}\cdot
\end{equation}
Therefore using Fubini's theorem we get
\begin{eqnarray*}
\bigintsss_0^1\sup_{|x_1-x_2|\leq r}|f^\prime(x_1)|\bigintsss_{r}^1\frac{\big|\mathcal{N}(x_2,y)\big(\mathcal{D}(x_2,y)-\mathcal{D}(x_1,y)\big)\big|}{\mathcal{D}(x_1,y)\mathcal{D}(x_2,y)} dy\frac{dr}{r}&\leq& C \|f^\prime\|_{D}^4\|g\|_{D}.
 \end{eqnarray*}
Putting together the preceding estimates we find that
\begin{eqnarray}\label{June2}
\nonumber\|f^\prime T_f^{2,2}g\|_{D}&\leq &C\|g\|_{D}\Big(\|f^\prime\|_{D}^2+\|f^\prime\|_{D}^2\gamma(f)+\|f^\prime\|_{D}^3\Big)\\
&\le& C\|g\|_{D}\Big(\|f^\prime\|_{D}^2+\|f^\prime\|_{D}^4\Big)
 \end{eqnarray}
with $C$ a constant depending only on the diameter of the compact $K.$ To get the desired estimate it suffices to  put together \eqref{EqwZZ23}, \eqref{June03} and \eqref{June2}.
\vspace{0,3cm}

${\bf{3)}}$ We shall proceed as in the proof of the  point ${\bf{2)}}$ of the   Theorem \ref{th1}. We use exactly the same splitting with similar estimates and to avoid redundancy we shall only give the basic estimates with some details for the terms that require new treatment. We use the decomposition described in  \eqref{splitjun1}. To estimate $T_f^{2,1}g$ in $C^s$ we use the expression \eqref{xident1}. Then following the same lines using in particular the law product \eqref{lawX3} and the composition law \eqref{comp1}, one has 
\begin{eqnarray*}
\|\mathcal{N}_1(\cdot,y)\|_{s}
&\le&C \|g\|_{s}\Big(\alpha^s +\beta^s \|f^\prime\|_{L^\infty}^s y^s\Big)\|f^\prime\|_{L^\infty}\\
&+&C\|g\|_{L^\infty}\|f^\prime\|_{s}\int_0^1\Big(1+\theta^s \|f^\prime\|_{L^\infty}^sy^s\Big)d\theta.
\end{eqnarray*}
Since  $\alpha,\beta\in[0,1]$ we deduce
\begin{equation*}
\|\mathcal{N}_1(\cdot,y)\|_{s}
\le C \big(\|g\|_{s} \|f^\prime\|_{L^\infty}+\|g\|_{L^\infty} \|f^\prime\|_{s}\big)\Big(1 + \|f^\prime\|_{L^\infty}^s y^s\Big).
\end{equation*}
Similarly we get
 \begin{eqnarray*}
\|\varphi(\cdot,\pm y)\|_{s}&\leq& C\big(1+y\|f^\prime\|_{L^\infty}\big) y\int_0^1\|f^\prime\circ(\hbox{Id}\pm \theta y f )\|_{s} d\theta \\
&\le& C\big(y+y^2\|f^\prime\|_{L^\infty}\big)\|f^\prime\|_{s}\Big(1+\|f^\prime\|_{L^\infty}^s y^s\Big).
 \end{eqnarray*}
This implies
$$
\|\mathcal{D}_1(\cdot,y)\|_{s}\leq C\big(1+y^{4+s}\big)\big(1+\|f^\prime\|_{L^\infty}^{3+s}\big)\|f^\prime\|_s
$$
and
$$
\|1/\mathcal{D}_1(\cdot,y)\|_{s}\leq \frac{C}{1+y^{4-s}}\big(1+\|f^\prime\|_{L^\infty}^{11+s}\big)\|f^\prime\|_s.
$$
Consequently for $s\in (0,1)$
 \begin{eqnarray*}
\|(\mathcal{N}_1/\mathcal{D}_1)(\cdot,y)\|_{s}&\leq& \|(\mathcal{N}_1(\cdot,y)\|_{L^\infty}\|1/\mathcal{D}_1)(\cdot,y)\|_{s}+\|\mathcal{N}_1(\cdot,y)\|_{s}\|1/\mathcal{D}_1(\cdot,y)\|_{L^\infty}\\ 
&\leq&
 \frac{C}{1+y^{4-s}}\big(1+\|f^\prime\|_{L^\infty}^{11+s}\big)\|f^\prime\|_s\|g\|_s.
  \end{eqnarray*}
Therefore we get similarly to \eqref{June4w}
\begin{eqnarray*}
\nonumber\|T_f^{2,1}g\|_s&\leq& C\big(1+\|f^\prime\|_{L^\infty}^{11+s}\big)\|f^\prime\|_s\|g\|_s\int_0^{+\infty} \frac{y^2}{1+y^{4-s}} ds\\
&\leq&C\big(1+\|f^\prime\|_{L^\infty}^{11+s}\big)\|f^\prime\|_s\|g\|_s.
\end{eqnarray*}
Combining law products with Sobolev embeddings and \eqref{embedd11}
we get
\begin{eqnarray*}
\nonumber\| f^\prime T_f^{2,1}g\|_s&\leq&\| f^\prime\|_{L^\infty}\| T_f^{2,1}g\|_s+\| f^\prime\|_s\| T_f^{2,1}g\|_{L^\infty}\\
\nonumber&\leq& C\big(1+\|f^\prime\|_{L^\infty}^{11+s}\big)\|f^\prime\|_s\|f^\prime\|_{L^\infty}\|g\|_s+\|g\|_{L^\infty}\|f^\prime\|_D\|f^\prime\|_s\\
&\le& C\big(1+\|f^\prime\|_{L^\infty}^{11+s}\big)\|f^\prime\|_s\|f^\prime\|_{D}\|g\|_s.
\end{eqnarray*}
Using once again Sobolev embeddings we get
\begin{eqnarray}\label{jul11}
\| f^\prime T_f^{2,1}g\|_s&\le& C\big(\|f^\prime\|_s+\|f^\prime\|_{s}^{13}\big)\|f^\prime\|_{D}\|g\|_s.
\end{eqnarray}

Now to estimate $T_f^1g$ we come back to the decomposition \eqref{jul34} and we easily get 
\begin{eqnarray*}
\|T_{f,\textnormal{int}}^{r,1}g\|_{L^\infty}&\leq& C \|g\|_{s}\int_{0}^r {y^{-1+s}}dy\\
&\leq& \|g\|_{s} r^{s}.
\end{eqnarray*}
Hence we obtain, since $r=|x_1-x_2|$,
\begin{equation*}
|T_{f,\textnormal{int}}^{r,1}g(x_1)-T_{f,\textnormal{int}}^{r,1}g(x_2)|\leq C \|g\|_{s}|x_1-x_2|^s.
\end{equation*}
and we also get
\begin{equation*}
|f^\prime(x_1)T_{f,\textnormal{int}}^{r,1}g(x_1)-f^\prime(x_2)T_{f,\textnormal{int}}^{r,1}g(x_2)|\leq C\|f^\prime\|_{L^\infty} \|g\|_{s}|x_1-x_2|^s.
\end{equation*}
To estimate  the term $f^\prime T_{f,\textnormal{ext}}^{r,1}g$ we come back to \eqref{June7} and \eqref{T002} and following the same estimates one gets
\begin{eqnarray*}
|\mu_1(x_1, x_2)|&\leq &|x_1-x_2|^s \|f^\prime\|_s \|T_{f,\textnormal{ext}}^{ r,1}g\|_{L^\infty}\\
&\le& C|x_1-x_2|^s \|f^\prime\|_s \|g\|_{D}.
\end{eqnarray*}
Moreover
$$
|\mu_2(x_1, x_2)|\leq |\mu_{2,1}(x_1, x_2)|+|\mu_{2,2}(x_1, x_2)|
$$

and 
\begin{eqnarray*}
|\mu_{2,2}(x_1, x_2)|&\leq& C|x_2-x_1|\|f^\prime\|_{L^\infty}^2\|g\|_s\bigintsss_{\{ \beta r \leq \beta y\leq 1\}}(\beta y)^{s} y^{-2} dy\\
&\le&C \|f^\prime\|_{L^\infty}^2\|g\|_s |x_1-x_2|^s.
\end{eqnarray*}
To deal with the term $\mu_{2,1}(x_1,x_2)$ in \eqref{T002}  one obtains in view of \eqref{T4}
\begin{eqnarray*}
|\mu_{2,1}(x_1,x_2)|&\leq& |x_1-x_2|^s\|g\|_s|f^\prime(x_1)|\bigintsss_{\{\beta r\leq \beta y\leq1\}}\frac{y}{y^2+f^2(x_1)}dy\\
&\le& |x_1-x_2|^s\|g\|_s|f^\prime(x_1)|\bigintsss_{0}^{\frac1\beta}\frac{1}{y+f(x_1)}dy.
\end{eqnarray*}
Using the second part of  Lemma \ref{lem1} one finds for $s^\prime\in (0,s]$
\begin{eqnarray*}
|f^\prime(x_1)|\bigintsss_{0}^{\frac1\beta}\frac{1}{y+f(x_1)}dy&\le& C\|f^\prime\|_{s^\prime}^{\frac{1}{1+s^\prime}}|f(x_1)|^{\frac{s^\prime}{1+s^\prime}}\bigintsss_0^{\frac1\beta}\frac{1}{y+f(x_1)}dy.
\end{eqnarray*}
Combining this inequality with
$$
\sup_{a>0}\frac{ a^{\frac{s^\prime}{1+s^\prime}}}{y+a}\leq C y^{-\frac{1}{1+s^\prime}}
$$
we get 
\begin{equation}\label{interp1}
\sup_{x_1\in\R}|f^\prime(x_1)|\bigintsss_{0}^{\frac1\beta}\frac{1}{y+f(x_1)}dy\le C\|f^\prime\|_{s^\prime}^{\frac{1}{1+s^\prime}} \beta^{-\frac{s^\prime}{1+s^\prime}}
\end{equation}
and therefore
\begin{eqnarray*}
|\mu_{2,1}(x_1,x_2)|
&\le& |x_1-x_2|^s\|g\|_s\|f^\prime\|_{s^\prime}^{\frac{1}{1+s^\prime}}\beta^{-\frac{s^\prime}{1+s^\prime}}.
\end{eqnarray*}
Hence
\begin{eqnarray*}
|f^\prime(x_1)T_{f,\textnormal{ext}}^{r,1}g(x_1)-f^\prime(x_2)T_{f,\textnormal{ext}}^{r,1}g(x_2)|&\leq& C\|g\|_{D}\|f^\prime\|_s |x_1-x_2|^s+C \|f^\prime\|_{L^\infty}^2\|g\|_s |x_1-x_2|^s\\
&+&C |x_1-x_2|^s\|g\|_s| \|f^\prime\|_{s^\prime}^{\frac{1}{1+s^\prime}}\beta^{-\frac{s^\prime}{1+s^\prime}}.
\end{eqnarray*}
It follows that
\begin{equation}\label{titt2}
\|f^\prime T_f^1g\|_{s}\leq C \|g\|_s\Big[\|f^\prime\|_{s^\prime}^{\frac{1}{1+s^\prime}}\beta^{-\frac{s^\prime}{1+s^\prime}}+\|f^\prime\|_{L^\infty}^2\Big]+C\|g\|_{D}\|f^\prime\|_s.
\end{equation}
It remains to estimate $f^\prime T_f^{2,2}g$ described  in \eqref{Judf} and \eqref{T03}.  First one may write
\begin{eqnarray*}
|T_{f,\textnormal{int}}^{r,2,2}g(x)|&\leq& C\|g\|_{L^\infty}\|f^\prime\|_{L^\infty}\|f^\prime\|_s\int_0^r y^{s-1}dy\\
&\leq&C\|g\|_{L^\infty}\|f^\prime\|_{L^\infty}\|f^\prime\|_s|x_1-x_2|^s.
\end{eqnarray*}
Therefore
\begin{equation*}
|f^\prime(x_1)T_{f,\textnormal{int}}^{r,2,2}g(x_1)-f^\prime(x_2)T_{f,\textnormal{int}}^{r,2,2}g(x_2)|\leq C \|g\|_{L^\infty}\|f^\prime\|_{L^\infty}^2\|f^\prime\|_s|x_1-x_2|^s.
\end{equation*}
By Sobolev embeddings we get
\begin{equation}\label{Eqjune4}
|f^\prime(x_1)T_{f,\textnormal{int}}^{r,2,2}g(x_1)-f^\prime(x_2)T_{f,\textnormal{int}}^{r,2,2}g(x_2)|\leq C \|g\|_{s}\|f^\prime\|_{L^\infty}\|f^\prime\|_s ^2|x_1-x_2|^s.
\end{equation}
From \eqref{aout1} and the analysis following this identity one has
\begin{eqnarray*}
|\big(f^\prime(x_1)-f^\prime(x_2)\big)T_{f,\textnormal{ext}}^{ r,2,2}(x_2)|&\le& \|f^\prime\|_s  \|T_{f,\textnormal{ext}}^{ r,2,2}g\|_{L^\infty} |x_1-x_2|^s\\
&\leq&C \|g\|_{L^\infty} \|f^\prime\|_s^2 \|f^\prime\|_{L^\infty}|x_1-x_2|^s
\end{eqnarray*}

 Using \eqref{aout2}, \eqref{aout3} and \eqref{interp1} (with $s^\prime=s$) combined with Sobolev embeddings one deduces
\begin{eqnarray*}
 |f^\prime(x_1)|\bigintsss_{ r}^1\frac{\big|\mathcal{N}(x_1,y)-\mathcal{N}(x_2,y)\big|}{\mathcal{D}(x_1,y)} dy&\leq& C\|f^\prime\|_{L^\infty}\|g\|_s\Big( \|f^\prime\|_{s}^2+ \|f^\prime\|_{s}\Big).\end{eqnarray*}
From \eqref{Jun8} we get
\begin{eqnarray*}
\frac{\big|\mathcal{N}(x_2,y)\big(\mathcal{D}(x_2,y)-\mathcal{D}(x_1,y)\big)\big|}{\mathcal{D}(x_1,y)\mathcal{D}(x_2,y)} \leq C\|g\|_{L^\infty}\|f^\prime\|_{L^\infty}\|f^\prime\|_s|x_1-x_2| y^{s-2}.
\end{eqnarray*}
Therefore  we get
\begin{eqnarray*}
|f^\prime(x_1)|\bigintsss_{ r}^1\frac{\big|\mathcal{N}(x_2,y)\big(\mathcal{D}(x_2,y)-\mathcal{D}(x_1,y)\big)\big|}{\mathcal{D}(x_1,y)\mathcal{D}(x_2,y)} dy&\leq&  C\|g\|_{L^\infty}\|f^\prime\|_{L^\infty}^2\|f^\prime\|_s|x_1-x_2|^s.
 \end{eqnarray*}
Hence plugging the preceding estimates into   \eqref{aout1} and \eqref{aout2}, we find
\begin{eqnarray*}
 \Big|f^\prime(x_1)T_{f,\textnormal{ext}}^{ r,2,2}(x_1)-|f^\prime(x_2)T_{f,\textnormal{ext}}^{ r,2,2}(x_2)\Big|&\leq &C \|g\|_{L^\infty} \|f^\prime\|_s^2 \|f^\prime\|_{L^\infty}|x_1-x_2|^s\\
&+&C\|f^\prime\|_{L^\infty}\|g\|_s\Big( \|f^\prime\|_{s}^2+ \|f^\prime\|_{s}\Big) |x_1-x_2|^s\\
&+&C\|g\|_{L^\infty}\|f^\prime\|_{L^\infty}^2\|f^\prime\|_s|x_1-x_2|^s.
\end{eqnarray*}
Using Standard  embeddings we get
\begin{equation}\label{Eqjune1}
 \Big|f^\prime(x_1)T_{f,\textnormal{ext}}^{ r,2,2}(x_1)-|f^\prime(x_2)T_{f,\textnormal{ext}}^{ r,2,2}(x_2)\Big|\leq C \|g\|_{s} \|f^\prime\|_{L^\infty}|x_1-x_2|^s\Big[\|f^\prime\|_s+\|f^\prime\|_s^2\Big].
\end{equation}
Putting together \eqref{Eqjune4}, \eqref{Eqjune1} and  \eqref{Judf} we obtain
\begin{equation}\label{Eqjune5}
\|f^\prime T_f^{2,2}g\|_s\leq C \|g\|_{s} \|f^\prime\|_{L^\infty}\Big[\|f^\prime\|_s+\|f^\prime\|_s^2\Big].
\end{equation}
Combining \eqref{jul11}, \eqref{titt2} and \eqref{Eqjune5} we get for any $s^\prime\in (0,s]$
\begin{eqnarray*}
\|f^\prime T_fg\|_s&\leq &C \|g\|_{s} \|f^\prime\|_{D}\Big[\|f^\prime\|_s+\|f^\prime\|_s^{13}\Big]\\
&+&C\|g\|_s\|f^\prime\|_{s^\prime}^{\frac{1}{1+s^\prime}}\beta^{-\frac{s^\prime}{1+s^\prime}}+C\|g\|_{D}\|f^\prime\|_{s}.
\end{eqnarray*}
Now using the embedding $ C^s\hookrightarrow C^{s^\prime}\hookrightarrow  D$  we get
\begin{eqnarray*}
\|f^\prime T_fg\|_s&\leq &C \|g\|_{s}\Big[\beta^{-\frac{s}{1+s}} \|f^\prime\|_{s}^{\frac{1}{1+s}}+\|f^\prime\|_s^{14}\Big]\\
&\leq& C \|g\|_{s}\Big[\beta^{-\frac12} \|f^\prime\|_{s}^{\frac{1}{1+s}}+\|f^\prime\|_s^{14}\Big].\end{eqnarray*}
Another useful estimate that one can get from taking $s^\prime=s/2$ and using  some  interpolation  inequalities
$$
\|f^\prime\|_D\leq C\|f^\prime\|_{s\frac{1+s}{2+s}}\leq C\|f^\prime\|_{L^\infty}^{\frac{1}{2+s}}\|f^\prime\|_s^{\frac{1+s}{2+s}},\quad \| f^\prime\|_{\frac{s}{2}}\leq C\| f^\prime\|_{L^\infty}^{\frac12}\| f^\prime\|_s^{\frac12},\quad \beta^{-\frac{s}{2+s}}\leq \beta^{-\frac12}
$$
is the following
\begin{eqnarray*}
\|f^\prime T_fg\|_s&\leq &C \|g\|_{s} \|f^\prime\|_{L^\infty}^{\frac{1}{2+s}}\Big[\|f^\prime\|_s^{\frac{1}{2+s}}\beta^{-\frac12}+\|f^\prime\|_s^{14}\Big]\\
&+&C\|g\|_{L^\infty}^{\frac{1}{2+s}}\|g\|_s^{\frac{1+s}{2+s}}\|f^\prime\|_{s}.
\end{eqnarray*}
This achieves the proof of Theorem \ref{th1}.
\end{proof}

   \section{Local well-posedness proof}
   
   The main objective of this section is to prove the local well-posedness result stated in  the first part  of Theorem \ref{thm1}. The approach that we shall follow is classical and  will be done in  several steps. We start  with a priori estimates of smooth solutions  in suitable Banach spaces and this will be the main concern of the Sections \ref{Secdez1} and \ref{Secdez2}. The rigorous  construction of classical solutions will be conducted in Section \ref{Const778}.

 \subsection{Estimates of the source terms}\label{Secdez1}
 The main goal of this section is to establish the following a priori estimates for the source terms $F$ and $G$ described in \eqref{TataX1} and \eqref{TataX2}.
%
 \begin{proposition}\label{prop10}
 Let $K$ be a compact set of $\R$ and $s\in (0,1)$. We denote by $X$ one of the space $C_K^\star$ and $C^s_K.$ There exits a constant $C>0$ depending only on $K$ such that the  following estimates hold true
 \begin{enumerate}
 \item For any $f\in X$ we have
 $$
\|F\|_{L^\infty}\leq C\|f^\prime\|_{L^\infty}\|f^\prime\|_D,\quad \|F\|_X\leq C\|f^\prime\|_{D}\big(\|f^\prime\|_X+\|f^\prime\|_X^3\big) .$$
 \item For any $f\in X$ we have
 $$
\|G\|_{L^\infty}\leq C\|f^\prime\|_{L^\infty}\Big(1+\|f^\prime\|_{D}^3\Big),\quad  \|G\|_{X}\le C\big(1+\|f^\prime\|_{D}^{\frac{1}{3}}\big)\Big(\|f^\prime\|_X+\|f^\prime\|_X^{16}\Big).
 $$
 \end{enumerate}
 \end{proposition}
 \begin{proof}
 For simplicity we denote along this proof   the operator $\Delta_y^-$  by $\Delta_y.$
 
 ${\bf{(1)}}$
 The estimate of  $F$ in $L^\infty$ is quite easy. Indeed,  it is obvious according to \eqref{L1} that
  \begin{eqnarray*}
\|F\|_{L^\infty}&\le& C\|f^\prime\|_{L^\infty} \bigintsss_{-M}^{M}\sup_{x\in\R}\frac{|f^\prime(x+y)-f^\prime(x)| }{|y|}dy
\\
&\leq& C\|f^\prime\|_{L^\infty}\bigintsss_{-M}^{M}\frac{\omega_{f^\prime}(|y|) }{|y|}dy\\
&\le& C\|f^\prime\|_{L^\infty} \|f^\prime\|_{D}.
 \end{eqnarray*}
Now let us move to the estimate of $F$ in the function space $X$ which is   Dini space $C^\star_K$ or H\"{o}lder spaces $C^s_K$. For this purpose we shall transform  slightly $F$ in order to apply Proposition \ref{propCart1}. In fact from Taylor formula one can write
$$
F(x)=\bigintsss_{-M}^{M}\bigintsss_0^1\frac{y\,\Delta_{\theta y}f^\prime(x)\Delta_yf^\prime(x) }{y^2+(\Delta_yf(x))^2} dy d\theta.
$$
From the notation \eqref{BiCauchop} one has
$$
F(x)=\int_0^1\mathcal{C}_f^{\theta,\Re}(f^\prime, f^\prime)(x) d\theta.
$$
At this stage it suffices to apply  Proposition \ref{propCart1} which implies that

      $$
      \|F\|_{X}\le C\big(\|f^\prime\|_{D}\|f^\prime\|_X+\|f^\prime\|_{L^\infty}\|f^\prime\|_D\|f^\prime\|_X^2\big)
      $$
      which in turn gives the desired result according to  the  embedding $X\hookrightarrow L^\infty$.\\
      
            ${\bf{(2)}}$ The expression of $G$ is given in \eqref{TataX2} and for simplicity we shall assume along this part that $M=1. $ We shall  first   split $G$ as follows
      
      \begin{eqnarray}\label{daca1}
 \nonumber G(x)&=&\textnormal{p.v.}\bigintsss_{-1}^{1}\frac{\big[2f(x)+\Delta_y^-f(x)+y f^\prime(x)\big]\big(f^\prime(x+y)+ f^\prime(x)\big) }{y^2+(f(x+y)+f(x))^2} dy\\
\nonumber &=&2f(x)\textnormal{ p.v.}\bigintsss_{-1}^{1}\frac{f^\prime(x+y)+ f^\prime(x) }{y^2+(f(x+y)+f(x))^2} dy\\
 \nonumber&+&\textnormal{p.v.}\bigintsss_{-1}^{1}\frac{\big[\Delta_y^-f(x)+y f^\prime(x)\big]\big(2f^\prime(x)+\Delta_y^- f^\prime(x)\big) }{y^2+(f(x+y)+f(x))^2} dy\\
 &\triangleq&G_1(x)+G_2(x).
     \end{eqnarray}  
    The estimate  $G_1$ in $L^\infty$ is quite  easy. To see this we can first assume that $f(x)>0$, otherwise the integral is vanishing. Thus by change of variables we get
       \begin{eqnarray*}
       |G_1(x)|&\le& 4\|f^\prime\|_{L^\infty}\int_{-1}^{1}\frac{|f(x)|}{y^2+f^2(x)} dy\\
       &\leq &C\|f^\prime\|_{L^\infty}.
          \end{eqnarray*} 
     Note that for  $x\in \textnormal{supp} f$ we have $f(x+y)=0, \forall y\notin [-1,1].$  Thus 
          \begin{eqnarray}\label{daca2}
\nonumber G_1(x)&=&2f(x)\bigintsss_{\R}\frac{f^\prime(x+y)+ f^\prime(x) }{y^2+(f(x+y)+f(x))^2} dy-4f(x)f^\prime(x)  \bigintsss_{1}^{+\infty}\frac{1}{y^2+(f(x))^2} dy\\
\nonumber &=&2f(x)\bigintsss_{\R}\frac{f^\prime(x+y)+ f^\prime(x) }{y^2+(f(x+y)+f(x))^2} dy-4f^\prime(x)\arctan(f(x))\\
 &\triangleq&G_{11}+G_{12}.
     \end{eqnarray}  
The estimate of $G_{12}$ in $L^\infty$ is elementray
\begin{equation}\label{kar1}
\|G_{12}\|_{L^\infty}\leq 4\|f^\prime\|_{L^\infty}\|f\|_{L^\infty}.
\end{equation}
However, to estimate   $G_{12}$ in $X$  we use the law product \eqref{L1} leading to
            \begin{eqnarray*}
 \|f^\prime\arctan f\|_{X}\le \|\arctan f\|_{L^\infty} \|f^\prime\|_{X} + \|f^\prime\|_{L^\infty} \|\arctan f\|_{X}.
     \end{eqnarray*} 
     It is easy to   check from the mean value theorem that
     $$
    \|\arctan f\|_{L^\infty} \leq \|f\|_{L^\infty}\quad \textnormal{and}\quad  \omega_{\arctan f}(r)\le \omega_f(r)
     $$
     which implies  in view of the embedding $\textnormal{Lip}\hookrightarrow X$ that
     \begin{eqnarray*}
 \|\arctan f\|_{X}&\le& \|f\|_{X}\\
& \leq&C\|f^\prime\|_{L^\infty}.
     \end{eqnarray*} 
     Therefore we obtain from the classical  embeddings       \begin{eqnarray}\label{dgh1}
\nonumber \|G_{12}\|_{X}&\le& C \big(\|f\|_{L^\infty}\|f^\prime\|_{X}+C\|f^\prime\|_{L^\infty}^2 \big)\\
 &\le&  C \|f^\prime\|_{L^\infty}\|f^\prime\|_{X}.
     \end{eqnarray} 
          We shall now estimate the term $G_{11}$ in the spaces $X$. First we use Taylor formula
          $$
         f(x+y)+ f(x)=2f(x)+y\int_0^1f^\prime(x+\theta y)d\theta
          $$
          which implies after a change of variables $y=f(x)z$ (assuming that $f(x)>0$)
                 \begin{eqnarray}\label{Expbb}
         \nonumber G_{11}(x)&=& 2f(x)\bigintsss_{\R}\frac{f^\prime(x)+f^\prime(x+y) }{y^2+\big[2f(x)+y\int_0^1f^\prime(x+\theta y)d\theta\big]^2} dy\\
         &=&2\bigintsss_{\R}\frac{f^\prime(x)+f^\prime(x+f(x)z) }{   \varphi(x,z)} dz  
           \end{eqnarray}  
with 
       $$
       \varphi(x,z)=z^2+\Bigg(2+z \int_0^1f^\prime\big(x+\theta f(x) z\big)d\theta\Bigg)^2.
       $$
       Note that for $f(x)=0$ we have from the definition $G_{11}(x)=0$ which agrees with the  expression \eqref{Expbb} because $f^\prime(x)=0$.    The estimate in  $L^\infty$ is easy to get in view of \eqref{Eqw1}
       \begin{eqnarray*}
         \|G_{11}\|_{L^\infty}&\le& 4\|f^\prime\|_{L^\infty}\int_{\R}  \|1/\varphi(\cdot,z)\|_{L^\infty} dz  \\
         &\le& C\Big(\|f^\prime\|_{L^\infty}+\|f^\prime\|_{L^\infty}^3\Big).
           \end{eqnarray*}  
               From the law products \eqref{law2}  and \eqref{lawX3} we deduce that
           \begin{eqnarray*}
         \|G_{11}\|_{X}&=& 2\int_{\R}\|f^\prime+f^\prime\circ(\hbox{Id}+z f)\|_{X }   \|1/\varphi(\cdot,z)\|_{L^\infty} dz  \\
         &+&2\int_{\R}\|f^\prime+f^\prime\circ(\hbox{Id}+z f)\|_{L^\infty }   \|1/\varphi(\cdot,z)\|_{X} dz\\
         &\triangleq& \ell_1+\ell_2.
           \end{eqnarray*}  
          According to the law products \eqref{comp1} and \eqref{comp3} one may write
     
     \begin{equation*}
     \|f^\prime+f^\prime\circ(\hbox{Id}+z f)\|_{X }\le \|f^\prime\|_{X}\Big(1+\mu\big(1+|z|\|f^\prime\|_{L^\infty}\big)\Big)
     \end{equation*}
     with
     \begin{equation*}
\mu(r)\triangleq\left\lbrace
\begin{array}{l}
\ln r, \quad \textnormal{if}\quad X=C^\star_K\\
r^s,\quad \textnormal{if}\quad X=C^s.
\end{array}
\right.
\end{equation*}
Observe that we can unify  both cases through the estimate
\begin{eqnarray}\label{Dqs1}
   \nonumber   \|f^\prime+f^\prime\circ(\hbox{Id}+z f)\|_{X }&\le& C\|f^\prime\|_{X}\Big(1+\big(1+|z|\|f^\prime\|_{L^\infty}\big)^s\Big)\\
     &\le& C\|f^\prime\|_{X}\Big(1+|z|^s\|f^\prime\|_{L^\infty}^s\Big).
     \end{eqnarray}
Putting together \eqref{Dqs1} and \eqref{Eqw1} we find  for any $s\in(0,1)$

\begin{eqnarray}\label{pgh1}
   \nonumber \ell_1&\le& C\|f^\prime\|_{X}\big(1+\|f^\prime\|_{L^\infty}^2\big)\int_{\R}\frac{1+|z|^s\|f^\prime\|_{L^\infty}^s}{1+z^2}dz\\
    &\le& C\|f^\prime\|_{X}\big(1+\|f^\prime\|_{L^\infty}^3\big).
     \end{eqnarray}
  To estimate $\ell_2$ we use the elementary estimate
  $$
  \|f^\prime+f^\prime\circ(\hbox{Id}+z f)\|_{L^\infty }\le 2\|f^\prime\|_{L^\infty}.
  $$
    
Notice  from  the definition of the spaces $X$  and \eqref{Eqw1} that one can deduce 
          \begin{eqnarray}\label{Invertt1}
        \nonumber  \|1/\varphi(\cdot,z)\|_{X}&\le&\|{1}/{\varphi(\cdot,z)}\|_{L^\infty}^2\|\varphi(\cdot,z)\|_{X}\\
          &\le& C\frac{1+\|f^\prime\|_{L^\infty}^4}{1+z^4} \|\varphi(\cdot,z)\|_{X}.
           \end{eqnarray} 
      Moreover by the law products
         \begin{eqnarray*}           
         \|\varphi(\cdot,z\|_{X}\le 2|z|\Big(2+|z|\|f^\prime\|_{L^\infty}\Big) \int_0^1\|f^\prime\circ\big(\hbox{Id}+\theta z f \big)\|_{X}d\theta,
         \end{eqnarray*} 
         which implies according to \eqref{Dqs1}
          \begin{eqnarray*}           
      \|\varphi(\cdot,z\|_{X}&\le& C|z|\Big(2+|z|\|f^\prime\|_{L^\infty}\Big)\|f^\prime\|_X \big(1+|z|^s\|f^\prime\|_{L^\infty}^s\big) \\
         &\leq& C \big(1+|z|^{2+s}\big)\big(1+\|f^\prime\|_{L^\infty}^{1+s}\big)  \|f^\prime\|_X.
              \end{eqnarray*} 
Putting together  this estimate with \eqref{Invertt1}      we find
        \begin{equation} \label{dezz1}          
      \|1/\varphi(\cdot,z\|_{X}\le C \frac{\big(1+\|f^\prime\|_{L^\infty}^{5+s}\big)  \|f^\prime\|_X}{{1+|z|^{2-s}}}.
              \end{equation} 
Therefore we deduce that
  \begin{eqnarray*}
\ell_2&\le &C\|f^\prime\|_{L^\infty}\big(1+\|f^\prime\|_{L^\infty}^{5+s}\big)\|f^\prime\|_{X}\\
&\le&C\big(1+\|f^\prime\|_{L^\infty}^{7}\big)\|f^\prime\|_{X}.
     \end{eqnarray*} 
     Combining this estimate with \eqref{pgh1} we obtain
     \begin{equation*}
     \|G_{11}\|_X\leq C\big(1+\|f^\prime\|_{L^\infty}^{7}\big)\|f^\prime\|_{X}.
     \end{equation*}
     It follows from  this latter estimate, \eqref{dgh1} and  \eqref{daca2} that
     \begin{equation}\label{pgh3}
     \|G_{1}\|_X\leq C\big(1+\|f^\prime\|_{L^\infty}^{7}\big)\|f^\prime\|_{X}.
     \end{equation}

   What is left is to estimate $G_2$. For this purpose  we write according to Taylor formula 
   
      \begin{eqnarray*}
 G_2(x)&=&
\textnormal{p.v.}\bigintsss_{\R}\frac{y f^\prime(x)\Big[f^\prime(x)\chi(y)+2\bigintsss_0^1f^\prime(x+\theta y)d\theta+f^\prime(x+y)\Big] }{y^2+(f(x)+f(x+y))^2} dy\\
&+&\textnormal{p.v.}\bigintsss_{\R}\frac{\Delta_yf(x) \Delta_y f^\prime(x) }{y^2+(f(x)+f(x+y))^2} dy+2f(x)f^\prime(x)\bigintsss_{1}^{+\infty}\frac{dy}{y^2+f^2(x)}\\
 &\triangleq&G_{2,1}(x)+G_{2,2}(x)+2 f^\prime(x)\arctan(f(x)).
     \end{eqnarray*}  
   where $\chi:\R\to\R$ is an even continuous compactly supported function belonging to $X$ ad taking the value $1$ on the neighborhood of $[-1,M1$. Note that we have used in the first line the identity: for any $x\in K$
   \begin{equation*}
   \textnormal{p.v.}\bigintsss_{-1}^{1}\frac{y}{y^2+[f(x+y)+f(x)]^2}dy=  \textnormal{p.v.}\bigintsss_{\R}\frac{y\chi(y)}{y^2+[f(x+y)+f(x)]^2}dy  
   \end{equation*}
   which follows from the fact that $f(x+y)=0, \forall y\notin [-1,1]$. 
Therefore we may write
   $$
   G_{2,1}(x)= (f^\prime(x))^2(T_f^{0,1}\chi)(x)+2\int_0^1 f^\prime(x) (T_f^{1,\theta} f^\prime)(x) d\theta+f^\prime(x)(T_f^{1,1}f^\prime)(x)
   $$
   where we use the notation $T_f^{\alpha,\beta}$ from Theorem \ref{th1}. 
   The estimate of $G_{2,1}$  in $L^\infty$ is  quite easy and follows from Theorem \ref{th1},
   \begin{equation*}
  \|G_{2,1}\|_{L^\infty}\leq   C\|f^\prime\|_{L^\infty}\|f^\prime\|_D\Big(1+\|f^\prime\|_D^2\Big).
   \end{equation*}
   However to  estimate of $G_{2,2}$ in $L^\infty$ it is more convenient to write it in the form
   $$
   G_{2,2}(x)=\textnormal{p.v.}\bigintsss_{-1}^1\frac{\Delta_yf(x) \Delta_y f^\prime(x) }{y^2+(f(x)+f(x+y))^2} dy+2 f^\prime(x)\arctan(f(x)).
   $$ 
 Thus using  the mean value theorem we find 
   \begin{equation*}
  \|G_{2,2}\|_{L^\infty}\leq   C\|f^\prime\|_{L^\infty}\|f^\prime\|_D.
   \end{equation*}
   Combining these estimates with \eqref{Imbed1} we obtain
   \begin{equation}\label{kar8}
  \|G_{2}\|_{L^\infty}\leq   C\|f^\prime\|_{L^\infty}\Big(\|f^\prime\|_D+\|f^\prime\|_D^3\Big).
   \end{equation} 
     We shall now implement  the estimates in $X$ and start with the term $G_{2,1}.$ According to   Theorem \ref{th1} we can unify the estimates in $C^\star_K$ and $C^s$ and get the following weak estimate
   \begin{equation}\label{tiit1}
  \|f^\prime T_f^{\alpha,\beta} g\|_X\leq   C\|g\|_X\Big(\|f^\prime\|_X^{\frac12}\beta^{-\frac12}+\|f^\prime\|_X^{15}\Big).
   \end{equation}
   From the law products \eqref{law2} and  \eqref{lawX3}  one has
   \begin{equation*}
   \|(f^\prime)^2(T_f^{0,1}\chi)\|_X\leq \|f^\prime\|_{L^\infty} \|f^\prime T_f^{0,1}\chi\|_X+\|f^\prime\|_{X}\|f^\prime\|_{L^\infty} \|T_f^{0,1}\chi\|_{L^\infty}.
   \end{equation*}
   
   Hence we find
\begin{eqnarray*}
   \|(f^\prime)^2(T_f^{0,1}\chi)\|_X&\leq& C \|f^\prime\|_{L^\infty}\Big(\|f^\prime\|_X^{\frac12} +\|f^\prime\|_X^{15}\Big)+\|f^\prime\|_{X}\|f^\prime\|_{L^\infty} \big(1+\|f^\prime\|_X^2\big)\\
   &\le& C \|f^\prime\|_{L^\infty}\Big(\|f^\prime\|_X^{\frac12} +\|f^\prime\|_X^{15}\Big).
\end{eqnarray*}

   Using \eqref{tiit1} we get successively 
  \begin{equation}\label{tataqq1}
     \|f^\prime T_f^{0,\theta} f^\prime\|_X\leq   C\|f^\prime\|_X\Big(\|f^\prime\|_X^{\frac12}\theta^{-\frac12}+\|f^\prime\|_X^{15}\Big).
\end{equation}
and
\begin{equation*}
     \|f^\prime T_f^{1,1} f^\prime\|_X\leq   C\|f^\prime\|_X\Big(\|f^\prime\|_X^{\frac12}+\|f^\prime\|_X^{15}\Big).
\end{equation*}
  Thus using the foregoing inequalities we deduce that

 \begin{eqnarray}\label{Tata11}
  \nonumber \|G_{2,1}\|_{X}&\leq& C \|f^\prime\|_{L^\infty}\Big(\|f^\prime\|_X^{\frac12} +\|f^\prime\|_X^{15}\Big)+C\|f^\prime\|_X\Big(\|f^\prime\|_X^{\frac12}+\|f^\prime\|_X^{15}\Big)\\
   &\le& C\Big(\|f^\prime\|_X^{\frac32}+\|f^\prime\|_X^{16}\Big).
   \end{eqnarray}
   When $X=C^s$ we can give a refined  estimate for \eqref{tataqq1} using \eqref{hmi24}
   \begin{eqnarray*}
     \|f^\prime T_f^{0,\theta} f^\prime\|_s &\leq&   C\|f^\prime\|_{L^\infty}^{\frac{1}{2+s}}\Big(\|f^\prime\|_s^{\frac{3+2s}{2+s}}\theta^{-\frac12}+\|f^\prime\|_s^{15}\Big)
\end{eqnarray*}
which implies that
\begin{eqnarray}\label{Tata12}
  \nonumber \|G_{2,1}\|_{s}&\leq& C \|f^\prime\|_{L^\infty}\Big(\|f^\prime\|_s^{\frac12} +\|f^\prime\|_s^{15}\Big)+C\|f^\prime\|_{L^\infty}^{\frac{1}{2+s}}\Big(\|f^\prime\|_s^{\frac{3+2s}{2+s}}+\|f^\prime\|_s^{15}\Big)
\\
   &\le&C\|f^\prime\|_{L^\infty}^{\frac{1}{3}}\Big(\|f^\prime\|_s+\|f^\prime\|_s^{16}\Big).
   \end{eqnarray}
   Hence one can unify \eqref{Tata11} and \eqref{Tata12} 
   \begin{equation}\label{Tata13}
\|G_{2,1}\|_{X}\leq C\|f^\prime\|_{D}^{\frac{1}{3}}\Big(\|f^\prime\|_X+\|f^\prime\|_X^{16}\Big).
   \end{equation}
As to the term $G_{2,2}$ we may  write
\begin{eqnarray*}
 G_{2,2}(x)&=&2f^\prime(x)\arctan(f(x))+\bigintsss_{-M}^{M}\frac{\big[\Delta_yf(x)-yf^\prime(x)\big] \Delta_y f^\prime(x)}{y^2+(f(x+y)+f(x))^2} dy\\
 &+&\textnormal{p.v.}\bigintsss_{\R}\frac{yf^\prime(x) \Delta_y f^\prime(x) }{y^2+(f(x+y)+f(x))^2} dy\\
 &\triangleq&2f^\prime(x)\arctan(f(x))+G_{2,2}^1(x)+G_{2,2}^2(x).
     \end{eqnarray*}  
    
     The last term  was  treated in the preceding estimates and we obtain as in \eqref{Tata13} 
     \begin{equation}\label{Tata14}
 \|G_{2,2}^2\|_{X}
   \le C\|f^\prime\|_{D}^{\frac{1}{3}}\Big(\|f^\prime\|_X+\|f^\prime\|_X^{16}\Big).
   \end{equation}
It remains to estimate $G_{2,2}^1$ which can be split into two terms 

$$
G_{2,2}^1(x)=\widehat{G}_{\textnormal{int},r}(x)+\widehat{G}_{\textnormal{ext},r}(x)
$$
with
$$
\widehat{G}_{\textnormal{int},r}(x)=\bigintsss_{|y|\leq r}\frac{\big[\Delta_yf(x)-yf^\prime(x)\big] \Delta_y f^\prime(x)}{y^2+(f(x+y)+f(x))^2} dy
$$
and
$$
\widehat{G}_{\textnormal{ext},r}(x)=\bigintsss_{M\geq |y|\geq r}\frac{\big[\Delta_yf(x)-yf^\prime(x)\big] \Delta_y f^\prime(x)}{y^2+(f(x+y)+f(x))^2} dy.
$$

Now we shall proceed as in the proof of \mbox{Theorem \ref{th1}.} Let $r\in (0,1)$ and $x_1,x_2\in \R$ such that $|x_1-x_2|\leq r$. First it is clear that 
\begin{equation}\label{zita1}
  |\Delta_y f^\prime(x)|\le \omega_{f^\prime}(|y|).
\end{equation}
In addition, using Taylor formula we get

\begin{equation}\label{zita2}
 |\Delta_yf(x)-yf^\prime(x)|\le |y|\omega_{f^\prime}(|y|).
\end{equation}
 
 Therefore 
 \begin{eqnarray*}
  |\widehat{G}_{\textnormal{int},r}(x)|\le \int_{|y|\leq r}\frac{[\omega_{f^\prime}(|y|)]^2}{|y|} dy.
  \end{eqnarray*}
  It follows that
  \begin{equation}\label{kaou11}
  \sup_{|x_1-x_2|\le r}|\widehat{G}_{\textnormal{int},r}(x_2)-\widehat{G}_{\textnormal{int},r}(x_1)|\le 4 \bigintsss_{0}^r\frac{[\omega_{f^\prime}(y)]^2}{y} dy.
  \end{equation}
Hence by Fubini's theorem
\begin{eqnarray*}
 \bigintsss_0^1 \sup_{|x_1-x_2|\le r}|\widehat{G}_{\textnormal{int},r}(x_1)-\widehat{G}_{\textnormal{int},r}(x_1)|\frac{dr}{r}&\le& 4 \bigintsss_{0}^1\frac{[\omega_{f^\prime}(y)]^2}{y} |\ln y|dy.
 \end{eqnarray*}
 From the definition and the monotonicity of the modulus of continuity one deduces that for any $r\in (0,1)$
 \begin{eqnarray*}
|\ln r|\omega_{f^\prime}(r)\le \bigintsss_r^1 \frac{\omega_{f^\prime}(y)}{y}dy\le  \|f^\prime\|_{D} 
\end{eqnarray*}
which implies that
\begin{equation}\label{gkl1}
 \bigintsss_0^1 \sup_{|x_1-x_2|\le r}|\widehat{G}_{\textnormal{int},r}(x_1)-\widehat{G}_{\textnormal{int},r}(x_1)|\frac{dr}{r}\le  4 \|f^\prime\|_{D}^2.
 \end{equation}
 To get the suitable estimate in $C^s$ we come back to \eqref{kaou11} which gives 
 \begin{eqnarray*}
  \sup_{|x_1-x_2|\le r}|\widehat{G}_{\textnormal{int},r}(x_2)-\widehat{G}_{\textnormal{int},r}(x_1)|&\le &4\|f^\prime\|_s^2 \int_{0}^ry^{2s-1} dy\\
  &\le& C \|f^\prime\|_s^2 r^{2s}\end{eqnarray*}
  and thus 
 \begin{equation}\label{gkl2}
  \sup_{|x_1-x_2|\le 1}\frac{|\widehat{G}_{\textnormal{int},r}(x_2)-\widehat{G}_{\textnormal{int},r}(x_1)|}{|x_1-x_2|^s}  \le C \|f^\prime\|_s^2.
   \end{equation}

As to $\widehat{G}_{\textnormal{ext},r}$ one writes 
\begin{eqnarray*}
\widehat{G}_{\textnormal{ext},r}(x_1)-\widehat{G}_{\textnormal{ext},r}(x_2)&=&\int_{M\geq |y|\geq r}\frac{\mathcal{N}(x_1,y)-\mathcal{N}(x_2,y)}{\mathcal{K}(x_1)} dy\\
&+& \int_{M\geq |y|\geq r}\frac{\mathcal{N}(x_2,y)\big[\mathcal{K}(x_2,y)-\mathcal{K}(x_1,y)}{\mathcal{K}(x_1,y)\mathcal{K}(x_2,y)} dy
\end{eqnarray*}
with
$$
\mathcal{N}(x,y)=[\Delta_yf(x)-yf^\prime(x)\big] \Delta_y f^\prime(x)\quad \textnormal{and} \quad  \mathcal{K}(x,y)=y^2+(f(x)+f(x+y))^2.
$$
Notice that from \eqref{zita1} and \eqref{zita2} one gets
\begin{equation}\label{draa1}
|\mathcal{N}(x_1,y)-\mathcal{N}(x_2,y)|\le C |y|\omega_{f^\prime}(r) \omega_{f^\prime}(|y|)\quad \hbox{and}\quad |\mathcal{N}(x,y)|\le 2  |y|\omega_{f^\prime}(|y|)\|f^\prime\|_{L^\infty}.
\end{equation}
In addition using  straightforward calculus  we obtain
\begin{eqnarray*}
|\mathcal{K}(x_1,y)-\mathcal{K}(x_2,y)|\le C r \|f^\prime\|_{L^\infty}\big(\sqrt{\mathcal{K}(x_1,y)}+\sqrt{\mathcal{K}(x_2,y)}\big).
\end{eqnarray*}
Thus
\begin{eqnarray*}
 \sup_{|x_1-x_2|\le r}\frac{|\mathcal{N}(x_2,y)||\mathcal{K}(x_2,y)-\mathcal{K}(x_1,y)|}{\mathcal{K}(x_1,y)\mathcal{K}(x_2,y)}&\le& C r \|f^\prime\|_{L^\infty}^2 \frac{\omega_{f^\prime}(|y|)}{|y|^2}\cdot
\end{eqnarray*}
Hence  we get by Fubini's theorem and \eqref{L1}
\begin{eqnarray*}
\bigintsss_0^1\sup_{|x_1-x_2|\le r}\bigintsss_{\{M\geq |y|\geq r\}}\frac{|\mathcal{N}(x_1,y)-\mathcal{N}(x_2,y)|}{\mathcal{K}(x_1)} dy\frac{dr}{r}
&\le& \bigintsss_0^1\bigintsss_{\{M\geq |y|\geq r\}}\omega_{f^\prime}(r) {\omega_{f^\prime}(|y|)}\frac{dy}{|y|}\frac{dr}{r}\\
&\le& C\|f^\prime\|_{D}^2
\end{eqnarray*}
and 
\begin{eqnarray*}
{{\bigintsss_0^1{\sup_{|x_1-x_2|\le r}\bigintsss_{\{M\geq |y|\geq r\}}\frac{\scriptstyle {|\mathcal{N}(x_2,y)||\mathcal{K}(x_2,y)-\mathcal{K}(x_1,y)|}}{\scriptstyle{\mathcal{K}(x_1,y)\mathcal{K}(x_2,y)}}dy \frac{dr}{r}}}}&\le& C  \|f^\prime\|_{L^\infty}^2\bigintsss_0^1\bigintsss_{\{M\geq |y|\geq r\}}
{{ \frac{\omega_{f^\prime}(|y|)}{|y|^2}}}dy dr
\\
&\le& C\|f^\prime\|_{L^\infty}^2\|f^\prime\|_{D}.
\end{eqnarray*}

Finally we obtain
\begin{eqnarray*}
\bigintsss_0^1\sup_{|x_1-x_2|\le r}|\widehat{G}_{\textnormal{ext},r}(x_1)-\widehat{G}_{\textnormal{ext},r}(x_2)|\frac{dr}{r}&\le &C\|f^\prime\|_{D}^2+C\|f^\prime\|_{L^\infty}^2\|f^\prime\|_{D}.
\end{eqnarray*}
As to the estimate in $C^s$ we use \eqref{draa1} which implies that
\begin{eqnarray*}
\bigintsss_{\{r\le|y|\le M\}}\frac{|\mathcal{N}(x_1,y)-\mathcal{N}(x_2,y)|}{\mathcal{K}(x_1)}dy&\le& C\|f^\prime\|_s r^s \bigintsss_{\scriptstyle\{r\le|y|\le M\}}\frac{\omega_{f^\prime}(|y|)}{|y| }dy\\
&\le& C\|f^\prime\|_s\|f^\prime\|_D r^s
\end{eqnarray*}
and 
\begin{eqnarray*}
\bigintsss_{\{M\geq |y|\geq r\}}\frac{|\mathcal{N}(x_2,y)||\mathcal{K}(x_2,y)-\mathcal{K}(x_1,y)|}{\mathcal{K}(x_1,y)\mathcal{K}(x_2,y)}dy&\le& C  \|f^\prime\|_{L^\infty}^2\|f^\prime\|_s  r\bigintsss_{\{M\geq |y|\geq r\}}
 \frac{dy}{|y|^{2-s}} 
\\
&\le& C\|f^\prime\|_{L^\infty}^2\|f^\prime\|_{s} r^s.
\end{eqnarray*}
It follows from Sobolev embedding $C^s\hookrightarrow L^\infty$  that
$$
\sup_{|x_1-x_2|\le r}\frac{|\widehat{G}_{\textnormal{ext},r}(x_1)-\widehat{G}_{\textnormal{ext},r}(x_2)|}{|x_1-x_2|^s}\leq C\|f^\prime\|_{D}\|f^\prime\|_s+C\|f^\prime\|_{D}\|f^\prime\|_{s}^2.
$$
Combining the foregoing  estimates with \eqref{gkl1} and \eqref{gkl2} we deduce 
$$
\|G_{2,2}^1\|_{X}\le C\|f^\prime\|_{D}\big(\|f^\prime\|_X+\|f^\prime\|_X^2\big).
$$
Putting together this estimate with  \eqref{Tata13} and \eqref{Tata14} we get
\begin{equation}\label{Tata17}
 \|G_2\|_{X}
   \le C\|f^\prime\|_{D}^{\frac{1}{3}}\Big(\|f^\prime\|_X+\|f^\prime\|_X^{16}\Big).
   \end{equation}
Now using  \eqref{pgh3} and  \eqref{Tata17} we find 
   \begin{eqnarray*}
 \nonumber \|G\|_{X}
  & \le &C\|f^\prime\|_X\big(1+\|f^\prime\|_D^{7}\big)+C\|f^\prime\|_{D}^{\frac{1}{3}}\Big(\|f^\prime\|_X+\|f^\prime\|_X^{16}\Big)\\
  &\le&C\big(1+\|f^\prime\|_{D}^{\frac{1}{3}}\big)\Big(\|f^\prime\|_X+\|f^\prime\|_X^{16}\Big) 
   \end{eqnarray*}
   which ends the proof of Proposition \ref{prop10}
  \end{proof}   
  \subsection{A priori estimates}\label{Secdez2}
  The aim of this section is to establish weak and strong a priori estimates for solutions to the equation \eqref{graph1}. This part  is the cornerstone of the local well-posedness theory. The main result of this section reads as follows.   \begin{proposition}\label{prop20}
  Let $f:[0,T]\times\R\to\R$ be  a smooth solution for the graph equation \eqref{graph1}.  Assume that the initial data is positive and with compact support $K_0.$ Then the following assertions hold true.
  \begin{enumerate}
    \item For any $t\in[0,T]$, the function $f_t$ is positive and  
  $$
 \forall t\in [0,T],\quad \| f(t)\|_{L^\infty}\leq  \|f_0\|_{L^\infty}.
  $$
  \item For any $t\in [0,T]$, we have
  $$
  \|f(t)\|_{L^1}= {\|f_0\|_{L^1}} e^{-t}.
  $$

  \item The support a $\hbox{supp }f_t$ is contained in the convex hull of  $K_0$, that is
    $$
\forall\,t\in[0,T], \quad \textnormal{supp } f(t)\subset \textnormal{Conv} K_0.
$$
\item Set $X=C^\star_K$ or  $X=C^s_K,$ with $s\in (0,1).$ If  $f_0^\prime\in X$ then there exists $T$ depending only on $\|f_0^\prime\|_X$ such that $f^\prime\in L^\infty([0,T];X)$. \end{enumerate}
  \end{proposition}
  \begin{proof}
  ${\bf{(1)}}$  To get the first part about the persistence of the positivity of  we shall prove that
  \begin{equation}\label{maximum}
  \forall x\in \R,\quad u_2(t,x)=f(t,x) U(t,x)
 \end{equation}
  with
  $$
  \|U(t)\|_{L^\infty}\leq C\big(1+\|f^\prime(t)\|_{D}^6\big)
  $$
  and  $C$  being a constant depending only on the size of  the support of $f_t.$ Note from the point $({\bf{2}})$ of the current proposition  that the support of $f_t$ is contained in a fixed compact and therefore the constant $C$ can be taken  independent of the time variable. 
  Assume for a while   \eqref{maximum} and let us see how to propagate the positivity. Denote by  $\psi$ the flow associated to the velocity $u_1$, that is, the solution of the ODE
 \begin{equation}\label{flot1}
  \partial_t\psi(t,x)=u_1(t,\psi(t,x)),\quad \psi(0,x)=x.
\end{equation}
Recall that
\begin{equation*}
 u_1(t,x)=\frac{1}{2\pi}\bigintsss_{\R}\Bigg\{\arctan\Big(\frac{f(t,x+y)-f(t,x)}{y}\Big)-\arctan\Big(\frac{f(t,x+y)+f(t,x)}{y}\Big)\Bigg\} dy.
\end{equation*}
Set
$$
\eta(t,x)=f(t,\psi(t,x))
$$
then 
\begin{eqnarray}\label{gfunct}
\partial_t \eta(t,x)&=&u_2(t,\psi(t,x))\\
\nonumber&=&\eta(t,x) U(t,\psi(t,x)).
\end{eqnarray}
Consequently
$$
\eta(t,x)=f_0(x) e^{\int_0^tU(\tau,\psi(\tau,x))d\tau}.
$$
Since the flow $\psi(t):\R\to\R$ is a diffeomorphism we get the representation   
\begin{equation}\label{flot2}
 f(t,x)=f_0\big(\psi^{-1}(t,x)\big) e^{\int_0^tU[\tau,\psi(\tau,\psi^{-1}(t,x))]d\tau}.
\end{equation}
 As an immediate consequence we get the persistance through the time of  the  positivity of the solution.  Let us now come back to the proof of  the identity \eqref{maximum}.  To alleviate the notation we remove the variable $t$ from the functions. Applying Taylor formula to the function
 $$
\tau\in [0, f(x)]
\mapsto g(\tau)\triangleq\log\Bigg[\frac{y^2+\big(\tau -f(x+y)\big)^2}{y^2+\big(\tau+f(x+y)\big)^2}\Bigg] $$
 yields to
\begin{eqnarray*}
 -2\pi u_2(x)&=&f(x)\bigintsss_0^1\bigintsss_{-M}^M\frac{f(x+y)-\tau f(x)}{y^2+\big[f(x+y)-\tau f(x)\big]^2}d\tau dy\\
 &+&f(x)\bigintsss_0^1\bigintsss_{-M}^M\frac{f(x+y)+\tau f(x)}{y^2+\big[f(x+y)+\tau f(x)\big]^2}d\tau dy\\
 &\triangleq&f(x)V_1(x)+f(x)V_2(x).
\end{eqnarray*}
 Using once again Taylor formula we get the following expressions
\begin{eqnarray*}
V_1(x)&=&\bigintsss_0^1\bigintsss_{-M}^M\frac{(1-\tau) f(x)}{y^2+\big[(1-\tau) f(x)+y\int_0^1f^\prime(x+\theta y) d\theta\big]^2}d\tau dy\\
&+&\textnormal{p.v.} \bigintsss_0^1\bigintsss_{-M}^M\frac{y\int_0^1f^\prime(x+\theta y) d\theta}{y^2+\big[f(x+y)-\tau f(x)\big]^2}d\tau dy \\
&\triangleq& V_{1,1}(x)+V_{1,2}(x)
\end{eqnarray*}
and
\begin{eqnarray*}
V_2(x)&=&\bigintsss_0^1\bigintsss_{-M}^M\frac{(1+\tau) f(x)}{y^2+\big[(1+\tau) f(x)+y\int_0^1f^\prime(x+\theta y) d\theta\big]^2}d\tau dy\\
&+&\textnormal{p.v.} \bigintsss_0^1\bigintsss_{-M}^M\frac{y\int_0^1f^\prime(x+\theta y) d\theta}{y^2+\big[f(x+y)+\tau f(x)\big]^2}d\tau dy \\
&\triangleq& V_{2,1}(x)+V_{2,2}(x).
\end{eqnarray*}
To estimate $V_{1,1}$ and $V_{2,1}$ we can assume that $f(x)>0$. Then making the change of variables $z\mapsto y=(1-\tau) f(x) z$ leads to
\begin{equation}\label{V11}
V_{1,1}(x)=\bigintsss_0^1\bigintsss_{-\frac{M}{(1-\tau)f(x)}}^{\frac{M}{(1-\tau)f(x)}}\frac{d\tau dz}{z^2+\big[1+z\int_0^1f^\prime\big(x+\theta (1-\tau) f(x) z\big) d\theta\big]^2}\cdot
\end{equation}
Using \eqref{Eqw1} we deduce that
\begin{equation}\label{kiri1}
\|V_{1,1}\|_{L^\infty}\leq C\big(1+\|f^\prime\|_{L^\infty}^2\big).
\end{equation}
Similarly we get
\begin{equation}\label{kiri2}
\|V_{2,1}\|_{L^\infty}\leq C\big(1+\|f^\prime\|_{L^\infty}^2\big).
\end{equation}
Let us now bound $V_{j,2}, j=1,2$ 
First  by symmetry we write  \begin{eqnarray*}
V_{1,2}(x)&=&\bigintsss_0^1 \bigintsss_{0}^{M}\frac{y\int_0^1f^\prime(x+\theta y) d\theta\big[f(x-y)-f(x+y)\big]\psi_\tau(x,y)}{\Big(y^2+[f(x+y)-\tau f(x)]^2\Big) \Big(y^2+[f(x-y)-\tau f(x)]^2\Big)} dy d\tau\\
&+&\bigintsss_0^1 \bigintsss_{0}^{M}\frac{y\int_0^1\big[f^\prime(x+\theta y)-f^\prime(x-\theta y)\big] d\theta}{y^2+[f(x-y)-\tau f(x)]^2} dy d\tau
\end{eqnarray*}
where
\begin{eqnarray*}
\psi_\tau(x,y)&=&f(x+y)+f(x-y)-2\tau f(x)\\
&=&2(1-\tau) f(x)+y\int_0^1\big[f^\prime(x+\theta y)-f^\prime(x-\theta y)\big] d\theta.
\end{eqnarray*}
Thus
\begin{eqnarray*}
\|V_{1,2}\|_{L^\infty}&\le& C\bigintsss_0^1\bigintsss_{0}^{M}\frac{\|f^\prime\|_{L^\infty}^2y^2\big[(1-\tau) f(x)+y\omega_{f^\prime}(y)\big]}{\Big(y^2+[f(x+y)-\tau f(x)]^2\Big) \Big(y^2+[f(x-y)-\tau f(x)]^2\Big)} dyd\tau\\
&+&C\bigintsss_0^1 \bigintsss_{0}^{M}\frac{\omega_{f^\prime}(y)}{y} dy d\tau.
\end{eqnarray*}
Similarly to  $V_{1,1}$ one gets 
\begin{eqnarray*}
\bigintsss_0^1\bigintsss_{0}^{M}\frac{y^2(1-\tau) f(x)\,  dy d\tau}{\Big(y^2+[f(x+y)-\tau f(x)]^2\Big) \Big(y^2+[f(x-y)-\tau f(x)]^2\Big)}&\le& C \big(1+\| f^\prime\|_{L^\infty}^4\big).
\end{eqnarray*}
It follows that
\begin{eqnarray}\label{V24}
\nonumber\|V_{1,2}\|_{L^\infty}&\leq&  C\|f^\prime\|_{L^\infty}^2\Big( 1+\| f^\prime\|_{L^\infty}^4+\int_0^M\frac{\omega_{f^\prime}(y)}{y} dy\Big)+C\|f^\prime\|_D\\
&\leq& C\|f^\prime\|_{L^\infty}^2\Big( 1+\| f^\prime\|_{L^\infty}^4+\|f^\prime\|_D\Big)+C\|f^\prime\|_D.
\end{eqnarray}
The estimate of $V_{2,2}$ can be done in a similar way and one obtains
\begin{equation}\label{V22}
\|V_{2,2}\|_{L^\infty}=  C\|f^\prime\|_{L^\infty}^2\Big( 1+\| f^\prime\|_{L^\infty}^4+\|f^\prime\|_D\Big)+C\|f^\prime\|_D.
\end{equation}
Combining both last estimates with \eqref{kiri1} and \eqref{kiri2} we finally get according to the embedding \eqref{Imbed1} 
$$
\|U\|_{L^\infty}\leq C\big(1+\|f^\prime\|_{D}^6\big)
$$
where the constant $C$ depends only on the size of the support of $f.$

Now let us establish the maximum principle. From \eqref{fields6} combined with the positivity of $f_t$ one gets 
$$
\forall t\in [0,T], \,\forall x\in \R\quad u_2(t,x)\leq0.
$$
Coming back to \eqref{gfunct} we deduce that
$$
\partial_t\eta(t,x)\leq0
$$
which implies in turn that
$$
\forall t\in [0,T], \,\forall x\in \R\quad f(t,x)\leq f_0\big(\psi^{-1}(t,x)\big).
$$
Combined with the positivity of $f(t)$ we deduce immediately the maximum principle
$$
\forall t\in [0,T],\quad \|f(t)\|_{L^\infty}\leq\|f_0\|_{L^\infty}.
$$
Now we intend to provide more refined identity that we shall use later in studying the asymptotic behavior of the solution. Actually we have 
\begin{equation}\label{VY1}
u_2(t,x)=-f(t,x)\big(1+R(t,x)),
\end{equation}
with
$$
 \|R(t)\|_{L^\infty}\leq C\|f^\prime(t)\|_{D}\Big(1+\|f^\prime(t)\|_{L^\infty}^5\Big).
$$
First note that $R=\sum_{i,j=1}^2 V_{i,j}$. The estimates of $V_{1,2}$ and $V_{2,2}$ are done in \eqref{V24} and  \eqref{V22}. However  to deal with $V_{1,1}$ and similarly $V_{2,1}$ we return to the expression \eqref{V11}. Set
$$
\tau\mapsto K(\tau)=\frac{1}{z^2+\big[1+z\tau\big]^2}\cdot
$$
Easy computations using \eqref{Eqw1} show the existence of a  positive constant $C$ such that
\begin{eqnarray*}
\forall\tau,z\in\R,\quad  |K^\prime(\tau)|&=&\frac{2|z||1+z\tau|}{\big(z^2+[1+z t]^2\big)^2}\\
&\le&\frac{1}{z^2+[1+z \tau]^2}\\
&\le& C\frac{1+\tau^2}{1+z^2}\cdot
\end{eqnarray*}
Applying  the mean value theorem yields
$$
|K(\tau)-\frac{1}{1+z^2}|\leq C|\tau|\frac{1+\tau^2}{1+z^2}\cdot
$$
Therefore we get
\begin{equation*}
\Bigg|V_{1,1}(x)-\bigintsss_0^1\bigintsss_{-\frac{M}{(1-\tau)f(x)}}^{\frac{M}{(1-\tau)f(x)}}\frac{dz d\tau}{1+z^2}\Bigg|\leq C \|f^\prime\|_{L^\infty}\Big(1+\|f^\prime\|_{L^\infty}^2\Big).
\end{equation*}
which implies that 
\begin{equation}\label{V1234}
\bigg|V_{1,1}(x)-\pi\big|\leq C \|f^\prime\|_{L^\infty}\Big(1+\|f^\prime\|_{L^\infty}^2\Big)+ C\|f\|_{L^\infty}.
\end{equation}
Similarly we obtain
\begin{equation}\label{V12345}
\big|V_{2,1}(x)-\pi\big|\leq C\|f^\prime\|_{L^\infty}\Big(1+\|f^\prime\|_{L^\infty}^2\Big)+C\|f\|_{L^\infty}.
\end{equation}
Putting together \eqref{V24},\eqref{V22}, \eqref{V1234}, \eqref{V12345} we get \eqref{VY1}.
 \vspace{0,3cm}
 
 ${\bf{(2)}}$
Integrating the equation \eqref{sqg} in the space variable we get after integration by parts 
\begin{eqnarray*}
\frac{d}{dt}\int_{\R}\rho(t,x)dx&=&\int_{\R}\textnormal{div}\,v(t,x)\rho(t,x) dx\\
&=&-\int_{\R}\rho^2(t,x) dx\\
&=&-\int_{\R}\rho(t,x) dx
\end{eqnarray*}
where in the last line we have used that for the characteristic function one has $\rho^2=\rho$. The time decay follows then easily.
 
  \vspace{0,3cm}
    ${\bf{(3)}}$ According to the representation of the solution given by \eqref{flot2} we have easily that the support of $f(t)$ is the image by the flow $\psi(t)$ of the initial support, that is, 
    \begin{equation}\label{repX1}
    K_t=\psi(t,K_0).
    \end{equation}
 We have to check that if $K_0\subset [a,b]$, with $a<b,$ then $K_t\subset [a,b].$ To do so it is enough to prove that
 $$
 \psi(t,[a,b])\subset [a,b].
 $$
 This means somehow that the flow is contractive. As $\psi(t)$ is an homeomorphism then necessary $ \psi(t,[a,b])=[\psi(t,a),\psi(t,b)]$. Hence to get the desired inclusion it suffices to establish that
 $$
a_t\triangleq  \psi(t, a)\geq a\quad\hbox{and}\quad b_t\triangleq \psi(t,b)\leq b.
 $$
This reduces to study the derivative in time of $a_t$ and $b_t$. First one has 
 $$
 \dot{a}_t=u_1(t,a_t)\quad\hbox{and}\quad  \dot{b}_t=u_1(t,b_t).
 $$
 Since $f(t,y)=0, \forall y\notin (a_t,b_t)$ and $f_t$ is positive eveywhere then
 \begin{equation*}
u_1(t,a_t)=\frac{1}{\pi}\int_{0}^{b_t-a_t}\arctan\Big(\frac{f_t(a_t+y)}{y}\Big) dy\geq0.
\end{equation*}
Hence $\dot{a}_t\geq0$ and therefore $a_t\geq a,$ for any $t\in[0;T].$
 
 Similarly we get
 \begin{equation*}
u_1(t,b_t)=-\frac{1}{\pi}\int_{0}^{b_t-a_t}\arctan\Big(\frac{f_t(b_t-y)}{y}\Big) dy\leq0
\end{equation*}
 which implies that $b_t\leq b,$ for any $t\in[0;T].$ This ends the proof of the point ${\bf{(2)}}$.
  \vspace{0,3cm}
 
    ${\bf{(4)}}$ Recall from \eqref{graphder} and \eqref{Ham11} that $g\triangleq f^\prime$ satisfies the equation
    $$
    \partial_t g+u_1\partial_1 g=\frac{1}{2\pi}\big( F-G\big).
    $$
  Set $h(t,x)=g(t,\psi(t,x))$, where $\psi$ is the flow defined in \eqref{flot1}. Then
  $$
  \partial_t h(t,x)=\frac{1}{2\pi}\Big( F\big(t,\psi(t,x)\big)-G\big(t,\psi(t,x)\big)\Big).
  $$
  Thus
 $$
  g(t,x)=g_0(\psi^{-1}(t,x)+\frac{1}{2\pi}\int_0^t (F-G)\big(\tau,\psi\big(\tau,\psi^{-1}(t,x\big)\big)d\tau.
    $$ 
  Recall the classical estimate
 \begin{equation}\label{Tham01}
  \big\|\partial_x\big[\psi\big(\tau,\psi^{-1}(t,\cdot)\big)\big]\big\|_{L^\infty}\leq e^{\int_\tau^t\|\partial_xu_1(t^\prime,\cdot)\|_{L^\infty}d t^\prime}
  \end{equation}
  that we may  combine with the 
composition laws \eqref{comp1} and \eqref{comp3}  to get
  \begin{equation}\label{Tham1}
  \|g(t)\|_{X}\leq Ce^{V(t)}\Big[\|g_0\|_{X}+\int_0^t\|(F-G)(\tau)\|_{X}d\tau\Big], \quad V(t)\triangleq \int_0^t\|\partial_x u_1(\tau)\|_{L^\infty} d\tau.
  \end{equation}  
  To estimate $\|\partial_x u_1(t)\|_{L^\infty}$ we come back to \eqref{Derr}. The first integral term can be restricted to a compact set $[-M,M]$ and thus 
  \begin{eqnarray*}
  \Bigg|\textnormal{p.v.}\bigintsss_{-M}^M\frac{f^\prime(x+y)-f^\prime(x) }{y^2+(f(x+y)-f(x))^2} ydy\Bigg|&\leq&2\bigintsss_0^M\frac{\omega_{f^\prime}(y)}{y} dy\\
  &\le& C\|f^\prime\|_D.
    \end{eqnarray*}
  As to the second term, the integral can be restricted to $[-M,M]$ and  we simply write
  \begin{eqnarray*}
     \textnormal{p.v.} \bigintsss_{\R}\frac{f^\prime(x+y)+f^\prime(x) }{y^2+(f(x+y)+f(x))^2} ydy&=&\textnormal{p.v.} \bigintsss_{-M}^M\frac{f^\prime(x+y)-f^\prime(x) }{y^2+(f(x+y)+f(x))^2} ydy\\
     &+&\textnormal{p.v.} \bigintsss_{\R}\frac{2f^\prime(x) }{y^2+(f(x+y)+f(x))^2} ydy.
   \end{eqnarray*}
   The first term of the right-hand side is controlled as before
   \begin{eqnarray*}
  \Big|\textnormal{p.v.}\int_{-M}^M\frac{f^\prime(x+y)-f^\prime(x) }{y^2+(f(x+y)+f(x))^2} ydy\Big|
  &\le& C\|f^\prime\|_D.  \end{eqnarray*}
  However for the last term it can be estimated as in the proof of Theorem \ref{th1}-$(1)$. One gets in view of \eqref{Eq10}, \eqref{lila90} and \eqref{embedd11} 
  $$
  \Bigg|\textnormal{p.v.} \int_{\R}\frac{y }{y^2+(f(x+y)+f(x))^2} dy\Bigg|\le C\Big(\|f^\prime\|_{L^\infty}^2+\|f^\prime\|_{L^\infty}\|f^\prime\|_{D}+\|f^\prime\|_{L^\infty}\Big).
  $$
  Hence using the embedding $X\hookrightarrow C^\star_K\hookrightarrow L^\infty$ we find
\begin{eqnarray}\label{zerda}
  \nonumber\|\partial_x u_1(t)\|_{L^\infty}&\leq& C\Big(\|f^\prime\|_{D}+\|f^\prime\|_{L^\infty}\|f^\prime\|_{D}\Big)\\
  &\le&C\Big(\|f^\prime(t)\|_{X}+\|f^\prime(t)\|_{X}^2\Big)
 \end{eqnarray}
 which implies that
\begin{equation}\label{ama56}
 V(t)\leq C t\Big(\|f^\prime\|_{L^\infty_tX}+\|f^\prime\|_{L^\infty_tX}^2\Big).
\end{equation}
 Using Proposition \ref{prop10} we obtain
 \begin{equation}\label{ama57}
 \|(F-G)(t)\|_X\le C\Big(\|f^\prime(t)\|_{X}+\|f^\prime(t)\|_{X}^{17}\Big).
\end{equation}
Plugging \eqref{ama56} and \eqref{ama57} into \eqref{Tham1} we obtain
$$
\|f^\prime\|_{L^\infty_T X}\leq e^{C T\big(\|f^\prime\|_{L^\infty_TX}+\|f^\prime\|_{L^\infty_TX}^2\big)}\Big[\|f_0^\prime\|_X+ T\big(\|f^\prime\|_{L^\infty_TX}+\|f^\prime\|_{L^\infty_TX}^{17}\big)\Big].
$$
This shows the existence of small $T$ depending only on $\| f_0^\prime\|_X$ and such that
$$
\|f^\prime\|_{L^\infty_T X}\le 2 \| f_0^\prime\|_X,
$$
which  ends the proof of the proposition.
  \end{proof}
  
  \subsection{Scheme construction of the solutions}\label{Const778}
  This section is devoted to the construction of the solutions to \eqref{graphder} in short time. Before giving a precise description about the method used here and based on a double regularization, let us explain the big lines of the strategy. The a priori estimates developed in the previous sections require some rigid properties like the confinement of the support, the positivity of the solution and  some nonlinear effects in order to control some singular terms as it was mentioned in Theorem \ref{th1}. So it appears so hard to find a linear scheme that respects all of  those constraints. The idea is to proceed with a nonlinear  double regularization scheme. First, we   fix a small parameter $\varepsilon>0$ used to regularize the singularity of the kernels around the origin,   and  second we elaborate an iterative nonlinear  scheme giving rise to  a family of solutions $(f_n^\varepsilon)_n$ that may  violate some of the mentioned constraints. With  this scheme we are able to derive  a priori estimates uniformly  with respect to $n$ during a short time $T_\varepsilon>0$, but this time  may shrink to zero as $\varepsilon$ goes to zero.  By compactness arguments we prove that this approximate solutions $(f_n^\varepsilon)_n$ converges as $n$ goes to infinity to a solution $f^\varepsilon$ living in our function space during the time interval $[0, T_\varepsilon]$. Now the function  $f^\varepsilon$ satisfies a modified nonlinear problem but the important fact is that all the a priori estimates developed in the preceding sections  hold uniformly  on $ \varepsilon$. This allows by a classical procedure  to implement the  bootstrap argument and prove that the family $(f^\varepsilon)_\varepsilon$ is actually  defined on some time interval $[0,T]$ independently on $\varepsilon.$ To conclude it remains to pass to the limit when $\varepsilon$ goes to zero and this allows to construct a solution for our initial problem.
  
  Let us now give more details about this double scheme regularization.  Consider the iterative scheme
  \begin{equation}\label{scheme}
\left\lbrace
\begin{array}{l}
\partial_t f_{n+1}^\EE+u_1^{\E}(f_n^\E)\partial_x f_{n+1}^\E=u_2^{\E}(f_{n+1}^\E), \, n\in\N, \\
f_0^\E(t,x)=f_0(x)\\
f_{n+1}^\E(0,x)= f_0(x)
\end{array}
\right.
\end{equation}
    with
  \begin{equation*}
 u_1^{\E}(g)(t,x)\triangleq\frac{1}{2\pi}\chi(x)\bigintsss_{|y|\geq\E}{{\chi}(y)\Bigg\{\arctan\Bigg(\frac{g(t,x+y)-g(t,x)}{y}\Bigg)+\arctan\Bigg(\frac{g(t,x+y)+g(t,x)}{y}\Bigg)\Bigg\}} dy
\end{equation*}
\begin{equation}\label{fieldsX6}
u_2^{\E}(g)(t,x)\triangleq\frac{1}{4\pi}\bigintsss_{|y|\geq\E}{\chi(y) \log\Bigg(\frac{y^2+(g(t,x+y)-g(t,x))^2}{y^2+(g(t,x+y)+g(t,x))^2} \Bigg)}dy.
\end{equation} 
The function $\chi$ is a positive smooth cut-off function taking the value $1$ on some interval $[-M,M]$ such that
$$
K_0, K_0-K_0\subset[-M,M]
$$  
with $K_0$ being the convex hull of  $\textnormal{supp }f_0$. The function $\chi$ is introduced in order to guarantee the convergence of the integrals. We shall see later by  using the support structure of the solutions that one can in fact remove this cut-off function.  Denote by
$$
\mathcal{E}_T=\Big\{f;\, f\in L^\infty([0,T]\times\R), f^\prime\in L^\infty([0,T],X)\Big\}
$$
equipped with the norm
$$
\|f\|_{\mathcal{E}_T}=\|f\|_{L^\infty([0,T]\times\R)}+\|\partial_xf\|_{ L^\infty([0,T],X)}
$$
where $X$ denotes  Dini space $C^\star$  or H\"{o}lder spaces $C^s(\R), 0<s<1$ and for the simplicity we shall during this part work only with H\"{o}lder space. We intend to explain the approach without giving all the details, because some of them are classical. Using the characteristics method, one can transform the equation \eqref{scheme} into a  fixed pint problem
$$
f_{n+1}=\mathcal{N}_n^\E(f_{n+1})\quad\hbox{with}\quad \mathcal{N}(f)(t,x)=f_0\big(\psi_{n,\EE}^{-1}(t,x)\big)+\int_0^t u_{2}^{\EE}(f)\Big(\tau,\psi_{n,\EE}\big(\tau,\psi_{n,\EE}^{-1}(t,x)\big)\Big)d\tau
$$
with $\psi_{n,\EE}$ being the one-dimensional flow associated to  $u_1^{n}(f_n^\E)$, that is, the solution of the ODE
\begin{equation}\label{FFlow1}
\psi_{n,\EE}(t,x)=x+\int_0^tu_1^{n}(f_n^\E)\big(\tau,\psi_{n,\EE}(\tau,x)\big) d\tau.
\end{equation}
It is plain that
$$
\|\mathcal{N}(f)(t)\|_{L^\infty}\leq \|f_0\|_{L^\infty}+\int_0^t\|u_{2}^{\EE}(f)(\tau)\|_{L^\infty}d\tau.
$$
Applying the elementary inequality: for $a>0, b,c\in\R_+$ 
$$\Big|\log\big(\frac{a+b}{a+c}\big)\Big|\leq \frac{b+c}{a}
$$
we get from  \eqref{fieldsX6} that
\begin{eqnarray*}
|u_{2}^{\EE}(f)(t,x)|&\leq& \frac{1}{4\pi}\int_{|y|\geq\E}\chi(y)\frac{f^2(t,x+y)+f^2(t,x)}{y^2}dy\\
&\le& C\E^{-2}\|f(t)\|_{L^\infty}^2.
\end{eqnarray*}
It follows that
\begin{equation}\label{FirstX1}
\|\mathcal{N}(f)\|_{L^\infty_TL^\infty}\leq \|f_0\|_{L^\infty}+C\E^{-2} T\|f\|_{L^\infty_TL^\infty}^2.
\end{equation}
We shall move to the estimate of $\|\partial_x \mathcal{N}(f)\|_{L^\infty_T X}$. Let us first start with the  estimate of $\|\partial_x\{f_0(\psi_{n,\EE}^{-1}\}\|_{L^\infty_T X}$. By straightforward computations using law products \eqref{lawX3}, composition laws \eqref{comp1} in H\"{o}lder spaces and the following classical estimates on the flow, 
$$
\|\partial_x\psi_{n,\EE}^{\pm1}\|_{L^\infty_TX}\leq Ce^{C\|\partial_x(u_1^{\E}(f_n^\E))\|_{L^1_TL^\infty} }\Big(1+\|\partial_x(u_1^{\E}(f_n^\E))\|_{L^1_TX}\Big)
$$
one gets
\begin{eqnarray*}
\|\partial_x\{f_0(\psi_{n,\EE}^{-1}\}\|_{L^\infty_TX}&\leq&\|\{\partial_x f_0\}(\psi_{n,\EE}^{-1})\|_{L^\infty_TX}\|\partial_x\psi_{n,\EE}^{-1}\|_{L^\infty_TX}\\
&\leq&
C\|\partial_xf_0\|_{X} e^{C\|\partial_x(u_1^{\E}(f_n^\E))\|_{L^1_TL^\infty} }\Big(1+\|\partial_x(u_1^{\E}(f_n^\E))\|_{L^1_TX}\Big).
\end{eqnarray*}
Differentiating the expression of $u_1^{\E}(f_n^\E))$ in \eqref{fieldsX6}  and making standard estimates we get easily
\begin{eqnarray*}
\|\partial_x\{u_1^{\E}(f_n^\E)(t)\}\|_{X}&\le& C+C\E^{-1}\|\partial_xf_n^\E(t)\|_{X}+C\E^{-3}\|\partial_xf_n^\E(t)\|_{L^\infty}\|f_n^\E(t)\|_{X}^2\\
&\le&C+C\E^{-1}\|f_n^\E\|_{\mathcal{E}_T}+C\E^{-3}\|f_n^\E\|_{\mathcal{E}_T}^3,
\end{eqnarray*}
where we have used
$$
\Big\|\frac{1}{y^2+f^2}\Big\|_{X}\leq C \|f\|_X^2 y^{-4}.
$$
Therefore
\begin{eqnarray}\label{SSin1}
\|\partial_x\{f_0(\psi_{n,\EE}^{-1}\}\|_{L^\infty_TX}&\leq&C\|\partial_xf_0\|_{X} e^{CT+C\E^{-1}T\|f_n^\E\|_{\mathcal{E}_T}+C\E^{-3}T\|f_n^\E\|_{\mathcal{E}_T}^3 }
\end{eqnarray}
and
\begin{equation}\label{Sflow1}
\|\partial_x\psi_{n,\EE}^{\pm1}\|_{L^\infty_TX}\leq Ce^{CT+CT\E^{-1}\|f_n^\E\|_{\mathcal{E}_T}+CT\E^{-3}\|f_n^\E\|_{\mathcal{E}_T}^3}.
\end{equation}
Similarly we get
\begin{eqnarray*}
\|\partial_x\{u_2^{\E}(f)\}\|_{L^\infty_TX}&\leq& C\E^{-2}\|\partial_xf\|_{L^\infty_TX}\|f\|_{L^\infty_TX}+C\E^{-4}\|\partial_xf\|_{L^\infty_TL^\infty}\|f\|_{L^\infty_TL^\infty}\|f\|_{L^\infty_TX}^2\\
&\leq& C\E^{-2}\|f\|_{\mathcal{E}_T}^2+C\E^{-4}\|f\|_{\mathcal{E}_T}^4.
\end{eqnarray*}

Combining this estimate with law products and \eqref{Sflow1} we deduce that
\begin{eqnarray*}
\|\partial_x\big\{u_{2}^{\EE}(f)\big(\tau,\psi_{n,\EE}(\tau,\psi_{n,\EE}^{-1})\big)\big\}\|_{X}\leq  C\big(\E^{-2}\|f\|_{\mathcal{E}_T}^2+\E^{-4}\|f\|_{\mathcal{E}_T}^4\big)e^{CT+CT\E^{-1}\|f_n^\E\|_{\mathcal{E}_T}+CT\E^{-3}\|f_n^\E\|_{\mathcal{E}_T}^3}.
\end{eqnarray*}
Putting together this estimate with \eqref{SSin1} we find that
$$
\|\partial_x \mathcal{N}(f)\|_{L^\infty_T X}\le C\Big(\|\partial_xf_0\|_{X}+T\E^{-2}\|f\|_{\mathcal{E}_T}^2+T\E^{-4}\|f\|_{\mathcal{E}_T}^4\Big)e^{CT+CT\E^{-1}\|f_n^\E\|_{\mathcal{E}_T}+CT\E^{-3}\|f_n^\E\|_{\mathcal{E}_T}^3}$$
which yields in view of \eqref{FirstX1}
$$
\| \mathcal{N}(f)\|_{\mathcal{E}_T}\le C\Big(\|f_0\|_{L^\infty}+\|\partial_xf_0\|_{X}+T\E^{-2}\|f\|_{\mathcal{E}_T}^2+T\E^{-4}\|f\|_{\mathcal{E}_T}^4\Big)e^{CT+CT\E^{-1}\|f_n^\E\|_{\mathcal{E}_T}+CT\E^{-3}\|f_n^\E\|_{\mathcal{E}_T}^3}.
$$
We can assume that $0<T\leq1$  and then 
$$
\| \mathcal{N}(f)\|_{\mathcal{E}_T}\le C\Big(\|f_0\|_{L^\infty}+\|\partial_xf_0\|_{X}+T\E^{-4}\|f\|_{\mathcal{E}_T}^4\Big)e^{CT\E^{-3}\|f_n^\E\|_{\mathcal{E}_T}^3}.
$$
Consider now the closed  ball 
$$
B=\Bigg\{f\in \mathcal{E}_T, \|f\|_{\mathcal{E}_T}\leq 2C\Big(\|f_0\|_{L^\infty}+\|\partial_xf_0\|_{X}\Big)e^{CT\E^{-3}\|f_n^\E\|_{\mathcal{E}_T}^3}\Bigg\},
$$
then if we choose $T$ such that 
\begin{equation}\label{Condd1}
16C^3\E^{-4}T\Big(\|f_0\|_{L^\infty}+\|\partial_xf_0\|_{X}\Big)^3e^{5CT\E^{-3}\|f_n^\E\|_{\mathcal{E}_T}^3}\leq1
\end{equation}
then $\mathcal{N}:B\to B$ is well-defined and proceeding as before we can show under this condition that it is also a contraction. This implies the existence in this ball of a unique solution   to the fixed point problem and so one can construct a solution  $f_{n+1}^\E\in \mathcal{E}_T$ to \eqref{scheme} and we have the estimates
$$
 \forall n\in\N,\,\|f_{n+1}^\E\|_{\mathcal{E}_T}\leq 2C\Big(\|f_0\|_{L^\infty}+\|\partial_xf_0\|_{X}\Big)e^{CT\E^{-3}\|f_n^\E\|_{\mathcal{E}_T}^3}.
 $$
Now we select $T$ such that it satisfies also
\begin{equation}\label{Condd2}
64C^4\Big(\|f_0\|_{L^\infty}+\|\partial_xf_0\|_{X}\Big)^3 T\E^{-3}\le\ln2
\end{equation}
then we get the uniform estimates
$$
\forall n\in\N, \quad \|f_n\|_{\mathcal{E}_T}\leq  4C\Big(\|f_0\|_{L^\infty}+\|\partial_xf_0\|_{X}\Big).
$$
In order to satisfy mutually the conditions \eqref{Condd1} and  \eqref{Condd2} it suffices to take 
\begin{equation}\label{T_E}
T_\E:=C_0\E^2
\end{equation}
with $C_0$ depending only on $\|f_0\|_{L^\infty}+\|\partial_xf_0\|_{X}$ such that 
\begin{equation}\label{uniformX1}
\forall n\in\N,\quad \|f_n\|_{\mathcal{E}_T}\leq  4C\Big(\|f_0\|_{L^\infty}+\|\partial_xf_0\|_{X}\Big).
\end{equation}
Now we shall check that we can remove the localization in space through the cut-off function $\chi$. To do so, it suffices to get suitable information on the support  of $(f_n^\E)$. We shall prove that 
\begin{equation}\label{suppz1}
\forall n\in \N,\quad \textnormal{supp }f_{n}^\EE(t)\subset K_0
\end{equation}
where $K_0$ is the convex hull of the support of $f_0$. Before giving the proof let us assume for a while this property and see  how to get rid of the localizations in the velocity fields. From the expression of $u_2^\EE(f_{n+1}^\E)$ one has
\begin{eqnarray*}
 u_2^{\E}(f_{n+1}^\E)(t,x)&=&\frac{1}{4\pi}\bigintsss_{|y|\geq\E}\tiny{\log\Bigg(\frac{y^2+(f_{n+1}^\E(t,x+y)-f_{n+1}^\E(t,x))^2}{y^2+(f_{n+1}^\E(t,x+y)+f_{n+1}^\E(t,x))^2} \Bigg)}dy\\
 &-&\frac{1}{4\pi}\bigintsss_{|y|\geq\E}\tiny{[1-\chi(y)] \log\Bigg(\frac{y^2+(f_{n+1}^\E(t,x+y)-f_{n+1}^\E(t,x))^2}{y^2+(f_{n+1}^\E(t,x+y)+f_{n+1}^\E(t,x))^2} \Bigg)}dy.
\end{eqnarray*} 
Since $\forall\, x\notin  K_0$ we have $f_{n+1}(t,x)=0$ and hence $u_2^\EE(f_{n+1}^\E)(t,x)=0$. Then $\textnormal{supp }u_2^\EE(f_{n+1}^\E)(t)\subset  K_0$. Thus $\forall\, x\in K_0$
\begin{eqnarray*}
&&\bigintsss_{|y|\geq\E}\tiny{[1-\chi(y)] \log\Bigg(\frac{y^2+(f_{n+1}^\E(t,x+y)-f_{n+1}^\E(t,x))^2}{y^2+(f_{n+1}^\E(t,x+y)+f_{n+1}^\E(t,x))^2} \Bigg)}dy=\\&&\bigintsss_{\{|y|\geq|\E\}\cap K_0-K_0}\tiny{[1-\chi(y)] \log\Bigg(\frac{y^2+(f_{n+1}^\E(t,x+y)-f_{n+1}^\E(t,x))^2}{y^2+(f_{n+1}^\E(t,x+y)+f_{n+1}^\E(t,x))^2} \Bigg)}dy=0
\end{eqnarray*} 
because $\chi=1$ on $K_0-K_0$. Now we claim that in the advection term  $ u_1^{\E}(f_{n+1}^\E)(t,x)\partial_x f_{n+1}^\EE$ of the equation \eqref{scheme} one can remove the cut-off function. Since $\partial_x f_{n+1}^\EE=0$ outside $K_0$ then one gets immediately  $\chi(x)\partial_x f_{n+1}^\EE=\partial_x f_{n+1}^\EE$. Similarly one has

\begin{eqnarray*}
\nonumber u_1^{\E}(g)(t,x)&\triangleq&\frac{1}{2\pi}\bigintsss_{|y|\geq\E}\tiny{\Bigg\{\arctan\Bigg(\frac{g(t,x+y)-g(t,x)}{y}\Bigg)+\arctan\Bigg(\frac{g(t,x+y)+g(t,x)}{y}\Bigg)\Bigg\}} dy\\
&-&\frac{1}{2\pi}\bigintsss_{|y|\geq\E}\tiny{{(1-\chi}(y))\Bigg\{\arctan\Bigg(\frac{g(t,x+y)-g(t,x)}{y}\Bigg)+\arctan\Bigg(\frac{g(t,x+y)+g(t,x)}{y}\Bigg)\Bigg\}} dy
\end{eqnarray*} 
and for  $x\in K_0$ it is clear that
\begin{eqnarray*}
&&\bigintsss_{|y|\geq\E}\tiny{{(1-\chi}(y))\Bigg\{\arctan\Bigg(\frac{g(t,x+y)-g(t,x)}{y}\Bigg)+\arctan\Bigg(\frac{g(t,x+y)+g(t,x)}{y}\Bigg)\Bigg\}} dy=\\
&&\bigintsss_{\{y|\geq\E\}\cap K_0-K_0}\tiny{{(1-\chi}(y))\Bigg\{\arctan\Bigg(\frac{g(t,x+y)-g(t,x)}{y}\Bigg)+\arctan\Bigg(\frac{g(t,x+y)+g(t,x)}{y}\Bigg)\Bigg\}} dy\\
&&\qquad\qquad = 0.
\end{eqnarray*} 
Now let us come back  to the proof of \eqref{suppz1} and provide further qualitative properties.
Now similarly to the identity \eqref{maximum} one obtains
$$
u_2^{n+1,\E}(t,x)=f_{n+1}^\E(t,x)\big(1+U_{n+1,\E}(t,x)\big),\quad  \|U_{n+1,\E}(t)\|_{L^\infty}\leq C(1+\|f_{n+1}(t)\|_{D}^6\big).
$$
So following the same line of the proof of  Proposition \ref{prop20} we get a similar formulae to \eqref{flot2} which implies that the positivity result,
$$
 f_{n+1}(t,x)\geq0
$$ 
where we have used in particular that the initial data $f_{n+1}^{\EE}(0,x)=f_0(x)\geq0$. Thus we obtain
$$
\forall n\in \N,\quad  f_{n}(t,x)\geq0.
$$ 
 As $u_{n,\EE}^2(t,x)\leq0$ then following the same proof of Proposition \ref{prop20} we get the maximum principle
$$
\forall \, n\in\N,\quad \|f_{n}^\EE(t)\|_{L^\infty}\leq \|f_0\|_{L^\infty}.
$$
The proof of the confinement of the support \eqref{suppz1} follows exactly  the same lines of the proof of \mbox{Proposition \ref{prop20}-$(3)$}. 
Now we shall study the strong convergence of the sequence $(f_n^\E)_n$. Set
$$
\theta_n^\E(t,x):=f_{n+1}(t,x)-f_n(t,x).
$$
Then
$$
\partial_t\theta_{n+1}^\E+u_1^{\E}(f_{n+1}^\E)\partial_x\theta_{n+1}^\E=-\big[u_1^{\E}(f_{n+1}^\E)-u_1^{\E}(f_{n}^\E)\big]\partial_xf_{n+1}^\E+u_2^{\E}(f_{n+2}^\E)-u_2^{\E}(f_{n+1}^\E)
$$
According to the mean value theorem one has for $a>0, x, y\in\R$
$$
|\arctan(x)-\arctan(y)|\leq |x-y|\quad\hbox{and}\quad |\log(a+|x|)-\log(a+|y|
)|\leq |x-y| a^{-1}$$
which imply that
\begin{eqnarray*}
\big\|u_1^{\E}(f_{n+1}^\E)(t)-u_1^{\E}(f_{n}^\E)(t)\big\|_{L^\infty}&\leq &C\|f_{n+1}^\E(t)-f_{n}^\E(t)\|_{L^\infty} \int_{|y|\geq\E}\frac{\chi(y)}{|y|}dy\\
&\leq&C\E^{-1}\|f_{n+1}^\E(t)-f_{n}^\E(t)\|_{L^\infty}.
\end{eqnarray*}
Similarly, we obtain
\begin{eqnarray*}
\big\|u_2^{\E}(f_{n+1}^\E)(t)-u_2^{\E}(f_{n}^\E)(t)\big\|_{L^\infty}&\leq &C\E^{-2}\big(\|f_{n+1}^\E(t)\|_{L^\infty}+\|f_{n}^\E(t)\|_{L^\infty}\big)\|f_{n+1}^\E(t)-f_{n}^\E(t)\|_{L^\infty}.
\end{eqnarray*}
Using the uniform estimates \eqref{uniformX1} we get for any $t\in[0,T_\E]$
\begin{eqnarray*}
\big\|u_2^{\E}(f_{n+1}^\E)(t)-u_2^{\E}(f_{n}^\E)(t)\big\|_{L^\infty}&\leq &C\|f_0^\prime\|_X\E^{-2}\|f_{n+1}^\E(t)-f_{n}^\E(t)\|_{L^\infty},\quad \|\partial_xf_{n+1}^\E(t)\|_{L^\infty}\leq C\|f_0^\prime\|_X.
\end{eqnarray*}
 Using the maximum principle for transport equation allows to get for any $t\in [0,T_\E]$
\begin{eqnarray*}
\big\|\theta_{n+1}(t)\big\|_{L^\infty}&\leq &C\E^{-2}\|f_0^\prime\|_X\|\int_0^t\Big[\big\|\theta_{n+1}(\tau)\big\|_{L^\infty}+\big\|\theta_{n}(\tau)\big\|_{L^\infty}\Big]d\tau.
\end{eqnarray*}
By virtue of Gronwall lemma  one finds that    for any $t\in [0,T_\E]$
\begin{eqnarray*}
\big\|\theta_{n+1}(t)\big\|_{L^\infty}&\leq &e^{C\E^{-2}\|f_0^\prime\|_X T_\E}\int_0^t\big\|\theta_{n}(\tau)\big\|_{L^\infty}d\tau.
\end{eqnarray*}
 Hence we obtain in view of \eqref{T_E}
 \begin{eqnarray*}
\big\|\theta_{n+1}(t)\big\|_{L^\infty}&\leq &C_0\int_0^t\big\|\theta_{n}(\tau)\big\|_{L^\infty}d\tau.
\end{eqnarray*}
By induction we find
$$
\forall n\in \N,\,\forall t\in [0,T_\E],\quad \big\|\theta_{n}\big\|_{L^\infty_{t}L^\infty}\leq C_0^n \frac{t^n}{n!}\|\theta_0\|_{L^\infty_{t}L^\infty}.
$$
This implies the convergence of the series
$$
\sum_{n\in\N}\big\|\theta_{n+1}\big\|_{L^\infty_{T_\E}L^\infty}<\infty.
$$
Therefore $(f_n^\E)_n$ converges strongly in $L^\infty_{T_\E}L^\infty$ to an element $f^\E\in L^\infty_{T_\E}L^\infty$. From the uniform estimates  \eqref{uniformX1}  we deduce that $f^\E\in \mathcal{E}_{T_\E}$. This allows to pass to the limit in the equation \eqref{scheme} and obtain that $f^\E$ is solution to 
  
  \begin{equation}\label{scheme01}
\left\lbrace
\begin{array}{l}
\partial_t f^\EE+u_1^{\E}(f^\E)\partial_x f^\E=u_2^{\E}(f^\E), \\
f_0^\E(t,x)=f_0(x)
\end{array}
\right.
\end{equation}
    with
  \begin{eqnarray}\label{fieldsX06}
\nonumber u_1^{\E}(f^\EE)(t,x)&\triangleq&\frac{1}{2\pi}\bigintsss_{|y|\geq\E}\tiny{\Bigg\{\arctan\Bigg(\frac{f^\EE(t,x+y)-f^\EE(t,x)}{y}\Bigg)+\arctan\Bigg(\frac{f^\EE(t,x+y)+f^\EE(t,x)}{y}\Bigg)\Bigg\}} dy\\
 && u_2^{\E}(f^\EE)(t,x)\triangleq\frac{1}{4\pi}\bigintsss_{|y|\geq\E}\tiny{\log\Bigg(\frac{y^2+(f^\EE(t,x+y)-f^\EE(t,x))^2}{y^2+(f^\EE(t,x+y)+f^\EE(t,x))^2} \Bigg)}dy.
\end{eqnarray} 
Now, looking to  the proofs used to get  the a priori estimates, they  can be adapted to the equation \eqref{scheme01} supplemented with \eqref{fieldsX06}. For instance  the a priori estimates obtained in Proposition \ref{prop20} hold for the modified equation  \eqref{scheme01} independently on vanishing $\E$. In particular  one can bound uniformly in $\E$ the solution $f^\E$ in the space $X_{T_\E}$ and therefore $T_\E$ is not maximal and by a standard bootstrap argument we can continue the solution  up to the local  time $T$  constructed  in \mbox{Proposition \ref{prop20}.} It follows that $f^\E$ belongs to $ \mathcal{E}_T$ uniformly with respect to small $\E$. This  yields according once again to the Proposition \ref{prop20} and the inequalities \eqref{maximum} and \eqref{zerda}
\begin{eqnarray*}
\sup_{\varepsilon\in [0,1]}\|\partial_t f^\E\|_{L^\infty_TL^\infty}&\le&\|u_1^{\E}(f^\E)\|_{L^\infty_TL^\infty}\|\partial_x f^\E\|_{L^\infty_TL^\infty}+\|u_2^{\E}(f^\E)\|_{L^\infty_TL^\infty}\\
&\leq&C_0,
\end{eqnarray*}
and $C_0$ is a constant depending on the size of the initial data.
Now from the compact \mbox{embedding $C_K^s\to C_b$} and Ascoli lemma we deduce that up to a sequence $(f^\E)$ converges strongly  in $L^\infty_TL^\infty$ to some element $f$ which belongs in turn to $\mathcal{E}_T$. This allows to pass to the limit in \eqref{scheme01} and \eqref{fieldsX06} and find a solution to the initial value problem \eqref{scheme}. We point out that by working more one may  obtain the strong convergence of the full sequence $(f^\E)$ to $f$. Note finally that the uniqueness follows easily  from    the arguments used to prove that $(\theta_n)$ is a Cauchy sequence.

  \section{Global well-posedness}
We are concerned here with the global existence of  strong  solutions already constructed in Theorem \ref{thm1}.  This will be established under a smallness condition on the initial data and it is probable  that for arbitrary large initial data the graph structure might be destroyed   in finite time. The basic ingredient  which allows to balance the energy amplification during the time evolution  is a damping effect generated by the source terms. Note that this damping effect  is plausible  from  the graph equation  \eqref{graph1} according to the identity  \eqref{VY1}. However, as we shall see in the next section, it is quite complicate  to extend  this behavior for higher regularity  at the level of the resolution space due to the existence of linear part  in  the source term governing the motion of the slope \eqref{graphder}. This part could in general  bring  an amplification  in time of  the energy. To circumvent this difficulty we establish a weakly dissipative property of the linearized operator associated to  the source term that we combine with  the time decay of the  solution for weak regularity using an interpolation argument. 
  \subsection{Weak and strong damping   behavior of the source term} 
  Note from Proposition \ref{prop10} that $F$ does not contribute at the linear level which is not the case of  the functional  $G$. We shall prove that actually there is no linear contribution for $G$. This will be done by establishing a  dissipative property that  occurs at least at the linear level.  This is described by the following proposition.
  \begin{proposition}\label{prop30}
   Let $K$ be a compact set of $\R$ and $s\in (0,1)$. Then for any $f\in C^s_K$ we have the decomposition 
   
   $$
   G(x)=2\pi f^\prime(x)+ {L}(x)+{N}(x)
   $$
   with
 $$
\|{L}\|_s\leq2\pi\big( \|f^\prime\|_s+2\|f^\prime\|_{L^\infty}\big)+C\|f^\prime\|_{L^\infty}^s\|f^\prime\|_s\quad \hbox{and}\quad  \|N\|_{s}\le C \|f^\prime\|_{D}^{\frac13}\Big(\|f^\prime\|_s+\|f^\prime\|_s^{16}\Big),
 $$
 where $C>0$ is a constant depending only on $K$. Moreover,
 $$
 \|L\|_{L^\infty}\le C\min\big(\|f\|_{L^\infty}^s\|f^\prime\|_s,\|f^\prime\|_{L^\infty}\big)\quad \hbox{and}\quad \|N\|_{L^\infty}\leq C\|f^\prime\|_{L^\infty}\Big(\|f^\prime\|_{D}+\|f^\prime\|_{D}^3\Big).
 $$
 \end{proposition}
 \begin{proof}
 In view of \eqref{daca1},\eqref{daca2}, \eqref{dgh1}, \eqref{Tata13}, \eqref{Tata14} and \eqref{Tata17} one gets
 $$
 G(x)=G_{11}(x)+H(x), \quad H=G_{12}+G_2
 $$
 with 
 \begin{equation}\label{H890}
 \|H\|_s\leq C\|f^\prime\|_{D}^{\frac13}\Big(\|f^\prime\|_s+\|f^\prime\|_s^{16}\Big).
 \end{equation}
 Note also that from \eqref{kar1} and \eqref{kar8} we get
 \begin{equation}\label{H870}
 \|H\|_{L^\infty}\leq C\|f^\prime\|_{L^\infty}\Big(\|f^\prime\|_s+\|f^\prime\|_s^{3}\Big).
 \end{equation}
 Now from \eqref{Expbb} we get
 \begin{equation*}
         \nonumber G_{11}(x)=2\int_{\R}\frac{f^\prime(x)+f^\prime(x+f(x)z) }{   \varphi(x,z)} dz  
           \end{equation*}  
with 
       $$
       \varphi(x,z)=z^2+\bigg(2+z \int_0^1f^\prime\big(x+\theta f(x) z\big)d\theta\bigg)^2.
       $$
We shall split again $G_{11}$ as follows
 \begin{eqnarray*}
G_{11}(x)&=&2\int_{\R}\frac{f^\prime(x)+f^\prime\big(x+f(x)z\big) }{  z^2+4} dz\\
&-&2\int_{\R}\frac{\big[f^\prime(x)+f^\prime\big(x+f(x)z\big)\big]\psi(x,z)}{ \varphi(x,z)( z^2+4)} dz\\
&\triangleq&\mathcal{L}(x)+\mathcal{N}(x),
 \end{eqnarray*}
 with
 $$
 \psi(x,z)\triangleq 4z \int_0^1f^\prime\big(x+\theta f(x) z\big)d\theta+z^2\Big( \int_0^1f^\prime\big(x+\theta f(x) z\big)d\theta\Big)^2.
 $$
 From  \eqref{Eqw1} one gets
 \begin{equation}\label{kar5}
 \|\mathcal{N}\|_{L^\infty}\leq C\|f^\prime\|_{L^\infty}^2\big(1+\|f^\prime\|_{L^\infty}^3\big).
 \end{equation}
 Using  the  law product \eqref{lawX3} we get
  \begin{eqnarray*}
\Bigg\|\frac{\big[f^\prime+f^\prime\circ\big(\textnormal{Id}+zf\big)\big]\psi(\cdot,z)}{ \varphi(\cdot,z)}\Bigg\|_s &\le& 2\|f^\prime\|_{L^\infty}\|\psi(\cdot,z)\|_{L^\infty}\big\|1/\varphi(\cdot,z)\big\|_s\\
&+&2\|f^\prime\|_{L^\infty}\|\psi(\cdot,z)\|_s\|1/\varphi(\cdot,z)\|_{L^\infty}\\
&+&\big[\|f^\prime\|_s+\|f^\prime\circ\big(\textnormal{Id}+zf\big)\|_s\big] \|\psi(\cdot,z)\|_{L^\infty} \|1/\varphi(\cdot,z)\|_{L^\infty}.
\end{eqnarray*}
In addition, it is clear that
$$
 \|\psi(\cdot,z)\|_{L^\infty} \le 4|z|\|f^\prime\|_{L^\infty}+|z|^2\|f^\prime\|_{L^\infty}^2.
$$
Performing the composition law  \eqref{comp1} we deduce that
\begin{eqnarray*}
\|\psi(\cdot,z)\|_s&\le& C|z|\|f^\prime\|_s\big[1+|z|^s\|f^\prime\|_{L^\infty}^s\big]+C|z|^2\|f^\prime\|_{L^\infty}\|f^\prime\|_s\big[1+|z|^s\|f^\prime\|_{L^\infty}^s\big].
\end{eqnarray*}
Combining this latter estimate  with \eqref{dezz1} and \eqref{Eqw1} yields

\begin{eqnarray*}
\bigg\|\frac{\big[f^\prime+f^\prime\circ\big(\textnormal{Id}+zf\big)\big]\psi(\cdot,z)}{ \varphi(\cdot,z)}\bigg\|_s \le  C\|f^\prime\|_{L^\infty}\|f^\prime\|_s\big[1+\|f^\prime\|_{L^\infty}^{7+s}\big]\big(1+|z|^s\big).
\end{eqnarray*}
Hence we get according to the embedding $C^s_K\hookrightarrow L^\infty$
\begin{eqnarray*}
\big\|\mathcal{N}\big\|_s& \le&  C\|f^\prime\|_{L^\infty}\|f^\prime\|_s\big[1+\|f^\prime\|_{L^\infty}^{7+s}\big]\\
&\leq&C\|f^\prime\|_{L^\infty}^{\frac13}\big[\|f^\prime\|_s^{\frac53}+\|f^\prime\|_{s}^{\frac{26}{3}+s}\big]\\
&\leq& C\|f^\prime\|_{L^\infty}^{\frac13}\big[\|f^\prime\|_s+\|f^\prime\|_{s}^{10}\big].
\end{eqnarray*}
Setting $N=\mathcal{N}+H$ and combining the latter estimate with \eqref{H890} we find the desired estimate for $N$ stated in the proposition. Putting together  \eqref{H870} and \eqref{kar5} combined with Sobolev embedding we find
$$
\|N\|_{L^\infty}\leq C\|f^\prime\|_{L^\infty}\Big(\|f^\prime\|_{s}+\|f^\prime\|_{s}^4\Big).
$$

Coming back to $\mathcal{L}$ one may write
  \begin{eqnarray}\label{L67}
\nonumber\mathcal{L}(x)&=&4f^\prime(x)\bigintsss_{\R}\frac{1 }{  z^2+4} dz+2\bigintsss_{\R}\frac{f^\prime\big(x+f(x)z\big)-f^\prime(x) }{  z^2+4} dz\\
&\triangleq&2\pi f^\prime(x)+L(x).
\end{eqnarray}
To estimate $L$ in $C^s$ we simply write
\begin{eqnarray*}
\|L\|_s&\le& 2\bigintsss_{\R}\frac{\|f^\prime\circ\big(\textnormal{Id}+zf\big)\|_s+\|f^\prime\|_s }{  z^2+4} dz.
\end{eqnarray*}
Combined with \eqref{comp1} we find
\begin{eqnarray*}
\|f^\prime\circ\big(\textnormal{Id}+zf\big)\|_s&\le&\big(\|f^\prime\|_s+2\|f^\prime\|_{L^\infty}\big)\big(1+|z|\|f^\prime\|_{L^\infty}\big)^s\\
&\leq& \big(\|f^\prime\|_s+2\|f^\prime\|_{L^\infty}\big)\big(1+|z|^s\|f^\prime\|_{L^\infty}^s\big),
\end{eqnarray*}
where in the last line we have use the inequality: $\forall s\in (0,1), \forall x, y\geq0$ one has

$$
(x+y)^s\leq x^s+y^s.
$$
Using \eqref{Imbed1}, it follows that
\begin{eqnarray*}
\|L\|_s&\le&2\pi\big(\|f^\prime\|_s +2\|f^\prime\|_{L^\infty}\big)+ C\|f^\prime\|_s\|f^\prime\|_{L^\infty}^s.
\end{eqnarray*}
The estimate of  $L$ in $L^\infty$ is easier and one gets according to \eqref{L67},
\begin{eqnarray*}
|L(x)|&\le&2|f(x)|^s\|f^\prime\|_s\int_{\R}\frac{|z|^s }{  z^2+4} dz\\
&\le&C|f(x)|^s\|f^\prime\|_s.
\end{eqnarray*}
Therefore we obtain
\begin{equation*}
\|L\|_{L^\infty}\le C\|f\|_{L^\infty}^s\|f^\prime\|_s.
\end{equation*}
We point out that we have obviously 
$$
\|L\|_{L^\infty}\le 2\pi\|f^\prime\|_{L^\infty}.
$$
Therefore we find 
\begin{equation}\label{Tham761}
\|L\|_{L^\infty}\le C\min\big(\|f\|_{L^\infty}^s\|f^\prime\|_s,\|f^\prime\|_{L^\infty}\big).
\end{equation}
This achieved the proof of Proposition \ref{prop30}.
 \end{proof}
\subsection{Global a priori estimates}

The main goal of this section is to show how we may use the weakly damping effect of the source terms  stated in Proposition \ref{prop30} in order to get global a priori estimates when the initial data  is small enough. The basic result reads as follows.
\begin{proposition}\label{prop-glob}
Let $K$ be a compact set of $\R$ and $s\in (0,1)$. There exists a constant $\varepsilon>0$ such that
 if $\|f^\prime_0\|_s\leq \varepsilon$ then the equation \eqref{graph1} admits a unique  global solution
 $$
 f^\prime\in L^\infty(\R_+; C^s_K).
 $$
 Moreover, there exists a constant $C_0$ depending on the initial data such that
 $$
 \forall\, t\geq0,\quad \| f^\prime(t)\|_{L^\infty}\leq C_0e^{-t}.
 $$
\end{proposition}
\begin{proof}
According to  the decomposition of Proposition \ref{prop30} combined with the equation \eqref{graphder} and \eqref{Ham11} we get that $g=\partial_xf$ satisfies the equation
\begin{equation}\label{kar17}
\partial_tg(t,x)+u_1(t,x)\partial_1g(t,x)+g(t,x)=\mathcal{R}(t,x),\quad \mathcal{R}\triangleq\frac{1}{2\pi}(F-L-N).
\end{equation}
Using  Proposition \ref{prop10} and Proposition \ref{prop30}  combined with the \eqref{Imbed1} we find
\begin{eqnarray*}
\|\mathcal{R}\|_s&\le&  \|f^\prime\|_s+2\|f^\prime\|_{L^\infty}+C\|f^\prime\|_D\big(\|f^\prime\|_s+\|f^\prime\|_s^3\big)\\
&+&C\|f^\prime\|_{L^\infty}^s\|f^\prime\|_s+C \|f^\prime\|_{L^\infty}^{\frac13}\Big(\|f^\prime\|_s+\|f^\prime\|_s^{16}\Big).
\end{eqnarray*}
The embedding $C^{\frac{s}{2}}_K\subset C^\star_K$ combined with interpolation inequalities in H\"{o}lder spaces  yield
\begin{equation}\label{zerda2}
\|f^\prime\|_D\leq C\|f^\prime\|_{L^\infty}^{\frac12}\|f^\prime\|_{s}^{\frac12}.
\end{equation}
Set $s_0=\min(s,\frac13)$ then it is easy to get
\begin{equation}\label{Tham72}
\|\mathcal{R}\|_s\le \|f^\prime\|_s+2\|f^\prime\|_{L^\infty}+C \|f^\prime\|_{L^\infty}^{s_0}\Big(\|f^\prime\|_s+\|f^\prime\|_s^{16}\Big).
\end{equation}
Let $h(t,x)\triangleq g(t,\psi(t,x))$, where $\psi$ is the flow introduced in \eqref{flot1}. Then it is obvious that
$$
\partial_t h(t,x)+h(t,x)=\mathcal{R}(t,\psi(t,x)).
$$
This allows to deduce the following  Duhamel integral representation 
$$
e^tg(t,x)=g_0(\psi^{-1}(t,x))+\int_0^{t}e^\tau \mathcal{R}\big(\tau,\psi\big(\tau,\psi^{-1}(t,x)\big)\big) d\tau.
$$
Thus
$$
e^t\|g(t)\|_s\le \|g_0(\psi^{-1}(t))\|_s+\int_0^te^\tau\|\mathcal{R}\big(\tau,\psi\big(\tau,\psi^{-1}(t)\big)\big)\|_s d\tau.
$$
According to  \eqref{Tham01} and \eqref{comp1} we obtain
$$
 \|g_0(\psi^{-1}(t))\|_s\leq C\|g_0\|_s e^{V(t)}, V(t)=\int_0^{t}\|\partial_x u_1(\tau)\|_{L^\infty} d\tau
$$
and
\begin{eqnarray*}
\|\mathcal{R}\big(\tau,\psi\big(\tau,\psi^{-1}(t)\big)\big)\|_s&\le&\big[\|\mathcal{R}(\tau)\|_{s}+2\|\mathcal{R}(\tau)\|_{L^\infty}\big] e^{V(t)-V(\tau)}.
\end{eqnarray*}
Note that the estimate of $\mathcal{R}$ in $C^s$ has been already stated in  \eqref{Tham72}. However to get a suitable estimate in $L^\infty$ we use  Proposition \ref{prop10} and  Proposition \ref{prop30} combined with Sobolev  embedding,
\begin{eqnarray}\label{tun33}
\nonumber\|\mathcal{R}(t)\|_{L^\infty}&\leq& C\|f^\prime(t)\|_{L^\infty}\Big(\|f^\prime(t)\|_{D}+\|f^\prime(t)\|_{D}^3\Big)+C\min\big(\|f(t)\|_{L^\infty}^s\|f^\prime(t)\|_s,\|f^\prime(t)\|_{L^\infty}\big)\\
&\leq& C\big(\|f^\prime(t)\|_{L^\infty}+\|f^\prime(t) \|_{L^\infty}^s\big)\Big(\|f^\prime(t)\|_{s}+\|f^\prime(t)\|_{s}^3\Big)\\
\nonumber&\leq& C\|f^\prime(t)\|_{L^\infty}^{s_0}\Big(\|f^\prime(t)\|_{s}+\|f^\prime(t)\|_{s}^4\Big).
\end{eqnarray}
It follows that
\begin{eqnarray*}
\|\mathcal{R}\big(\tau,\psi\big(\tau,\psi^{-1}(t)\big)\big)\|_s&\le& \big(\|f^\prime(\tau)\|_s+2\|f^\prime(\tau)\|_{L^\infty}\big)e^{V(t)-V(\tau)}\\
&+&C\|f^\prime(\tau)\|_{L^\infty}^{s_0}\big(\|f^\prime(\tau)\|_s+\|f^\prime(\tau)\|_s^{16}\big)e^{V(t)-V(\tau)}.
\end{eqnarray*}
Set 
$
K(t)= e^{-V(t)} e^{t}\|f^\prime(t)\|_s$ and
$$
 S(t)=Ce^{t} e^{-V(t)}\Big[\|f^\prime(t)\|_{L^\infty}+\|f^\prime(t)\|_{L^\infty}^{s_0}\big(\|f^\prime(t)\|_s+\|f^\prime(t)\|_s^{16}\big)\Big]
$$
then
$$
K(t)\le C K(0)+\int_0^t K(\tau)d\tau+\int_0^t S(\tau) d\tau.
$$
By virtue of  Gronwall lemma we deduce that
$$
K(t)\leq Ce^{t}K(0)+\int_0^t e^{t-\tau}S(\tau) d\tau.
$$
This implies that
\begin{eqnarray}\label{JunX1}
\nonumber\|f^\prime(t)\|_s&\le&Ce^{V(t)}\|f^\prime_0\|_s+C e^{V(t)}\int_0^t\|f^\prime(\tau)\|_{L^\infty} d\tau\\
&+&e^{V(t)}\int_0^t \|f^\prime(\tau)\|_{L^\infty}^{s_0}\Big(\|f^\prime(\tau)\|_s+\|f^\prime(\tau)\|_s^{16}\Big)d\tau.
\end{eqnarray}
Combining  the following interpolation inequality
$$
\|f^\prime_t\|_{L^\infty}\leq C\|f_t\|_{L^1}^{\frac{s}{2+s}}\|f^\prime_t\|_{s}^{\frac{2}{2+s}},
$$
with Proposition \ref{prop20}-$(2)$ we obtain
\begin{equation}\label{zerda3}
\|f^\prime(t)\|_{L^\infty}\leq Ce^{-\frac{s}{2+s}t}\|f_0\|_{L^1}^{\frac{s}{2+s}}\|f^\prime(t)\|_{s}^{\frac{2}{2+s}}.
\end{equation}


Plugging this estimate into  \eqref{zerda}  we find
\begin{equation}\label{zarda88}
\|\partial_x u_1(t)\|_{L^\infty}\le Ce^{-\frac{s}{2+s}t}\|f_0\|_{L^1}^{\frac{s}{2+s}}\Big(\|f^\prime(t)\|_{s}^{\frac{2}{2+s}}+\|f^\prime(t)\|_{s}^{\frac{4+s}{2+s}}\Big).
\end{equation}
It is quite obvious  from \eqref{Imbed1} and the compactness of the support that  
$$
\|f_0\|_{L^1}\leq C\|f_0^\prime\|_{s}
$$
with $C$ a constant depending on the size of the support of $f_0.$ Set
$$
\rho(T)=\sup_{t\in[0,T]}\|f^\prime(t)\|_{s}
$$ 
then combining \eqref{JunX1} with \eqref{zerda3} and \eqref{zarda88} yields to 
$$
\rho(T)\leq C e^{C\|f_0^\prime\|_{s}^{\frac{s}{2+s}}\big([\rho(T)]^{\frac{2}{2+s}}+[\rho(T)]^{\frac{4+s}{2+s}}\big)}\mu(T)
$$
with
$$
\mu(T)=\|f_0^\prime\|_s+\|f_0^\prime\|_{s}^{\frac{s}{2+s}}[\rho(T)]^{\frac{2}{2+s}}+\|f_0^\prime\|_{s}^{\frac{ss_0}{2+s}}[\rho(T)]^{\frac{2s_0}{2+s}}\Big(\rho(T)+[\rho(T)|^{16}\Big)
$$
This implies the existence of small number $\varepsilon>0$ depending only on $C$ and therefore on the size of the support of $f_0$ such that if
\begin{equation}\label{GW1}
\|f_0^\prime\|_{s}\leq \varepsilon \Longrightarrow \forall T>0,\quad \rho(T)\leq \delta(\|f_0^\prime\|_{s})
\end{equation}
with $\lim_{x\to 0}\delta(x)=0$. 
 This gives the global a priori estimates. 
 
 What is left is to  establish the precise  time decay of $\|f^\prime(t)\|_{L^\infty}$ stated in Proposition \ref{prop-glob} . From the equation  \eqref{kar17} it is easy to establish the following estimate using the characteristic method,
\begin{equation}\label{tun44}
\|g(t)\|_{L^\infty}\leq e^{-t}\|g_0\|_{L^\infty}+\int_{0}^t e^{-(t-\tau)}\|\mathcal{R}(\tau)\|_{L^\infty}d\tau.
\end{equation} 
According to    \eqref{tun33} we obtain
 $$
 e^t\|f^\prime(t)\|_{L^\infty}\leq \|f_0^\prime\|_{L^\infty}+C \int_0^t e^\tau\|f^\prime(\tau)\|_{L^\infty}\big(\|f^\prime(\tau)\|_D+\|f^\prime\|_D^3\big)d\tau.
 $$
 
 Using Gronwall lemma we obtain
 $$
 e^t\|f^\prime(t)\|_{L^\infty}\leq \|f_0^\prime\|_{L^\infty} e^{W(t)},\quad W(t)= C\int_0^t \big(\|f^\prime(\tau)\|_D+\|f^\prime\|_D^3\big)d\tau.
 $$

Putting together  \eqref{zerda2} with \eqref{zerda3} we obtain
$$
\|f^\prime(t)\|_D\leq Ce^{-\frac{s}{4+2s}t}\|f_0\|_{L^1}^{\frac{s}{4+2s}}\|f^\prime(t)\|_{s}^{\frac{4+s}{4+2s}}.
$$
Hence we deduce from \eqref{GW1} that
$$
\forall t\geq0, W(t)\leq C_0
$$
and therefore
\begin{equation}\label{tik0}
\forall t\geq0,\quad \|f^\prime(t)\|_{L^\infty}\leq C_0 e^{-t} ,\quad \|f^\prime(t)\|_D\leq C_0e^{-\frac{s}{4+2s}t},
\end{equation}
for a  suitable constant $C_0$  depending on the initial data. Inserting these estimates into \eqref{tun33} we obtain
\begin{equation}\label{tikgh0}
\forall t\geq0,\quad \|\mathcal{R}(t)\|_{L^\infty}\leq C_0 e^{-t}.\end{equation}
 Since $f_t$ is compactly supported in a fixed compact then
\begin{equation}\label{tik1}
\forall t\geq0,\quad \|f(t)\|_{L^\infty}\leq C_1 e^{-t}.
\end{equation}
Finally, we point out  that all the constants involved in the preceding estimates and related to the support of $f_t$ are actually  independent of the time due to the fact that the support of $f_t$ is confined  in the convex hull of the support of the initial data, as it has been stated in Proposition \ref{prop20}-$(3)$.

\end{proof}

\section{Scattering and collapse to singular measure}
The aim of the last  section is to analyze and identify  the longtime  behavior of the global solutions stated  in Theorem \ref{thm2}. It attempts to investigate the time evolution of the following probability measure,\begin{eqnarray*}
 dP_t(x)&\triangleq& \frac{\rho(t,x)}{\|\rho_t\|_{L^1}}dA(x)\\
&=&e^t{\bf{1}}_{D_t}(x) dA(x)
\end{eqnarray*}
where $dA$ denotes the usual Lebesgue measure. Note that without loss of generality we have assumed in the last line  that $\|\rho_0\|_{L^1}=1.$ As we shall see this measure converges weakly  as $t$ goes to infinity to a probability measure concentrated on the real line and absolutely continuous with respect to   Lebesgue measure on the real line. The description of the density and the support of this limiting measure will be the subject of the next two sections.
\subsection{Structure of the singular   measure }
In this section we shall prove the part of Theorem \ref{thm2} dealing with the weak convergence of the measure $dP_t$ when $t$ goes to $+\infty.$ First,
it is obvious that the  probability measure $dP_t$ is absolutely continuous with respect to the Lebesgue measure. The  convergence of the family $\{dP_t, t\geq0\}$ will be done in a weak sense as follows.
 Let $\varphi\in\mathcal{D}(\R^2)$ be a test function, one can write using Fubini's theorem
\begin{eqnarray*}
 I_t&\triangleq&\int_{\R^2}\varphi(x,y)dP_t\\
 &=&e^{t}\int_{\R}\int_{-f_t(x)}^{f_t(x)}\varphi(x,y)dy.
\end{eqnarray*}
According to  Taylor expansion in the second variable one gets
$$
\forall (x,y)\in\R^2,\quad \varphi(x,y)=\varphi(x,0)+ y\psi(x,y)\quad \textnormal{and}\quad \|\psi\|_{L^\infty}\leq C.
$$
This implies that
\begin{equation}\label{meas1}
I_t=2e^t\int_{\R} f_t(x)\varphi(x,0) dx+I_t^1, \quad I_t^1\triangleq e^{t}\int_{\R}\int_{-f_t(x)}^{f_t(x)}y\psi(x,y)dy.
\end{equation}
We shall  check taht the term $I_t^1$ does not contribute in the limiting behavior. Actually it vanishes for $t$ going to infinity. Indeed,
\begin{eqnarray*}
|I_t^1|&\leq&e^t\|\psi\|_{L^\infty}\int_{\R} [f_t(x)]^2 dx.
\end{eqnarray*}
Using \eqref{tik1} and the localization of the support of $f_t$ in the convex hull of the initial support, we deduce that
$$
|I_t^1|\leq C e^{-t}
$$
and thus
$$
\lim_{t\to+\infty}I_t^1=0.
$$
Combining \eqref{graph1}, \eqref{VY1}, \eqref{tik0}, \eqref{tik1} and \eqref{GW1} we deduce that
\begin{equation}\label{vezX1}
 \partial_t f(t,x)+u_1\partial_x f(t,x)+f(t,x)=-f(t,x)R(t,x)
\end{equation}
with 
\begin{eqnarray}\label{integra1}
 \nonumber \|R(t)\|_{L^\infty}&\leq& \|f^\prime(t)\|_{D}\big(1+\|f^\prime(t)\|_{\infty}^5\big)\\
&\le& C e^{-\frac{s}{4+2s}{t}}.
\end{eqnarray}
From the characteristic method developed in studying \eqref{kar17} we get the representation
\begin{equation}\label{verd1}
e^tf\big(t,\psi(t,x)\big)=f_0(x)e^{\int_0^t R(\tau,\psi(\tau, x)) d\tau}.
\end{equation}
From the integrability property \eqref{integra1} we deduce that $\{e^t f(t,\psi(t))\}$ converges uniformly as $t$ goes to $+\infty$  to the positive function
\begin{equation}\label{zaf1}
x\mapsto f_0(x)e^{\int_0^{+\infty}R(\tau,\psi(\tau, x)) d\tau}\triangleq R_2(x).
\end{equation}
More precisely, using straightforward computations we easily get
\begin{eqnarray}\label{verd0}
\nonumber \|e^t f_t\circ\psi(t)-R_2\|_{L^\infty}&\leq&\|R_2\|_{L^\infty}\int_{t}^{+\infty}\|R(\tau)\|_{L^\infty} d\tau\\
&\leq& Ce^{-\frac{s}{4+2s}{t}} .
\end{eqnarray}
The next goal is  prove that the flow  $\psi(t)$ converges uniformly as $t$ goes to infinity  to some homeomorphism $\psi_\infty:\R\to\R$ which belongs to the bi-Lipschitz class. First, recall from the definition \eqref{flot1} that
$$
\psi(t,x)=x+\int_0^t u_1\big(\tau,\psi(\tau,x)\big) d\tau.
$$
Recall from Section \ref{Sec-graph} that $u_1(x)=v_1(x,f(x))$ and the velocity is computed from the density $\rho$ according to the second equation of \eqref{sqg}. Hence we get
$$
\|u_1(t)\|_{L^\infty} \le \|\Delta^{-1}\nabla \rho\|_{L^\infty}.
$$
Now using the  classical interpolation inequality 
$$
\|\Delta^{-1}\nabla \rho\|_{L^\infty}\le C\|\rho\|_{L^1}^{\frac12}\|\rho\|_{L^\infty}^{\frac12}
$$
combined with the decay rate stated in Proposition \ref{prop20}-$(2)$ we deduce that
\begin{equation}\label{Ra1}
\|u_1(t)\|_{L^\infty} \leq C e^{-t/2}
\end{equation}
Consequently, it follows that $\psi(t)$ converges uniformly to the function
$$
\psi_\infty(x)\triangleq x+\int_0^{+\infty} u_1\big(\tau,\psi(\tau,x)\big) d\tau.
$$
More precisely, we have
\begin{eqnarray}\label{dima1}
\nonumber\|\psi(t)-\psi_\infty\|_{L^\infty}&\leq& \int_{t}^{+\infty}\|u_1(\tau)\|_{L^\infty}d\tau\\
&\le&Ce^{-t/2}.
\end{eqnarray}
It remains to check that $\psi_\infty$ is bi-Lipschitz. First we know that
$$
\|\partial_x\psi(t)\|_{L^\infty}\leq e^{V(t)}, \quad V(t)=\int_0^t \|\partial_x u_1(\tau)\|_{L^\infty} d\tau.
$$
Using \eqref{zarda88} and \eqref{GW1} we deduce that
\begin{equation}\label{mqs1}
\forall t\geq0, \, \|\partial_x\psi(\tau)\|_{L^\infty}\leq C, \quad\|\partial_x u_1(t)\|_{L^\infty}\leq C \varepsilon^{\frac{s}{2+s}} e^{-\frac{s}{2+s}t}.
\end{equation}
Differentiating $\psi_\infty$ and using the triangle inequality we get
$$
1-\int_0^{+\infty}\|\partial_x u_1(\tau)\|_{L^\infty}\|\partial_x\psi(\tau)\|_{L^\infty} d\tau\le\psi_\infty^\prime(x)\leq 1+\int_0^{+\infty}\|\partial_x u_1(\tau)\|_{L^\infty}\|\partial_x\psi(\tau)\|_{L^\infty} d\tau.
$$
Therefore we obtain
$$
\forall x\in\R,\quad 1-C\varepsilon^{\frac{s}{2+s}}\leq\psi_\infty^\prime(x)\leq 1+C\varepsilon^{\frac{s}{2+s}}.
$$
Taking $\varepsilon$ small enough,  meaning that the initial data is very small, we get 
\begin{equation}\label{Lip1}
\forall x\in\R,\quad \frac12\leq\psi_\infty^\prime(x)\leq \frac32\cdot
\end{equation}
This shows that $\psi_\infty$ is a bi-Lipschitz function from $\R$ to $\R$. Further,  it is obvious that 
$$
\psi_\infty(x)=\psi(t,x)+\int_t^{+\infty} u_1(\tau,\psi(\tau,x))d\tau
$$
and hence
$$
\psi_\infty\big((\psi^{-1}(t,x)\big)=x+\int_t^{+\infty} u_1\big(\tau,\psi\big(\tau,\psi^{-1}(t,x)\big)\big)d\tau.
$$
Combined this identity with $\psi_\infty\circ\psi_\infty^{-1}=\hbox{Id}$ and \eqref{Ra1} yields
\begin{eqnarray*}
\big|\psi_\infty\big((\psi^{-1}(t,x)\big)-\psi_\infty(\psi_\infty^{-1}x)\big|&\leq& \int_{t}^{+\infty}\|u_1(\tau)\|_{L^\infty}d\tau\\
&\le&Ce^{-t/2}.
\end{eqnarray*}
Applying  \eqref{Lip1} we deduce that
$$
\big\|\psi^{-1}(t)-\psi_\infty^{-1}\big\|_{L^\infty}\leq C e^{-t/2}.
$$

This shows that $\psi^{-1}(t)$ converges uniformly to $\psi_\infty^{-1}$ with an exponential rate. Set 
\begin{equation}\label{Tf1}
\Phi=R_2\circ\psi_\infty^{-1}
\end{equation}
and assume for a while that $R_2$  belongs to $C^\alpha$ for any $\alpha\in(0,1)$, then we deduce from the preceding estimates, especially  \eqref{verd0}  and \eqref{verd1}, that
\begin{eqnarray*}
\|e^t f(t)-\Phi\|_{L^\infty}&\leq&\|e^t f(t)-R_2\circ\psi^{-1}(t)\|_{L^\infty}+\|R_2\circ\psi^{-1}(t)-R_2\circ\psi_\infty^{-1}\|_{L^\infty}\\
&\le& Ce^{-\frac{s}{4+2s}t}+\| R_2\|_{\alpha}\|\psi^{-1}(t)-\psi_\infty^{-1}\|_{L^\infty}^\alpha\\
&\le&
 C e^{-\frac{s}{4+2s}t}+C e^{-\alpha t/2}.
\end{eqnarray*}
Taking $\alpha=\frac{2s}{4+2s}$ we get
\begin{eqnarray}\label{Domm6}
\|e^t f(t)-\Phi\|_{L^\infty}&\leq&
 C e^{-\frac{s}{4+2s}t}.
\end{eqnarray}
Let us now check that $R_2$ belongs to $C^\alpha$ for any $\alpha\in(0,1).$ For this goal we shall express differently the  function $R_2.$ Set $R_1(t,x)=-f(t,x)R(t,x)$ then from 
 the characteristic method the solution to \eqref{vezX1} may be recovered as follows
\begin{equation*}
e^tf\big(t,\psi(t,x)\big)=f_0(x)+\int_0^t e^{\tau}R_1\big(\tau,\psi(\tau, x)\big) d\tau.
\end{equation*}
Putting together  \eqref{integra1} and \eqref{tik1} we deduce that
\begin{equation}\label{Cozz1}
\|R_1\big(\tau,\psi(\tau)\big) \|_{L^\infty}\leq C e^{-\frac{4+3s}{4+2s} t}.
\end{equation}
Therefore we find the identity
\begin{equation}\label{Fqs1}
R_2(x)=f_0(x)+\int_0^{+\infty} e^{\tau}R_1\big(\tau,\psi(\tau, x)\big) d\tau.
\end{equation}
We shall now study the regularity of $R_2$ through the use of this representation.

Differentiating in $x$ the equation \eqref{vezX1} and comparing it to the equation \eqref{kar17} we get the identity
$$
\partial_x R_1(t,x)=\mathcal{R}(t,x)+\partial_x u_1(t,x)\partial_xf(t,x).
$$
Using  \eqref{tik0}, \eqref{tikgh0} and \eqref{mqs1} we find
$$
\forall t\geq0,\quad \|\partial_x R_1(t)\|_{L^\infty}\leq Ce^{-t}.
$$
Combining this latter estimate  with  Leibniz formula and  \eqref{mqs1} implies
\begin{equation}\label{Cozz2}
\forall t\geq0,\quad \|\partial_x \big(R_1(t, \psi(t,\cdot))\big)\|_{L^\infty}\leq Ce^{-t}.
\end{equation}
It suffices now to apply the following classical interpolation inequality: for any $\alpha\in(0,1)$ there exists $C>0$ such that 
$$
\|h\|_{\alpha}\leq C\|h\|_{L^\infty}^{1-\alpha}\|h^\prime\|_{L^\infty}^{\alpha}
$$
which implies that according to \eqref{Cozz1} and \eqref{Cozz2}
\begin{equation}\label{Cozz3}
\forall t\geq0,\quad \|R_1(t, \psi(t,\cdot))\|_{\alpha}\leq Ce^{-t} e^{-t(1-\alpha)\frac{s}{4+2s}}.
\end{equation}
Returning to the identity \eqref{Fqs1}, one obtains in view of  \eqref{Cozz3}
\begin{eqnarray*}
\| R_2\|_{\alpha}&\leq& \|f_0\|_\alpha+\int_0^{+\infty} e^\tau \|R_1\big(\tau,\psi(\tau, \cdot)\big)\|_{\alpha} d\tau\\
&\leq& C,
\end{eqnarray*}
for any $\alpha\in(0,1).$ As an immediate consequence of \eqref{Tf1}, \eqref{Lip1} and \eqref{comp1} we find that $\Phi$ belongs to $C^\alpha$ for any $\alpha\in (0,1).$ We guess the profile $\Phi$ to keep the same regularity as $f_0$, that is, in $C^{1+s}$ but this could require much more refined analysis.

Now coming back to  \eqref{meas1} we find in view of \eqref{Domm6} and Lebesgue theorem
$$
\lim_{t\to+\infty} I(t)=2\int_{\R} \Phi(x)\varphi(x,0)dx.
$$
This is equivalent to write  in   the weak sense
\begin{eqnarray}\label{Mezz1}
\nonumber\lim_{t\to+\infty}dP_t&=&2\Phi\, \delta_{\R\otimes\{0\}}\\
&\triangleq& dP_\infty.
\end{eqnarray}
Now we shall discuss some properties  of $\Phi.$ From  \eqref{zaf1} and \eqref{Tf1}
 we have
\begin{equation}\label{SupX}
 \textnormal{supp } \Phi=\psi_\infty(K_0),\quad K_0=\textnormal{supp }f_0.
\end{equation}
 According to \eqref{Lip1}, the measure of $\textnormal{supp } \Phi$ is strictly positive with 
\begin{equation}\label{DomS1}
 | \textnormal{supp } \Phi|\geq \frac12 |K_0|.
\end{equation}
  It remains to check that $dP_\infty$ is a probability measure on the real axis, which reduces to verify that
 $$
2\int_{\R}\Phi(x) dx=1.
$$
First note that using Proposition \ref{prop20}-$(2)$ one obtains for any $t\geq0$,
$$
1=2\int_{\R} e^t f(t,x) dx.
$$ 
To exchange limit and integral it suffices to apply  
Lebesgue theorem thanks to the conditions  \eqref{Domm6} and and the fact that $\textnormal{supp } f_t\in \textnormal{Conv} K_0$, recall that for simplicity we have assumed that $\|\rho_0\|_{L^1}=1,$ 
$$
\lim_{t\to+\infty}\int_{\R} e^t f(t,x) dx=\int_{\R}\Phi(x)dx,
$$ 
which provides the desired result. We point out that with the normalization $\|\rho_0\|_{L^1}=1$ one gets instead of \eqref{Mezz1}
$$
dP_\infty=\frac{\Phi}{\|f_0\|_{L^1}}\delta_{\R\otimes\{0\}}
$$
which gives the structure of the limiting measure stated in Theorem \ref{thm2} thanks to \eqref{zaf1} and \eqref{Tf1}.
 \subsection{Concentration of the support}
 In this section we shall complete the study of the limiting measure $dP_\infty$  and  identify its support denoted by $K_\infty$. What is left to conclude  the proof of Theorem \ref{thm2} is just to check that the support  $D_t$ of the solution $\rho_t$ converges in the Hausdorff sense to $K_\infty.$ 
 Recall that  $K_0$ is  the support of $f_0$ and is assumed to be  is a finite collection of increasing segments $[a_i;b_i], i=1,..n,$ such that $a_i<b_i<a_{i+1}$. According to \eqref{SupX} one has
$$
 \textnormal{supp }\Phi=\psi_\infty(K_0)\triangleq K_\infty.
$$ 
Since $\Psi_\infty$ is strictly increasing due to \eqref{Lip1} one deduces easily that 
 $$
  \textnormal{supp }\Phi=\cup_{i=1}^n[a_i^\infty, b_i^\infty],\quad a_i^\infty\triangleq\psi_\infty( a_i),\quad  b_i^\infty\triangleq\psi_\infty( b_i).
 $$
 Using more again \eqref{Lip1}  one may easily obtain that 
 $$
\forall \, i,\quad  |a_i^\infty-b_i^\infty|\geq \frac12 |a_i-b_i|.
 $$
 Now to establish the convergence in the Hausdorff sense of $D_t$ towards  $K_\infty$ it suffices to prove the result for each connected component, that is,
 $$
 \forall i=1,..,n,\quad d_H\big(\Gamma^i_t,[a_i^{\infty}, b_i^\infty]\big)\leq C e^{-t}
 $$
 with 
 $$
 \Gamma^i_t\triangleq\Big\{\big(x, f_t(x)\big), x\in [a_i^t, b_i^t]  \Big\}.
 $$
 By straightforward analysis using \eqref{tik1} one obtains
 $$
 d_H\big(\Gamma^i_t,[a_i^{\infty}, b_i^\infty]\big)\le C e^{-t}+C\max\big(|a_i^t-a_i^\infty|, |b_i^t-b_i^\infty|\big).
 $$
 From \eqref{dima1} one gets
 $$
 \max\big(|a_i^t-a_i^\infty|, |b_i^t-b_i^\infty|\big)\leq C e^{-t}
 $$
and therefore 
$$
\forall t\geq 0,\quad d_H(D_t, K_\infty)\leq C e^{-t}.
$$
The proof of Theorem \ref{thm1} is now achieved.
  
 \begin{ackname}
The first author gratefully acknowledges the hospitality of Dong Li and the Hong Kong University of Science and Technology where  portions of this work were achieved.\end{ackname}

   \end{document}